\documentclass{svproc}
\usepackage{amssymb}
\usepackage{amsmath}
\usepackage{mathtools}
\usepackage{amsfonts}
\usepackage{amsxtra}
\usepackage[all]{xy}
\usepackage{verbatim}
\usepackage{color}
\usepackage[normalem]{ulem}
\usepackage[usenames,dvipsnames]{xcolor}
\usepackage{graphicx}
\usepackage{multicol}
\usepackage{footmisc}
\usepackage{comment}
\usepackage{etoolbox}
\usepackage{tikz-cd}
\usepackage{scalerel}
\usetikzlibrary{matrix,arrows}
\usepackage{extarrows}
\usepackage{url}

\usepackage{bbm}
\usepackage[mathscr]{euscript}
\usepackage{makeidx}
\makeindex

\usepackage[pagebackref, colorlinks, citecolor=DarkOrchid, linkcolor=NavyBlue, urlcolor=NavyBlue]{hyperref}
\hypersetup{linktoc=all,} %set to all if you want both sections and subsections linked
\usepackage{bbm}
\usepackage{bbold}

\usepackage[nopostdot,  sort=use, toc]{glossaries}%%%Needs to be uploaded after hyperref
\usepackage{glossary-mcols}
\setglossarystyle{mcolindex}
 \makeglossaries
 \glsdisablehyper %%%Removes hyperlink when it appears

\usepackage[textwidth=25mm, textsize=tiny]{todonotes}
\usepackage[final]{microtype}

\usepackage{enumerate}

\usepackage{appendix}

\numberwithin{equation}{section} 
\numberwithin{figure}{section}
\usepackage[nameinlink,capitalise]{cleveref}
\crefname{section}{Sect.}{Sects.}

\usepackage{stmaryrd}%for mapsfrom

\newcommand\restr[2]{{% we make the whole thing an ordinary symbol
  \left.\kern-\nulldelimiterspace % automatically resize the bar with \right
  #1 % the function
  \vphantom{\big|} % pretend it's a little taller at normal size
  \right|_{#2} % this is the delimiter
  }}

  \numberwithin{equation}{subsection}

\newcommand{\C}{\mathscr{C}}

\newcommand{\V}{\mathscr{V}}
\newcommand{\cM}{\mathscr{M}}

\newcommand{\I}{\mathbb I}
\newcommand{\Set}{\mathrm{Set}}

\newcommand{\N}{\mathbb N}
\newcommand{\eN}{{\mathbb N_\infty}}

\newcommand{\alg}[2]{\mathrm{Alg}_{#2}{#1}}
\newcommand{\id}{\mathrm{id}}

\newcommand{\inthomp}[3]{\underline{\hom}_{#1}(#2,#3)}

\newcommand{\term}{1}

\newcommand{\Vect}{\mathrm{Vect}}
\newcommand{\Alg}{\mathrm{Alg}}
\newcommand{\coAlg}{\mathrm{coAlg}}
\newcommand{\CoAlg}{\mathrm{coAlg}}
\newcommand{\coCAlg}{\mathrm{coCAlg}}
\newcommand{\op}{\mathrm{op}}
\newcommand{\inthom}{\underline{\hom}}
\newcommand{\tv}{T^\vee}

\newcommand{\Limi}{C_\infty}

\newcommand{\ev}{\mathrm{ev}}

\newcommand{\uni}[2]{\underline{\Alg}_k(#1, #2)}

\newcommand{\pull}{\arrow[dr, phantom, "\lrcorner", very near start]}

\newcommand{\dom}{\textsc{dom}}

\newcommand{\n}{\mathbb{n}}
\renewcommand{\k}{\mathbb{k}}

\newcommand{\inv}{^{-1}}
\newcommand{\im}{\operatorname{im}}

\newcommand{\ind}[1]{\llbracket{#1}\rrbracket}
\newcommand{\zip}{\mathtt{zip}}

\newcommand{\head}{\mathtt{head}}
\newcommand{\take}{\mathtt{take}}

\newcommand{\filter}{\mathtt{filter}}
\newcommand{\map}{\mathtt{map}}
\newcommand{\fold}{\mathtt{fold}}

\newcommand{\depth}{\mathtt{depth}}
\newcommand{\shape}{\mathtt{shape}}
\newcommand{\argmax}{\mathtt{argmax}}

\newcommand{\xto}[1]{\xrightarrow{#1}}
\renewcommand{\phi}{\varphi}
\renewcommand{\epsilon}{\varepsilon}
\newcommand{\const}{\operatorname{const}}
\newcommand{\pr}{\operatorname{pr}}
\newcommand{\inr}{\operatorname{inr}}

\newcommand{\len}{\mathtt{len}}
\newcommand{\Fr}{\operatorname{Fr}}
\newcommand{\Cof}{\operatorname{Cof}}
\newcommand{\Sh}{\operatorname{Sh}}
\newcommand{\CoSh}{\operatorname{CoSh}}
\newcommand{\m}{\mu}

\DeclareMathSymbol{\mathinvertedexclamationmark}{\mathord}{operators}{'074}
\DeclareMathSymbol{\mathexclamationmark}{\mathord}{operators}{'041}
\makeatletter
\newcommand{\raisedmathinvertedexclamationmark}{%
  \mathord{\mathpalette\raised@mathinvertedexclamationmark\relax}%
}
\newcommand{\raised@mathinvertedexclamationmark}[2]{%
  \raisebox{\depth}{$\m@th#1\mathinvertedexclamationmark$}%
}
\makeatother

\newcommand{\fromI}[1]{\raisedmathinvertedexclamationmark_{#1}}
\newcommand{\toT}[1]{\mathexclamationmark_{#1}}

\renewcommand{\partial}{\textsf{\reflectbox{\textup 6}}}

\newcommand{\cofree}{\mathrm{Cof}}
\newcommand{\free}{\mathrm{Fr}}
\newcommand{\tone}{\mathtt{t}}
\newcommand{\bcoalg}{{\mathrm{coAlg}}}
\newcommand{\ubcoalg}{\underline{\mathrm{coAlg}}}
\newcommand{\balg}{{\mathrm{Alg}}} 
\newcommand{\ubalg}{{\underline{\mathrm{Alg}}}}

\newcommand{\standalg}[1]{{\mathbb {#1}}}
\newcommand{\standcoalg}[1]{{\mathbb {#1}^\circ}}
\newcommand{\coaN}{\mathbb N^c}

\setglossarypreamble{} % optional if you want no intro text

\newglossaryentry{freeAlg}{
name=\ensuremath{T(V)},
description={the free $k$-algebra on $V$}
}

\newglossaryentry{freeFAlg}{
name=\ensuremath{\free},
description={free algebra over endofunctor}
}

\newglossaryentry{field}{
name=\ensuremath{k},
description={a field}
}

\newglossaryentry{cofreecoAlg}{
name=\ensuremath{\tv(V)},
description={the cofree $k$-coalgebra on $V$}
}

\newglossaryentry{partialhm}{
name=\ensuremath{f_c},
description={partial homomorphism on $c\in C$}
}

\newglossaryentry{cofreeFcoAlg}{
name=\ensuremath{\cofree},
description={cofree algebra over endofunctor}
}

\newglossaryentry{ConvProd}{
name=\ensuremath{f \ast g},
description={convolution product}
}

\newglossaryentry{ConvAlg}{
name=\ensuremath{[C,B]},
description={convolution algebra}
}

\newglossaryentry{MeasTens}{
name=\ensuremath{C \triangleright A},
description={measuring tensor}
}

\newglossaryentry{CoAlgHom}{
name=\ensuremath{\ubcoalg(C, D)},
description={coalgebra of $F$-coalgebra homomorphisms}
}

\newglossaryentry{ConcList}{
name=\ensuremath{m : \ell},
description={concatenation of an element $m$ to a list $\ell$}
}

\newglossaryentry{List}{
name=\ensuremath{M^*},
description={set of list of elements in $M$ of finite length}
}

\newglossaryentry{Listn}{
name=\ensuremath{M^{\leq n}},
description={set of list of elements in $M$ of at most length $n$}
}

\newglossaryentry{Treen}{
name=\ensuremath{T_{M,n}},
description={binary trees with values in $M$ of depth at most $n$}
}

\newglossaryentry{STreen}{
name=\ensuremath{S_{M,n}},
description={trees with values in $M$ of depth at most $n$}
}

\newglossaryentry{Tree}{
name=\ensuremath{T_{M}},
description={binary trees with values in $M$ of finite length}
}
\newglossaryentry{TreeInf}{
name=\ensuremath{{T}_{M,\infty}},
description={binary trees with values in $M$ of infinite depth}
}

\newglossaryentry{map}{
name=\ensuremath{\map},
description={evaluating a function to a list}
}

\newglossaryentry{zipapply}{
name=\ensuremath{\mathtt{zipapply}},
description={evaluation a collection of functions to a list}
}

\newglossaryentry{len}{
name=\ensuremath{\len},
description={length of a list}
}

\newglossaryentry{zip}{
name=\ensuremath{\zip},
description={concatenating lists together}
}

\newglossaryentry{STreeInf}{
name=\ensuremath{{S}_{M,\infty}},
description={binary trees with values in $M$ of infinite depth}
}

\newglossaryentry{STree}{
name=\ensuremath{S_{M}},
description={trees with values in $M$ of finite length}
}

\newglossaryentry{EmptList}{
name=\ensuremath{[\:]},
description={empty list}
}

\newglossaryentry{AlgHomk}{
name=\ensuremath{\Alg_k(A,B)},
description={set of $k$-algebra homomorphisms}
}

\newglossaryentry{subterm1}{
name=\ensuremath{\standcoalg{n}},
description={subterminal coalgebra of $\eN$ of the first $(n+1)$-terms}
}

\newglossaryentry{subterm2}{
name=\ensuremath{\coaN},
description={subterminal coalgebra of $\eN$ spanned by $\N$}
}

\newglossaryentry{coAlgHomk}{
name=\ensuremath{\coAlg_k(A,B)},
description={set of $k$-coalgebra homomorphisms}
}

\newglossaryentry{IntcoAlgHomk}{
name=\ensuremath{\underline{\coAlg}_k(C, D)},
description={$k$-coalgebra of $k$-coalgebra homomorphisms}
}

\newglossaryentry{TermOfType}{
name=\ensuremath{::},
description={term of a type}
}

\newglossaryentry{AlgHom}{
name=\ensuremath{\Alg(A,B)},
description={set of algebra homomorphisms over an endofunctor}
}

\newglossaryentry{PolyEndo}{
name=\ensuremath{\mathcal{P}_f},
description={polynomial endofunctor of a function $f$}
}

\newglossaryentry{TermiCoalgNat}{
name=\ensuremath{\eN},
description={terminal coalgebra of $\id+1$}
}

\newglossaryentry{MonUnit}{
name=\ensuremath{\I},
description={monoidal unit of a monoidal category}
}

\newglossaryentry{TermiCoalgNatMap}{
name=\ensuremath{\ind{-}},
description={unique coalgebra homomorphism to $\eN$}
}

\newglossaryentry{coAlgHom}{
name=\ensuremath{\Alg(C,D)},
description={set of coalgebra homomorphisms over an endofunctor}
}

\newglossaryentry{Vectk}{
name=\ensuremath{\Vect_k},
description={category of $k$-vector spaces}
}

\newglossaryentry{SweedlerNot}{
name=\ensuremath{\sum_{(c)}c_{(1)}\otimes c_{(2)}},
description={Sweedler notation}
}

\newglossaryentry{MeasuringSet}
{
    name = \ensuremath{\mu_C(A,B)},
    description= {set of measurings}
}

\newglossaryentry{UnivMeasuring}{
name=\ensuremath{\uni{A}{B}},
description={universal measuring coalgebra over field $k$}
}

\newglossaryentry{UnivFMeasuring}{
name=\ensuremath{\ubalg(A,B)},
description={universal measuring coalgebra over endofunctor}
}

\begin{document}

\mainmatter              % start of the contributions
\title{Measuring data types}
\titlerunning{Measuring data types}  % abbreviated title (for running head)
\author{Lukas Mulder\inst{1} \and Paige Randall North\inst{2}
\and Maximilien P\'eroux\inst{3}}
\authorrunning{Mulder, North, P\'eroux} % abbreviated author list (for running head)
\tocauthor{Lukas Mulder, Paige North, Maximilien P\'eroux}
\institute{
Institute for Computing and Information Science, Radboud University\\
Toernooiveld 212, Nijmegen, 6525 EC, Netherlands\\
\email{lukas.mulder@ru.nl}\\
\and
Department of Mathematics and Department of Information and Computing Sciences, Utrecht University\\
Budapestlaan 6, Utrecht, 3584 CD, Netherlands\\
\email{p.r.north@uu.nl}\\
\and
Department of Mathematics, Michigan State University\\
619 Red Cedar Road, East Lansing, MI 48824, USA\\
\email{peroux@msu.edu}\\
}

\maketitle              % typeset the title of the contribution

\begin{abstract}
In this article, we combine Sweedler's classic theory of measuring coalgebras -- by which $k$-algebras are enriched in $k$-coalgebras for $k$ a field -- with the theory of W-types -- by which the categorical semantics of inductive data types in functional programming languages are understood.
In our main theorem, we find that under some hypotheses, algebras of an endofunctor are enriched in coalgebras of the same endofunctor, and we find polynomial endofunctors provide many interesting examples of this phenomenon.
We then generalize the notion of initial algebra of an endofunctor using this enrichment, thus generalizing the notion of W-type.
This article is an extended version of \cite{north2023coinductive}: it adds expository introductions to the original theories of measuring coalgebras and W-types along with some improvements to the main theory and many explicitly worked examples.
\end{abstract}

\section{Introduction}

Many features of functional programming languages like Haskell and Agda are justified by their categorical semantics.
For instance, perhaps the central feature -- inductive data types -- is justified by the categorical theory of categorical W-types -- initial algebras of polynomial endofunctors \cite{dybjer,mp}.
When programming in a functional programming language, most (data) types -- such as the type of natural numbers, of booleans, of lists, of strings, of trees, etc. -- are defined as inductive types.
Passing to the categorical semantics of a functional programming language, this way of defining types corresponds exactly to specifying an object of a category by asking that it is the initial algebra of a specified endofunctor.
Depending on the category, not all endofunctors have initial algebras of course.
Thus of central importance are \emph{polynomial} endofunctors, which have initial algebras (called categorical W-types) in elementary toposes with natural numbers objects, in particular the category of sets.
These correspond to the inductive types that are guaranteed in languages such as Coq and Agda, and though in languages such as Haskell more inductive types are guaranteed, categorical W-types correspond to the inductive types that are usually used.

Such categorical semantics are important in general because they make programming languages susceptible to mathematical study.
And more specifically, the categorical semantics of inductive data types is important because it connects types -- e.g., the type of natural numbers -- living in an abstract programming language with `real' mathematical objects -- e.g., the `real' natural numbers studied by mathematicians and living in the category of sets, thus justifying (or giving meaning to, hence the word \emph{semantics}) programs written in such a language.

We seek in this line of work to further elucidate the mathematical structure in which categorical W-types are embedded -- driven not only by motivations from pure mathematics, but also the fact that categorical phenomena that are at first adjacent to the semantics of programming languages often find their way into the programming languages themselves.

The mathematical structure surrounding categorical W-types that we study here is an analogue for our setting of Sweedler's classic theory of measuring coalgebras \cite{Sweedler}. In that theory, $k$-algebras are found to be enriched in $k$-coalgebras, where $k$ is a field in the usual sense. This enrichment adds extra structure to the category of $k$-algebras. One can see this extra structure as giving a notion of \emph{partial algebra homomorphism}. That is, while the hom-sets in the category of $k$-algebras consist of algebra homomorphisms, the hom-coalgebras in the \emph{enriched} category of $k$-algebras consist not only of algebra homomorphisms but also include other morphisms which are `almost' algebra homomorphisms. The hom-coalgebras also include information about `how close' a morphism is to being a homomorphism, hence the word \emph{measuring}. That is, if one thinks of the hom-objects of an (enriched) category as being the tools with which one studies the objects, this classic Sweedler theory gives us higher-precision tools with which to study the interactions between algebras.

Inspired by this, we give an analogue in our setting.

\begin{theorem}[{\cref{thm:enriched}}]
Let $\C$ be a locally presentable, closed symmetric monoidal category.
Let $F\colon \C\rightarrow \C$ be an accessible, lax symmetric monoidal functor. 
Then the category of $F$-algebras in $\C$ is enriched, tensored and cotensored over the closed symmetric monoidal category of $F$-coalgebras in $\C$.
\end{theorem}
%TODO: Make theorem*
We use this to give a notion of partial algebra homomorphism. 
The theorem says that, given $F$-algebras $A$ and $B$, in lieu of a set $\Alg(A,B)$ of algebra homomorphism $A\rightarrow B$, there is a coalgebra $\ubalg(A,B)$ of partial algebra homomorphisms. 
Inspired by the generalization of (co)limits in ordinary categories to weighted (co)limits in enriched categories, we generalize the notion of initial algebra to $C$-initial algebra for any coalgebra $C$.
Thus, by considering polynomial endofunctors, we generalize the notion of categorical W-types (i.e., initial algebra of a polynomial endofunctor), and we give many such examples.
We emphasize that not all polynomial endofunctors can be considered, as the endofunctor must be lax monoidal. For instance, this prevents the theory to be applied to general list type, but only lists over a monoid, see \cref{subsec:list} below. Moreover, the requirement of ``symmetry" for $\C$ and $F$ above is not strictly necessary, and is only there for convenience.

We also make sense of partial homomorphisms up to index $n$ and construct a tower of coalgebras
 \[
\begin{tikzcd}[column sep=small]
    \ubalg_\infty(A,B)\ar{r} & \cdots \ar{r} &
    \ubalg_2(A,B) \ar{r}&
    \ubalg_1(A,B) \ar{r}&
    \ubalg_0(A,B).
\end{tikzcd}    
\]
Its limit is composed of the $\infty$-partial homomorphisms which are morphisms $A\rightarrow B$, built layer by layer in the towers, that approximates inductively a total algebra homomorphism.

This is not an isolated phenomenon: several other analogues in other settings have appeared since Sweedler's version: \cite{anel2013sweedler,hylandetal,vasila,Per22,MRU22}.
One of our motivations from pure mathematics is purely to understand this phenomenon better.
Indeed, our setting is the `simplest' in the sense that our notion of (co)algebra does not require any coherence: that is, an algebra of an endofunctor is purely an object $X$ with a morphism $FX \to X$; there are no units or multiplication to keep track of.
Thus, here, Sweedler's phenomenon is visible in its purest form.

In future work, we hope to push both the categorical, pure mathematical story further as well as develop applications in programming languages.
In this article, we fix a category $\C$ and an endofunctor $F$ to obtain an enriched category $\Alg_F (\C)$.
In future work, we hope to show that this construction is categorical, i.e., functorial in $\C$ and $F$.
On the programming side, we hope to see how such W-types could be usefully implemented in a programming language like Haskell.

\subsection{Outline and Contents}

In \cref{sec: sweedler coalgebras}, we review Sweedler's classic story of $k$-algebras enriched in $k$-coalgebras, for $k$ a field.
In \cref{sec: data types}, we review the basic theory of W-types, that is, of data types defined as initial algebras for an endofunctor.
In these first two sections, we aim to make this paper accessible to master's students or advanced undergraduates with knowledge of category theory; in addition, we hope to bridge the gap between algebraists in the tradition of Sweedler and (functional) programming language theorists. Thus, while a familiarity with abstract algebra is useful for \cref{sec: sweedler coalgebras} and a familiarity with functional programming is useful for \cref{sec: data types}, we hope that the basic theory is explained here clearly enough to get the ideas across to members of either camp.
Thus, in particular, \cref{sec: sweedler coalgebras} and \cref{sec: data types} are expository.

In \cref{sec: measuring}, we put the content of Sections \ref{sec: sweedler coalgebras} and \ref{sec: data types} to find that an analogue of Sweedler's story holds for endofunctors: under some hypotheses, algebras for endofunctors are enriched in coalgebras for the same endofunctor (\cref{thm:enriched}).
\Cref{sec: measuring} is an extended version of \cite{north2023coinductive}. There are several differences between that section and \cite{north2023coinductive}. First, all proofs have been included (these were already available in the preprint \cite{arxiv}). Second, while \cite{north2023coinductive} presented an extended example before the categorical theory, here we have first given the categorical theory first with that example given as a running example. Third, we have included some small improvements. In particular, \cref{thm: most general} now generalizes \cite[Thm.~33]{north2023coinductive}.
We also clarify and generalize an inductive construction of partial homomorphisms in \cref{ex: partial induction}, \cref{ex: partial induction part deux} and \cref{remark: generalized tower} that was originally found in \cite[Construction 1]{north2023coinductive}.

In \cref{sec: examples}, we give many worked-out examples of the theory presented in \cref{sec: measuring}. Because we want this paper to be accessible to a wide range of readers, we provide many details for the calculations in these examples.

\section{Sweedler Theory}
\label{sec: sweedler coalgebras}

In this section, we present the classical theory of algebras and coalgebras found in \cite{Sweedler}, with a modern perspective that will allow us to extend the arguments in the next sections. Sweedler arguments was formalized in a more general context in \cite{Fox}, and his approach was extended further in \cite{hylandetal}.
 
\subsection{Algebras}

Let \gls{field} be a field.
Let $\gls{Vectk}$ be the category of $k$-vector spaces with $k$-linear transformations.
Recall that a \emph{$k$-algebra}\index{algebra!over a field} $A$ is a vector space $A$ together with a multiplication
\begin{align*}
    A\times A & \longrightarrow A\\
    (x, y) & \longmapsto xy
\end{align*}
that is $k$-bilinear, meaning that it defines a $k$-linear homomorphism $\mu\colon A\otimes_k A\rightarrow A$, associative, meaning that for all $x,y,z\in A$ we have
\[
(xy)z=x(yz),
\]
and unital, meaning there exists $1_A\in A$ such that
\[
x1_A=x=1_Ax
\]
for all $x\in A$. The existence of such element $1_A\in A$ is recorded by a $k$-linear homomorphism $\eta\colon k\rightarrow A$. Associativity and unitality can now be rephrased to say that the following diagrams commute in $k$-vector spaces
\[
\begin{tikzcd}
    A\otimes_k A \otimes_k A \ar{r}{1\otimes \mu}\ar{d}[swap]{\mu\otimes 1} & A\otimes_k A \ar{d}{\mu} & A\otimes_k k\ar[phantom, "\cong"]{r} \ar{d}[swap]{1\otimes \eta} & [-2em] A \ar[equals,shorten <=1.5ex, shorten >=1.5ex, shift right=.5]{drr} \ar[phantom, "\cong"]{r}& [-2em] k\otimes_k A \ar{r}{\eta\otimes 1} & A\otimes_k A \ar{d}{\mu}\\
    A\otimes_k A \ar{r}{\mu} & A & A\otimes_k A \ar{rrr}[swap]{\mu} & & &  A.
\end{tikzcd}
\]
Given two $k$-algebras $(A, \mu_A, \eta_A)$ and $(B, \mu_B, \eta_B)$, an\emph{ algebra homomorphism} $f\colon A\rightarrow B$ is a $k$-linear homomorphism $f \colon A\rightarrow B$ such that $f(xy)=f(x)f(y)$ and $f(1_A)=1_B$: i.e., the following diagrams commute in $k$-vector spaces
\[
\begin{tikzcd}
    A\otimes_k A \ar{r}{f\otimes f} \ar{d}[swap]{\mu_A} & B\otimes_k B \ar{d}{\mu_B} & k \ar{d}[swap]{\eta_A}\ar{dr}{\eta_B}\\
    A\ar{r}{f} & B & A\ar{r}[swap]{f} & B.
\end{tikzcd}
\]
Denote by $\Alg_k$ the resulting category of $k$-algebras. We denote by $\gls{AlgHomk}$ the set of algebra homomorphisms $A\to B$.
The initial $k$-algebra is $k$ by unitality.

The \emph{free $k$-algebra}\index{free algebra!over a field} on a vector space $V$ is given by the \emph{tensor algebra}
%\nomenclature{$T(V)$}{free $k$-algebra on $V$}
\[
\gls{freeAlg}=\bigoplus_{n\geq 0} V^{\otimes n}
\]
with multiplication given by concatenation: the identifications $V^{\otimes i}\otimes V^{\otimes j}\stackrel{\cong}\rightarrow V^{\otimes i+j}$ assemble to give the multiplication $TV \otimes TV \rightarrow TV $.
The free algebra is the left adjoint with respect to the forgetful functor $U\colon \Alg_k\rightarrow \Vect_k$: i.e., we obtain the following adjunction
\[
\begin{tikzcd}[column sep=large]
    \Vect_k \ar[bend left]{r}{T} \ar[phantom, "\perp" description, xshift=-0.5ex]{r} & \Alg_k. \ar[bend left]{l}{U}
\end{tikzcd}
\]
In particular, for a vector space $V$ and a $k$-algebra $A$, we obtain the following natural bijection
\[
\Alg_k(TV, A) \cong \Vect_k(V, UA).
\]

Here are some non-examples of algebra homomorphisms that we should keep in mind. We will see that Sweedler's theory \cite{Sweedler} allows to consider these as examples of \emph{partial} algebra homomorphisms.

\begin{example}
Consider the polynomial ring $\mathbb{R}[x]$ which is an $\mathbb{R}$-algebra with the usual multiplication of polynomials.
 Consider the derivation operator 
 \begin{align*}
     D\colon \mathbb{R}[x] & \longrightarrow \mathbb{R}[x]\\
     f(x) & \longmapsto f'(x).
 \end{align*}
 This is a $\mathbb{R}$-linear homomorphism but not an $\mathbb{R}$-algebra homomorphism as $D(fg)\neq D(f)D(g)$ and $D(1)\neq 1$.
\end{example}

\begin{example}
   Consider the multiplication by 2 on the rationals: $m_2\colon\mathbb{Q}\rightarrow\mathbb{Q}$ sending $a \mapsto 2a$ for all $a\in \mathbb{Q}$. It is a $\mathbb{Q}$-linear homomorphism but not an algebra homomorphism when $\mathbb{Q}$ is viewed as a $\mathbb{Q}$-algebra with its usual multiplication, as $m_2(ab)=2ab \neq 4ab = m_2(a)m_2(b)$.
\end{example}

\subsection{Coalgebras}

We can dualize the notion of $k$-algebras as follows. A \emph{$k$-coalgebra}\index{coalgebra!over a field} $C$ is a vector space $C$ together with $k$-linear homomorphisms $\Delta\colon C\rightarrow C\otimes_k C$ (called \emph{comultiplication}) and $\varepsilon\colon C\rightarrow k$ (called the \emph{co-unit}) that is co-associative and co-unital in the sense that the following diagrams commute
\[
\begin{tikzcd}
    C\otimes_k C \otimes_k C \ar[leftarrow]{r}{1\otimes \Delta}\ar[leftarrow]{d}[swap]{\Delta\otimes 1} & C\otimes_k C \ar[leftarrow]{d}{\Delta} & C\otimes_k k\ar[phantom, "\cong"]{r} \ar[leftarrow]{d}[swap]{1\otimes \varepsilon} & [-2em] C \ar[equals]{drr} \ar[phantom, "\cong"]{r}& [-2em] k\otimes_k C \ar[leftarrow]{r}{\varepsilon\otimes 1} & C\otimes_k C \ar[leftarrow]{d}{\Delta}\\
    C\otimes_k C \ar[leftarrow]{r}{\Delta} & C & C\otimes_k C \ar[leftarrow]{rrr}[swap]{\Delta} & & &  C.
\end{tikzcd}
\]
Compare the definition of $k$-coalgebra with that of $k$-algebra, and notice that we have simply reversed the direction of the arrows. We can express the commutativity of these diagrams nicely using \emph{Sweedler notation} from \cite[\S 1.2]{Sweedler}.
For $c\in C$, we abbreviate the sum notation in the comultiplication as follows 
%\nomenclature{$\sum_{(c)}c_{(1)}\otimes c_{(2)}$}{Sweedler notation}
\[
\Delta(c)=\sum_i {c_{(1)}}_i \otimes {c_{(2)}}_i \eqqcolon\gls{SweedlerNot}.
\]
The co-associativity then says that
\[
\sum_{(c)} \left(\sum_{(c_{(1)})}{(c_{(1)})}_{(1)}\otimes {(c_{(1)})}_{(2)}\right)\otimes {c_{(2)}}= \sum_{(c)} c_{(1)}\otimes \left( \sum_{(c_{(2)})} {(c_{(2)})}_{(1)}\otimes {(c_{(2)})}_{(2)}\right),
\]
and that element is often denoted $\sum_{(c)} c_{(1)}\otimes c_{(2)}\otimes c_{(3)}$.
Co-unitality says that
\[
\sum_{(c)} \varepsilon(c_{(1)})c_{(2)}=c=\sum_{(c)}\varepsilon(c_{(2)})c_{(1)}.
\]
Therefore, it is often useful to think of the comultiplication on $C$ to be a (co-associative) decomposition of each element $c\in C$ into parts ${c_{(1)}}_i$ and ${c_{(2)}}_i$ and to think of the co-unit as a projection such that when applied to one `side' of the result of decomposition, returns $c$.

\begin{example}\label{ex:binomial}
The \emph{binomial} $k$-coalgebra structure on $k[x]$ is given by the comultiplication
\begin{align*}
   \Delta\colon  k[x] & \longrightarrow k[x]\otimes_k k[x]\\
   x^n & \longmapsto \sum_i \binom{n}{i} x^{n-i}\otimes x^{i} 
\end{align*}
and co-unit
\begin{align*}
    \varepsilon\colon k[x] & \longrightarrow k\\
    x^n & \longmapsto 0.
\end{align*}
\end{example}

Given two $k$-coalgebras $(C, \Delta_C, \varepsilon_C)$ and $(D, \Delta_D, \varepsilon_D)$, a \emph{coalgebra homomorphism} $f\colon C\rightarrow D$ is a $k$-linear homomorphism $f\colon C\rightarrow D$ such that \[\sum_{(c)}f(c_{(1)})\otimes f(c_{(2)})= \sum_{(f(c))} f(c)_{(1)}\otimes f(c)_{(2)}\]
and
\[
\varepsilon_D(f(c))=\varepsilon_C(c)
\]
for all $c\in C$.
In other words, the following diagrams commute:
\[
\begin{tikzcd}
    C\otimes_k C \ar{r}{f\otimes f} \ar[leftarrow]{d}[swap]{\Delta_C} & D\otimes_k D \ar[leftarrow]{d}{\Delta_D} & k \ar[leftarrow]{d}[swap]{\varepsilon_C}\ar[leftarrow]{dr}{\varepsilon_D}\\
    C\ar{r}{f} & D & C\ar{r}[swap]{f} & D.
\end{tikzcd}
\]
Denote by $\coAlg_k$ the resulting category of $k$-coalgebras. Denote the set of coalgebra homomorphisms $C\to D$ by $\gls{coAlgHomk}$.
The terminal $k$-coalgebra is $k$ by co-unitality.

Although more subtle than its algebra analogue, there is a notion of \emph{cofree coalgebra} on a vector space $V$.
A candidate for this cofree coalgebra could be the direct product  $\prod_{n\geq 0} V^{\otimes n}$ coming from ``de-concatenation" $V^{\otimes (i + j)}\stackrel{\cong}\rightarrow V^{\otimes i}\otimes_k V^{\otimes j}$.
However, because direct products do not commute with tensor products, the de-concatenation only leads to an injective $k$-linear homomorphism
\[
\Delta\colon \prod_{n\geq 0}V^{\otimes n} \hookrightarrow \prod_{i,j\geq  0} V^{\otimes i}\otimes V^{\otimes j}.
\]
The cofree $k$-coalgebra\index{cofree coalgebra!over a field} \gls{cofreecoAlg} is defined to be the largest vector subspace of $\prod_{n\geq 0}V^{\otimes n}$ such that \[\Delta(\tv(V))\subseteq \tv(V)\otimes_k \tv(V).\]
In fact, by \cite{blockleroux,hazewinkel}, this can be defined as the following pullback (i.e.~intersection) in vector spaces
\[
\begin{tikzcd}
    \tv(V) \pull \ar[hook]{d} \ar[hook]{r} & \left( \prod_{i \geq  0} V^{\otimes i}\right)\otimes_k \left(\prod_{j\geq 0}  V^{\otimes j}\right) \ar[hook]{d}\\
    \prod_{n\geq 0}V^{\otimes n} \ar[hook]{r}{\Delta} & \prod_{i,j\geq 0} V^{\otimes i}\otimes_k V^{\otimes j}.
\end{tikzcd}
\]
The cofree coalgebra is right adjoint to the forgetful functor $U\colon \coAlg_k\rightarrow \Vect_k$: i.e.~we obtain an adjunction
\[
\begin{tikzcd}[column sep=large]
    \coAlg_k \ar[bend left]{r}{U} \ar[phantom, "\perp" description, xshift=-0.5ex]{r} & \Vect_k. \ar[bend left]{l}{\tv}
\end{tikzcd}
\]
In particular, for a vector space  $V$ and a $k$-coalgebra $C$, we obtain a natural bijection
\begin{equation}\label{eq: forgetful-cofree}
    \coAlg_k(C, \tv(V))\cong \Vect_k(U(C), V).
\end{equation}

\subsection{The Convolution Algebra}
Denote by $[V,W]=\hom_k(V,W)$ the $k$-vector space of linear homomorphisms $V\rightarrow W$.
Suppose that $(C,\Delta, \varepsilon)$ is a $k$-coalgebra and that $(B, \mu, \eta)$ is a $k$-algebra. 
Then \gls{ConvAlg} is also a $k$-algebra with the following convolution product 
\begin{align*}
[C,B]\otimes_k [C,B] & \longrightarrow  [C, B]\\
f \otimes g & \longmapsto \gls{ConvProd}.
\end{align*}
Here, $f\ast g\colon C\rightarrow B$ is the composite
\[
\begin{tikzcd}
    C \ar{r}{\Delta} & C\otimes_k C \ar{r}{f\otimes g} & B\otimes_k B \ar{r}{\mu} & B.
\end{tikzcd}
\]
In Sweedler notation:
\[(f\ast g) (c)=\sum_{(c)} f(c_{(1)})g(c_{(2)}).\]
The unit $k\rightarrow [C,B]$ corresponds to $1_{[C,B]}\colon C\rightarrow B$ defined as the composite
\[
\begin{tikzcd}
    C \ar{r}{\varepsilon} & k \ar{r}{\eta} & B,
\end{tikzcd}
\]
i.e.~$1_{[C,B]}(c)=\varepsilon(c)1_B$.

\begin{definition}[convolution algebra]
Given a $k$-coalgebra $C$ and a $k$-algebra $B$, we call $([C, B], \ast, 1_{[C,B]})$ the \emph{convolution algebra}\index{convolution algebra!over a field} of $C$ with $B$.
The construction is natural in $C$ and $B$ and defines a functor
\[
[-,-]\colon \coAlg_k^\op\times \Alg_k \longrightarrow \Alg_k.
\]
\end{definition}

\begin{example}
If we let $C=k$, the terminal $k$-coalgebra,  then the convolution algebra $[k, B]$ is isomorphic to $B$ as a $k$-algebra.
\end{example}

\subsection{Partial Algebra Homomorphism}
Recall that given $k$-vector spaces $V,W, Z$, we have the $k$-linear exponential correspondence
\[
\Vect_k\left( V\otimes_k W, Z\right) \cong \Vect_k(V, [W, Z])\cong \Vect_k(W, [V,Z])
\]
for which a map $f\colon V\otimes_k W \rightarrow Z$ 
%\nomenclature{$\widetilde{f}$}{the map $\widetilde{f}\colon V\rightarrow [W,Z]$ adjunct of $f\colon V\otimes W \rightarrow Z$} 
corresponds to $\widetilde{f}\colon V\rightarrow [W,Z]$ (where $\widetilde{f}(v)(w)=f(v\otimes w)$) which corresponds to $\widehat{f}\colon W \rightarrow [V,Z]$ (where $\widehat{f}(w)(v)=f(v\otimes w)$).
%\nomenclature{$\widehat{f}$}{the map $\widehat{f}\colon W \rightarrow [V,Z]$ adjunct of $f\colon V\otimes W \rightarrow Z$}
It is customary to denote $\widetilde{f}(v)\colon W\rightarrow Z$ as $f_v$.

Let us now consider $k$-algebras $(A, \mu_A, 1_A)$ and $(B, \mu_B, 1_B)$ and a $k$-coalgebra $(C, \Delta, \varepsilon)$. Then by our discussion above, we know that a $k$-linear homomorphism $f\colon C\otimes_k A\rightarrow B$ corresponds to $k$-linear homomorphisms $\widetilde{f}\colon C\rightarrow [A,B]$ and $\widehat{f}\colon A \rightarrow [C,B]$.
Since $[C,B]$ can be given the structure of a $k$-algebra via the convolution product, then we can consider maps $\widehat{f}\colon A\rightarrow [C,B]$ that are algebra homomorphisms, i.e., for all $x,y\in A$:
\[
\widehat{f}(xy)=\widehat{f}(x)\ast \widehat{f}(y), \quad\quad\quad  \widehat{f}(1_A)=1_{[C,B]}.
\]

\begin{definition}[partial homomorphism]
    \label{def: partial homomorphism}
Given $k$-algebras $A$ and $B$ and a $k$-coalgebra $C$, a $k$-linear homomorphism $\widetilde{f}\colon C\rightarrow [A,B]$ is called a \emph{$C$-indexed partial homomorphism from $A$ to $B$}\index{partial homomorphism indexed over a coalgebra!over a field} if its adjunct $\widehat{f}\colon A\rightarrow [C, B]$ is an algebra homomorphism. 
Denoting $\gls{partialhm}\coloneqq \widetilde{f}(c)$ 
for all $c\in C$, this means that for all $x,y\in A$ we have
\[
f_c(xy)=\sum_{(c)}f_{c_{(1)}}(x) f_{c_{(2)}}(y) \quad \quad \quad  f_c(1_A)=\varepsilon(c)1_B.
\]
In other words, we obtain the following commutative diagrams in vector spaces
\[
\begin{tikzcd}[row sep= large]
    C\ar{r}{\Delta} \ar{d}[swap]{\widetilde{f}}& C\otimes_k C \ar{r}{\widetilde{f}\otimes \widetilde{f}} & {[A,B]\otimes_k [A,B]} \ar{r}{\nabla} & {[A\otimes_k A, B\otimes_k B]} \ar{d}{{\mu_B}_*}\\
    {[A,B]} \ar{rrr}{\mu_A^*} &&& {[A\otimes_k A, B]}
\end{tikzcd}
\]
\[
\begin{tikzcd}
    C \ar{rr}{\widetilde{f}} \ar{d}[swap]{\varepsilon} & [-2em] & {[A,B]}\ar{d}{{\eta_B}_*}\\
    k \ar[phantom, "\cong"]{r}&  {[k,k]} \ar{r}{\eta_A^*} & {[A,k]}.
\end{tikzcd}
\]
Here, $\nabla$ is the transformation
\begin{align*}
\nabla\colon [V_1, V_2]\otimes_k [V_3, V_4]& \longrightarrow [V_1 \otimes_k V_3, V_2 \otimes_k V_4]\\
f \otimes g & \longmapsto f\otimes g
\end{align*}
natural in vector spaces $V_1$, $V_2$, $V_3$ and $V_4$.
\end{definition}

\begin{remark}
A $C$-indexed partial homomorphism is \emph{not} a partial function. It is a function that should be regarded as only partially multiplicative. Indeed, a homomorphism is multiplicative function, therefore, the term partial homomorphism should not be confused with the notion of partial function.
\end{remark}

\begin{example}
Choose $C=\mathbb{R}[x]$ and $A,B=\mathbb{R}[x]$ as in \cref{ex:binomial}.    We can view the derivation $D\colon \mathbb{R}[x]\rightarrow \mathbb{R}[x]$ as an $\mathbb{R}[x]$-indexed partial homomorphism
\begin{align*}
    \mathbb{R}[x] & \longrightarrow [\mathbb{R}[x], \mathbb{R}[x]]\\
    x^n & \longmapsto D^n.
\end{align*}
We have $\Delta(x)=x\otimes 1 + 1\otimes x$, and so we recover the formula 
\[
D(fg)=D(f)g+fD(g),
\]
and more generally
\[
D^n(fg)=\sum_{i=0}^n \binom{n}{i} D^{n-i}(f)D^{i}(g).
\]
\end{example}

\begin{example}
 We can view $\mathbb{Q}$ as a coalgebra over itself via $\mathbb{Q}\stackrel{\cong}\rightarrow \mathbb{Q}\otimes_\mathbb{Q}\mathbb{Q}$  defined by $q\mapsto q\otimes 1$.  
 We have a $\mathbb{Q}$-indexed partial homomorphism 
 \begin{align*}
     \mathbb{Q} &\longrightarrow [\mathbb{Q}, \mathbb{Q}]\\
     q & \longmapsto m_q
 \end{align*}
 where $m_q(a)=qa$. 
 The partial homomorphism says that \[m_q(ab)=q(ab)= (qa)b=m_q(a)m_1(b).\]
\end{example}

\begin{example}
Building on the previous example, given any $k$-algebra homomorphism $f\colon A\rightarrow B$, choose the trivial $k$-coalgebra structure on $k$ via $\Delta\colon k\stackrel{\cong}\rightarrow k\otimes_k k$, and define a $k$-partial homomorphism
\begin{align*}
    k& \longrightarrow [A,B]\\
    q  & \longmapsto f_q
\end{align*}
where  $f_q(x)=qf(x)$ for all $x\in A$  and $q\in k$.
As $\Delta(1)=1\otimes 1$, we get:
\[
f(xy)=f_1(xy)=f_1(x)f_1(y)=f(x)f(y),
\]
recovering the fact that $f$ is an algebra homomorphism. 
We obtain in fact a correspondence between algebra homomorphisms $A\rightarrow B$ and $k$-partial homomorphisms $k\rightarrow [A,B]$.
\end{example}

\subsection{Measuring}
By the adjunction $\Vect_k(C\otimes_k A, B)\cong \Vect_k(C, [A,B])$, we can rephrase the notion of partial homomorphisms as follows.

\begin{definition}
    [measuring]
    \label{def: measurings}
Given $k$-algebras $A$ and $B$ and a $k$-coalgebra $C$, a $k$-linear homomorphism $f\colon C\otimes_k A\rightarrow B$ is called
\emph{a $C$-measuring of $A$ and $B$}\index{measuring of algebras!over a field} if its adjunct $\widetilde{f}\colon C\rightarrow [A,B]$ is a $C$-indexed partial homomorphism, or equivalently, the other adjoint $\widehat{f}\colon A\rightarrow [C, B]$ is an algebra homomorphism.
This means that for all $c\in C$ and $x,y\in A$, we have
\[
f(c\otimes xy)=\sum_{(c)} f(c_{(1)}\otimes x) f(c_{(2)}\otimes y), \quad  \quad f(c\otimes 1_A)=\varepsilon(c)1_B.
\]
In other words, we have the following commutative diagrams in vector spaces:
\[
\begin{tikzcd}[row sep=large]
    C\otimes_k A^{\otimes 2} \ar{d}[swap]{1\otimes \mu_A}\ar{r}{\Delta \otimes 1} & C^{\otimes 2}\otimes_k A^{\otimes 2} \ar[phantom, "\cong"]{r}& [-2em] (C\otimes_k A)^{\otimes 2} \ar{r}{f\otimes f} & B\otimes_k B \ar{d}{\mu_B}\\
    C\otimes_k A \ar{rrr}{f} & & & B
\end{tikzcd}
\]
\[
\begin{tikzcd}
    C\otimes_k k \ar{d}[swap]{1\otimes \eta_A} \ar[phantom, "\cong"]{r}& [-2em] C \ar{r}{\varepsilon} & k \ar{d}{\eta_B}\\
    C\otimes_k A \ar{rr}{f} & & B.
\end{tikzcd}
\]
\end{definition}

We can summarize the equivalent definitions as in the table below.
\vspace{1em}
\begin{center}
    \begin{tabular}{|c|c|c|c|c|}
    \hline  $C\otimes_k A\rightarrow B$ &  $C\rightarrow [A,B]$ & $A\rightarrow [C,B]$  \\
      \hline  $C$-\vspace{-0.08em} & $C$-indexed partial  \vspace{-0.08em}& algebra\vspace{-0.08em} \vspace{-0.08em} \\ 
       measuring & homomorphism & homomorphism  \\\hline 
    \end{tabular}
    \end{center}
\vspace{1em}
These notions can be defined in any closed symmetric monoidal category. 
If not closed, we cannot make sense of internal hom, and thus the notion of  $C$-indexed partial homomorphism $C\rightarrow [A,B]$ is not well  defined, nor a convolution algebra $[C,B]$. However, we can still make sense of $C$-measurings $C\otimes_k A\rightarrow B$ in the non-closed case, which may explain why this term is more present in the literature.

\begin{definition}[functoriality of measurings]
    \label{def: measuring functor to set}
Let $A$, $B$ be $k$-algebras and $C$ a $k$-coalgebra.
Let \gls{MeasuringSet} be the set of $C$-measurings $C\otimes_k A\rightarrow B$ of $A$ to $B$, or equivalently, the set of $C$-indexed partial homomorphisms $C\rightarrow [A, B]$. 
Given a measuring $f\colon C\otimes_k A\rightarrow B$, a coalgebra homomorphism $c\colon C'\rightarrow C$, and algebra homomorphisms $a\colon A'\rightarrow A$ and $b\colon B\rightarrow B'$, we obtain a new $C$-measuring $f'\colon C'\otimes_k A'\rightarrow B'$ as the composite
\[
\begin{tikzcd}
C'\otimes_k A' \ar{r}{c\otimes a} & C\otimes_k A \ar{r}{f} & B \ar{r}{b} & B'. 
\end{tikzcd}
\]
By doing this for any such $f$, we obtain a function $\mu_C(A,B)\rightarrow \mu_{C'}(A',B')$.
This makes $\mu$ into a functor
%\nomenclature[b1]{$\mu$}{the measuring functor in vector spaces}
\begin{align*}
    \mu\colon \coAlg_k^\op\times \Alg_k^\op \times \Alg & \longrightarrow \Set\\
    (C, A, B) & \longmapsto \mu_C(A,B).
\end{align*}
\end{definition}

As a $C$-measuring $f\colon C\otimes_k A\rightarrow B$ (equivalently, $C$-indexed partial homomorphism $C\rightarrow [A,B]$) is equivalent to an algebra homomorphism $\widehat{f}\colon A\rightarrow [C,B]$, we get an identification %\nomenclature[b10]{$\mu_C(A,B)$}{set of $C$-measurings between $A$ and $B$}
\[
\mu_C(A,B)\cong \Alg_k(A, [C, B]),
\]
showing that $\mu_C(-, B)\colon \Alg_k^\op\rightarrow \Set$ is represented by the convolution algebra $[C,B]$. 
We will show how $\mu\colon \coAlg_k^\op\times \Alg_k^\op \times \Alg \longrightarrow \Set$ is represented in its other variables.

\subsection{Composing Partial Homomorphisms}
Given $k$-algebras $A_1$, $A_2$ and $A_3$, we can compose algebra homomorphisms $f\colon A_1\rightarrow A_2$ and $g\colon A_2\rightarrow A_3$ into an algebra homomorphism $g\circ f\colon A_1\rightarrow A_3$. We wish to do the same for partial homomorphisms.

Let $(C, \Delta_C, \varepsilon_C)$ and $(D, \Delta_D, \varepsilon_D)$ be $k$-coalgebras.
Given partial homomorphisms $f\colon C\rightarrow [A_1, A_2]$ and $g\colon D\rightarrow [A_2, A_3]$, we can define a new partial homomorphism $g\circ f\colon D\otimes_k C\rightarrow [A_1, A_3]$ by
\[
(g\circ f)_{d\otimes c}\coloneqq g_d\circ f_c.
\]
First, notice that $D\otimes_k C$ remains a coalgebra with comultiplication given by
\begin{align*}
    D\otimes_k C &\stackrel{\Delta_D\otimes \Delta_C}\longrightarrow (D\otimes_k D)\otimes_k (C\otimes_k C)\cong (D\otimes_k C)\otimes_k (D\otimes_k C)\\
    d\otimes c & \longmapsto  \sum_{(d)}\sum_{(c)} (d_{(1)}\otimes c_{(1)})\otimes (d_{(2)}\otimes c_{(2)})
\end{align*}
and counit given by
\[
D\otimes_k C \stackrel{\varepsilon_D\otimes \varepsilon_C}\longrightarrow k.
\]
We can check that $g\circ f$ remains a $(D\otimes_k C)$-partial homomorphisms. Indeed for all $x, y\in A_1$, $c\in C$ and $d\in D$, we obtain
\begin{align*}
    (g\circ f)_{d\otimes c}(xy) & = g_d(f_c(xy)) \\
    & = g_d \left( \sum_{(c)}f_{c_{(1)}}(x) f_{c_{(2)}}(y)\right)\\
    &= \sum_{(c)} g_d\left(f_{c_{(1)}}(x) f_{c_{(2)}}(y)\right)\\
    & =\sum_{(c)}\sum_{(d)}g_{d_{(1)}}(f_{c_{(1)}}(x)) g_{d_{(2)}}(f_{c_{(2)}}(y))\\
    & = \sum_{(d)}\sum_{(c)} (g\circ f)_{d_{(1)}\otimes c_{(1)}}(x) (g\circ f)_{d_{(2)}\otimes c_{(2)}}(y).
\end{align*}
Similarly, we have
\begin{align*}
    (g\circ f)_{d\otimes c}(1_{A_1}) & = g_d(f_c(1_{A_1}))\\
    & = g_d(\varepsilon_C(c)1_{A_2})\\
    & = \varepsilon_D(d)\varepsilon_C(c)1_{A_3}.
\end{align*}
Therefore we have defined a composition of partial homomorphisms
\begin{align}\label{eq: composition of partial}
\begin{split}
    \mu_D(A_2, A_3)\times \mu_C(A_1, A_2) & \longrightarrow \mu_{D\otimes_k C}(A_1, A_3)\\
    (g, f) & \longmapsto g\circ f.
    \end{split}
\end{align}
The usual identity homomorphism $\id_A\colon A\rightarrow A$ defines a $k$-indexed partial homomorphism $\id_A\colon k\rightarrow [A,A]$, i.e.~an element of $\mu_k(A,A)$. Given any partial homomorphism $f\colon C\rightarrow [A,B]$, we can check that $f\circ \id_A=f=\id_B\circ f$.

We are getting closer to being able to construct a new category of $k$-algebras, where instead of algebra homomorphisms we have partial homomorphisms.
 An issue that arises is that there is no control on the coalgebra being indexed on. That is, for what coalgebra $C$ should we set $\hom(A , B)$ to $\mu_C(A,B)$?

 \subsection{The Measuring Coalgebra}
 Fix $k$-algebras $A$ and $B$. 
 Define $\cM(A,B)$ 
 %\nomenclature{$\cM(A,B)$}{category of measurings between algebras} 
 to be the category of measurings from $A$ and to $B$. Each object consists of a pair $(C,\widetilde{f})$ where $C$ is a $k$-coalgebra and $\widetilde{f}\colon C\rightarrow [A,B]$ is a $C$-indexed partial homomorphism (equivalently a measuring $f\colon C\otimes_kA\rightarrow B$). 
 A morphism $(C, \widetilde{f})\rightarrow (D, \widetilde{g})$ consists of a coalgebra homomorphism $\gamma\colon C\rightarrow D$ such that the following commutes
 \[
 \begin{tikzcd}
C \ar{rr}{\gamma}\ar{dr}[swap]{\widetilde{f}} & & D \ar{dl}{\widetilde{g}}\\
& {[A,B].}
 \end{tikzcd}
 \]
Composition and identity on $\cM(A,B)$ are induced by $\coAlg_k$.

\begin{definition}[universal measuring]\label{def: universal measuring}
Given $k$-algebras $A$ and $B$, denote by $(\gls{UnivMeasuring}, \widetilde{\ev})$ the terminal object (if it exists) of the category $\cM(A,B)$ of measurings from $A$ to $B$. The coalgebra $\uni{A}{B}$ is called the \emph{universal measuring coalgebra}\index{universal measuring coalgebra!over a field}
%\nomenclature{$\uni{A}{B}$}{measuring coalgebra over a field}\nomenclature{$\ev$}{evaluation measuring} 
from $A$ to $B$ and its measuring $\ev\colon\uni{A}{B}\otimes_k  A\rightarrow B$ is called the \emph{evaluation}. 
This means given an $C$-indexed partial homomorphisms $\widetilde{f}\colon C\rightarrow [A, B]$,  there exists a unique coalgebra homomorphism $u_f\colon C\rightarrow \uni{A}{B}$ such that the following diagram commutes:
\[
\begin{tikzcd}
    C \ar{r}{\widetilde{f}} \ar[dashed]{d}[swap]{\exists! u_f}& {[A, B]}\\
\uni{A}{B}. \ar[bend right]{ur}[swap]{\widetilde{\ev}}
\end{tikzcd}
\]
The construction is natural in $A$ and $B$ and defines a functor
\[
\uni{-}{-}\colon \Alg_k^\op\times \Alg_k \longrightarrow \coAlg_k.
\]
\end{definition}

It is useful to think of $\uni{A}{B}$ as the coalgebra containing all possible partial homomorphisms from  $A$ to $B$. 
In order to construct $\uni{A}{B}$, we rephrase the data of a $C$-indexed partial homomorphism $\widetilde{f}\colon C\rightarrow [A,B]$ entirely coalgebraically. 
By the forgetful-cofree adjunction (\cref{eq: forgetful-cofree}), a $k$-linear homomorphism $\widetilde{f}\colon C\rightarrow [A,B]$ corresponds to a coalgebra homomorphism $\overline{f}\colon C\rightarrow \tv([A, B])$.
Following the commutative diagrams of \cref{def: partial homomorphism}, the map $\widetilde{f}$ is a partial homomorphism if and only if the map  $\overline{f}$ fits into the following equalizing diagram in $\coAlg_k$
\[
 \begin{tikzcd}
    C \ar{r}{\overline{f}} & \tv ([A, B]) \ar[shift left  =1]{r}{\Phi}\ar[shift right =1]{r}[swap]{\Psi} &  \tv\left([A^{\otimes 2}, B]\right) \oplus \tv([A, k]).
 \end{tikzcd}
 \]
 The morphism $\Phi$ is the cofree map induced by summing the $k$-linear homomorphisms $\mu_A^*\colon [A,B]\rightarrow [A^{\otimes 2}, B]$ and ${\eta_B}_*\colon [A,B]\rightarrow [A, k]$.
 The morphism $\Psi$ is induced from using the forgetful-cofree adjunction on a $k$-linear homomorphism $\tv([A,B])\rightarrow [A^{\otimes 2}, B]\oplus [A, k]$ that we now define.
 On the first summand, it is defined as the following composite
 \[
 \begin{tikzcd}[column sep=scriptsize]
    \tv([A, B])  \ar{r}{\Delta} & \tv([A, B])^{\otimes 2} \ar{r} & {[A, B]}^{\otimes 2}\ar{r}{\nabla} & [-1em]{[A^{\otimes 2}, B^{\otimes 2}]}\ar{r}{{\mu_B}_*}& {[A^{\otimes 2}, B]}
 \end{tikzcd}
 \]
 where the unlabeled arrow is induced by $\tv([A,B])\rightarrow [A, B]$, the counit of the forgetful-cofree adjunction (we have omitted $U$ for clarity).
 On the second summand, it is defined as the composite
 \[
 \begin{tikzcd}
     \tv([A,B])\ar{r} & k\cong {[k,k]} \ar{r}{\eta_A^*} & {[A, k]},
 \end{tikzcd}
 \]
 where $\tv([A,B])\rightarrow k$ is the counit of $\tv([A,B])$.

 With this coalgebraic perspective, the category $\cM(A,B)$ can now be reformulated as the category of cones over the following parallel morphisms in $\coAlg_k$ that we have just described
 \[
 \begin{tikzcd}
   \tv ([A, B]) \ar[shift left  =1]{r}{\Phi}\ar[shift right =1]{r}[swap]{\Psi} &  \tv\left([A^{\otimes 2}, B]\right) \oplus \tv([A, k]).
 \end{tikzcd}
 \]
 Therefore, the measuring coalgebra is the terminal object of the category of cones, i.e.~the limit. We have just proved the  following theorem.

\begin{theorem}[{\cite[7.0.4]{Sweedler}}]
 Given $k$-algebras $A$ and $B$, the measuring coalgebra $(\uni{A}{B}, \widetilde{\ev})$ exists and is determined by the following equalizer in $\coAlg_k$
 \[
 \begin{tikzcd}
     \uni{A}{B} \ar[dashed]{r}{\overline{\ev}} & \tv ([A, B]) \ar[shift left  =1]{r}{\Phi}\ar[shift right =1]{r}[swap]{\Psi} &  \tv\left([A^{\otimes 2}, B]\right)\oplus \tv([A, k]).
 \end{tikzcd}
 \]
\end{theorem}

The existence of the universal measuring means that we obtain the natural bijection for all coalgebras $C$ and algebras $A$ and $B$
\[
\mu_C(A,B)\cong \coAlg_k(C, \uni{A}{B})
\]
showing that $\mu_{-}(A, B)\colon \coAlg_k^\op\rightarrow \Set$ is represented by the universal measuring $\uni{A}{B}$.

Moreover, we can now rephrase the composition of partial homomorphisms from \cref{eq: composition of partial}.
Given $k$-algebras $A_1$, $A_2$, $A_3$ and $k$-coalgebras $C$ and $D$, our previously defined composition of partial homomorphisms produces a function
\begin{align*}
\coAlg_k(D, \uni{A_2}{A_3})& \times \coAlg_k(C, \uni{A_1}{A_2})\\
&\rightarrow \coAlg_k(D\otimes_k C, \uni{A_1}{A_3}).
\end{align*}
Choose $C =\uni{A_1}{A_2}$ and $D=\uni{A_2}{A_3}$, and then plug $\id_{\uni{A_2}{A_3}}$ and $\id_{\uni{A_1}{A_2}}$ into this function to obtain a morphism of coalgebras
\[
\circ \colon \uni{A_2}{A_3}\otimes_k \uni{A_1}{A_2} \longrightarrow \uni{A_1}{A_3}.
\]
It is universal in the sense that given any partial homomorphisms $f\colon C\rightarrow [A_1, A_2]$ and $g\colon D\rightarrow [A_2, A_3]$, the universal map $u_{g\circ f}$ (as in \cref{def: universal measuring}) associated to the composition $g\circ f$ fits into the commutative diagram
\[
\begin{tikzcd}
    D\otimes_k C \ar[bend left]{dr}{u_{g\circ f}}\ar{d}[swap]{u_g\otimes u_f}\\
    \uni{A_2}{A_3}\otimes_k \uni{A_1}{A_2} \ar{r}{\circ} & \uni{A_1}{A_3}.
\end{tikzcd}
\]
Moreover, the identity $\id_A$ on a algebra $A$ is an element of \[\Alg_k(A,A)\cong\mu_k(A,A)\cong \coAlg_k(k, \uni{A}{A}),\]
and thus can be viewed as a coalgebra homomorphism $\id_A\colon k\rightarrow \uni{A}{A}$.
The universal composition and identities defined above assemble into a compatible enriched structure on $\Alg_k$ (see \cref{thm: alg enriched in coalg}). Instead of a set of algebra homomorphisms, we obtain a coalgebra of all possible partial algebra homomorphisms.

\begin{example}
 Given a $k$-coalgebra $(C, \Delta, \varepsilon)$, a grouplike element in  $C$ is an element $g\in C$ such that $\Delta(g)=g\otimes g$ and $\varepsilon(g)=1$.  
 The set of grouplike elements in $C$ is in bijective correspondence with coalgebra homomorphisms $k\rightarrow C$.
 Applying this to a universal measuring $C=\uni{A}{B}$ we obtain
 \[
 \coAlg_k(k, \uni{A}{B})\cong \Alg_k(A, [k, B])\cong \Alg_k(A, B).
 \]
 Therefore, the set of grouplike elements of $\uni{A}{B}$ is precisely the hom-set of  algebra homomorphisms $A\rightarrow B$.
\end{example}

\begin{example}
Given a $k$-coalgebra $(C, \Delta, \varepsilon)$ and a grouplike element $g\in C$, a \emph{primitive element over $g$} of $C$ is an element $d\in C$ such that $\Delta(d)=d\otimes g+ g\otimes d$.
Given a $C$-partial homomorphism $\widetilde{f}\colon C\rightarrow [A,B]$, we obtain for all $x,y\in A$
\[
f_d(xy)=f_d(x)f_g(y)+f_g(x)f_d(y).
\]
Here, $f_g\colon A\to B$ is an algebra  homomorphism by the previous example, and thus $f_d\colon A\rightarrow B$ is a derivation from $A$ to $B$.  Therefore, the set of primitive elements of $\uni{A}{B}$ is the set of derivations $A\rightarrow B$.
\end{example}

\begin{example}
Given a vector space $V$ and an algebra $B$, we obtain an isomorphism of coalgebras:
\[
\uni{T(V)}{B} \cong \tv([V ,B]).
\]
This is due to the fact that for any coalgebra $C$ we have the series of invertible linear maps
\begin{align*}
  \coAlg_k(C,  \uni{T(V)}{B}) & \cong \Alg_k(T(V), [C, B])\\
  & \cong \Vect_k(V, [C, B])\\
  & \cong \Vect_k(C, [V, B])\\
  & \cong \coAlg_k(C, \tv([V, B])).
\end{align*}
\end{example}

\begin{example}
Given a $k$-coalgebra $C$ and a $k$-algebra $A$, we can view $k$ as both a $k$-algebra and $k$-coalgebra, and we obtain a bijection
\[
\coAlg_k(C, \uni{A}{k})\cong \Alg_k(A, [C, k]).
\]
The $k$-algebra $C^*\coloneqq [C, k]$ is called the \emph{dual algebra of $C$}, and the $k$-coalgebra $A^\circ\coloneqq \uni{A}{k}$ is called the \emph{finite dual} or \emph{Sweedler dual of $A$}.
We can check that in vector spaces, we have $A^\circ\cong A^*$ if and only if $A$ is finite dimensional.
We obtain natural identifications for all $C\in \coAlg_k$ and $A\in \Alg_k$:
    \[
    \balg_k(A, C^*)\cong \bcoalg_k(C, A^\circ).
    \]
\end{example}

\subsection{The Measuring Tensor}
While building the measuring coalgebra $\uni{A}{B}$, we re-framed the definition of a $C$-indexed partial homomorphism $\widetilde{f}\colon C\rightarrow [A, B]$ as a certain coalgebra homomorphism $\overline{f}\colon C\rightarrow \tv([A,B])$ using the cofree coalgebra. 
We can provide a dual framework using the free algebra.
We first make the following consideration and provide a dual formulation of the category $\cM(A,B)$ of measurings.

\begin{definition}
 Given a $k$-coalgebra $C$ and $k$-algebra $A$, define $\cM_C(A)$ to be the category whose objects consist of pairs $(B, f)$ where $B$ is a $k$-algebra and $f\colon C\otimes_k A\rightarrow B$ is a $C$-measuring.
 A morphism $(B_1, f_1)\rightarrow (B_2, f_2)$ in $\cM_C(A)$ consists of an algebra homomorphism $\beta\colon B_1\rightarrow B_2$ such that the following diagram commutes.
 \[
 \begin{tikzcd}
  &  C\otimes_k A\ar{dr}{f_2} \ar{dl}[swap]{f_1} &\\
  B_1 \ar{rr}{\beta} & & B_2
 \end{tikzcd}
 \]
 Composition and identity on $\cM_C(A)$ are induced from $\Alg_k$.
\end{definition}

\begin{definition}[measuring tensor]
Given a $k$-coalgebra $C$ and $k$-algebra $A$, denote by $(C\triangleright A, u)$ the initial object of the category $\cM_C(A)$.
The $k$-algebra $\gls{MeasTens}$ is called \emph{the measuring tensor}\index{measuring tensor!over a field} of $C$ with $A$, and its measuring $u\colon C\otimes_k A\rightarrow C\triangleright A$ is called the universal measuring of $C$ with $A$.
This means given any $C$-measuring $f\colon C\otimes_k A\rightarrow B$, there exists a unique algebra homomorphism $C\triangleright A\rightarrow B$ such that:
\[
\begin{tikzcd}
    C\otimes_k A \ar{r}{f} \ar{d}[swap]{u} &  B.\\
    C\triangleright A \ar[dashed]{ur}[swap]{\exists !}
\end{tikzcd}
\]  
\end{definition}

We can build the measuring tensor $C\triangleright A$ in a dual way as the measuring coalgebra $\uni{A}{B}$.
By the free-forgetful adjunction, a $k$-linear homomorphism $f\colon C\otimes_k A\rightarrow B$  corresponds to an algebra homomorphism ${f'}\colon T(C\otimes A)\rightarrow B$.
Following the commutative diagrams of \cref{def: measurings}, the map $f$ is a measuring if and only if its corresponding map ${f'}$ fits into the following coequalizing diagram in $\Alg_k$
\[
\begin{tikzcd}
 {T(C\otimes_k A^{\otimes 2})} \oplus T(C\otimes_k k) \ar[shift left  =1]{r}{\Phi}\ar[shift right =1]{r}[swap]{\Psi} &  T(C\otimes_k A) \ar{r}{{f'}} & B.   
\end{tikzcd}
\]
The morphism $\Phi$ is the free map induced by summing the $k$-linear homomorphism $1\otimes \mu_A\colon C\otimes_k A^{\otimes 2}\rightarrow C\otimes_k A$ and $1\otimes \eta_A\colon C\otimes_k k\rightarrow C\otimes_k A$.
The morphism $\Psi$ is induced from using the forgetful-free adjunction on a $k$-linear homomorphism $(C\otimes_k A^{\otimes 2})\oplus (C\otimes_k k)\rightarrow T(C\otimes_k A)$ that we now define.
On the first summand, it is defined as the composite
\[
\begin{tikzcd}[column sep=scriptsize]
 C\otimes_k A^{\otimes 2} \ar{r}{\Delta\otimes 1} & C^{\otimes 2}\otimes_k A^{\otimes 2}\cong (C\otimes_k A)^{\otimes 2} \ar{r} & T(C\otimes A)^{\otimes 2} \ar{r}{\mu} & [-1em]T(C\otimes A)    
\end{tikzcd}
\]
where the unlabeled arrow is induced by $C\otimes_k A\rightarrow T(C\otimes_k A)$, the unit of the free-forgetful adjunction (we have omitted to write $U$ for clarity).
On the second summand, it is defined as the composite
\[
\begin{tikzcd}
 C\otimes_k k\cong C\ar{r}{\varepsilon} & k \ar{r} & T(C\otimes_k A)
\end{tikzcd}
\]
where $k\rightarrow T(C\otimes_k A)$ is the unit of $T(C\otimes_k A)$.

With this algebraic perspective, a measuring $C\otimes_k A\rightarrow B$ can now be reformulated as a choice of a cocone over the parallel morphisms in $\Alg_k$ that we just described
\[
\begin{tikzcd}
 {T(C\otimes_k A^{\otimes 2})} \oplus T(C\otimes_k k) \ar[shift left  =1]{r}{\Phi}\ar[shift right =1]{r}[swap]{\Psi} &  T(C\otimes_k A).   
\end{tikzcd}
\]
As coequalizers in $\Alg_k$ exist, we find the following.

\begin{theorem}
    \label{thm: measuring tensor as coequalizer}
 Given a $k$-algebra $B$ and a $k$-coalgebra $C$, the measuring tensor $(C\triangleright A, u)$ exists and is the following coequalizer in $\Alg_k$
 \[
\begin{tikzcd}
{T(C\otimes_k A^{\otimes 2})} \oplus T(C\otimes_k k) \ar[shift left  =1]{r}{\Phi}\ar[shift right =1]{r}[swap]{\Psi} &  T(C\otimes_k A)  \ar[dashed]{r}{u'}  & C\triangleright A.
\end{tikzcd}
\]
\end{theorem}

\begin{remark}
   The existence of the measuring tensor was already known in \cite{Fox}.
    The notation $C\triangleright A$ first appeared in \cite{anel2013sweedler}. 
\end{remark}

Given any algebra $B$, a measuring $C\otimes_k A\rightarrow B$ corresponds to an algebra homomorphism $C\triangleright A\rightarrow B$, and thus we obtain a natural bijection
\[\mu_C(A,B)\cong \Alg_k(C\triangleright A, B)\]
showing that $\mu_C(A, -)\colon \Alg_k\rightarrow \Set$ is represented by the measuring tensor.

\begin{remark}
The measuring tensor induces a functor
\[
-\triangleright -\colon \coAlg_k \times \Alg_k \longrightarrow \Alg_k.
\]
It is useful to describe $C\triangleright A$ as follows. 
Given $c\in C$, and $x\in A$, and the universal measuring $u\colon C\otimes_k A\rightarrow C\triangleright A$,
put $c\triangleright x\coloneqq u(c\otimes x)$.
Then $C\triangleright A$ is the algebra generated by the symbols $c\triangleright x$ under the relations
\[
c\triangleright (xy) = \sum_{(c)}(c_{(1)}\triangleright x)(c_{(2)}\triangleright y), \quad \quad c\triangleright 1_A=\varepsilon(c)
\]
for all $c\in C$ and $x,y \in A$.
\end{remark}

We have now shown that the measuring functor  $\mu\colon \coAlg_k^\op\times \Alg_k^\op \times \Alg \longrightarrow \Set$ of \cref{def: measuring functor to set} is represented in each of its variables:
\[
\mu_C(A,B)\cong \Alg_k(A, [C,B]) \cong \coAlg_k(C, \uni{A}{B}) \cong \Alg_k(C\triangleright A, B).
\]
We can strengthen the identification in the following section.

\subsection{Enrichment in Coalgebras}\label{subsection: enrichment in vector space case}

Sometimes the hom-sets of a category carry extra structure. For instance, the hom-sets of the category $\Vect_k$ can be seen to have the structure themselves of a vector space. This paper is about the hom-sets of categories of algebras actually having the structure of a coalgebra. 

\begin{definition}
    [enriched category]
    \label{def: enriched category}
    Consider a monoidal category $(\V, \otimes, \I)$, the \emph{enriching} category \index{category enriched over a monoidal category}, and a category $\C$, the \emph{underlying} category.
We say that $\C$ has an \emph{enrichment} in $\V$ if there is a functor
    \[ \inthom_\C(-,-) \colon \C^\op \times \C \to \V \]
    together with the following.
    \begin{itemize}
        \item For each $A\in \C$, a morphism called the \emph{identity} on $A$:
        \[ \id_A\colon \I\rightarrow \inthom_\C(A,A).\] 
        \item For each $A, B, C\in \C$, a morphism in $\V$ called \emph{composition}
        \[
        \circ\colon \inthom_\C(B, C)\otimes \inthom_\C(A, B) \longrightarrow \inthom_\C(A, C).
        \]
        \item For each $A, B \in \C$, a bijection:
        \[ \hom_\C (A, B) \cong \hom_\C (\I, \inthomp{\C}{A}{B}). \]
    \end{itemize}
    such that the standard associativity and unitality diagrams commute, all the above morphisms are (extra)natural in $A,B,C$, and such the identity and composition $\V$-morphisms are compatible with the identity and composition morphisms of $\C$.
We also call this structure a \emph{$\V$-enriched category}, and when we are considering $\C$ with this extra structure, we will sometimes write $\underline{\C}$. The notion was introduced in \cite{benabou1965categories}, see \cite{Kelly} for a standard reference. 
\end{definition}

\begin{remark}
    It is more standard to define a $\V$-enriched category $\underline{\C}$ using only the first two bullet points, without introducing a category $\C$, and deducing an underlying category $\C$ using third bullet point. See more details in \cite{Kelly}.
\end{remark}

\begin{example}
    If $\V$ is a closed symmetric monoidal category, its internal hom, often denoted $\inthom_\V(-,-)$ determines $\V$ as enriched over itself.
\end{example}

The category $\coAlg_k$ is symmetric monoidal closed, with monoidal product given by the tensor product $\otimes_k$ in $\Vect_k$. Indeed, given $k$-coalgebras $(C, \Delta_C, \varepsilon_C)$ and $(D, \Delta_D, \varepsilon_D)$, then $C\otimes_k D$ is a $k$-coalgebra with comultiplication
\[
\begin{tikzcd}
    C\otimes_k D \ar{r}{\Delta_C\otimes \Delta_D} &[0.5 em] (C\otimes_k C)\otimes_k (D\otimes_k D)\cong (C\otimes_k D)\otimes_k (C\otimes_k D)
\end{tikzcd}
\]
and counit
\[
\begin{tikzcd}
    C\otimes_k D \ar{r}{\varepsilon_C\otimes \varepsilon_D} & k\otimes_k k \cong k.
\end{tikzcd}
\]
The monoidal structure on $\coAlg_k$ remains closed as $C\otimes_k-\colon \coAlg_k\rightarrow \coAlg_k$ preserves all colimits as they are computed in $\Vect_k$, which is closed.
Its internal hom $\gls{IntcoAlgHomk}$ is computed as the following equalizer in $\coAlg_k$
\[
\begin{tikzcd}
\underline{\coAlg}_k(C, D) \ar[dashed]{r} & \tv ([C, D])\ar[shift left  =1]{r}{\Phi}\ar[shift right =1]{r}[swap]{\Psi} &  \tv([C, D^{\otimes 2}]).
\end{tikzcd}
\]
The morphism $\Phi$ is induced by the comultiplication $D\rightarrow  D\otimes_k D$. The morphism $\Psi$ is induced the forgetful-cofree adjunction on a morphism $\tv([C,D])\rightarrow [C, D^{\otimes 2}]$ defined as the composite
\[
\begin{tikzcd}[column sep = scriptsize]
    \tv([C,D])\ar{r}{\Delta} & \tv([C,D])^{\otimes 2} \ar{r} & {[C,D]}^{\otimes 2} \ar{r}{\nabla} & [-0.5em]{[C^{\otimes 2}, D^{\otimes 2}]} \ar{r}{\Delta_C^*} & {[C, D^{\otimes 2}]}
\end{tikzcd}
\]
where the unlabeled map is induced by the counit of the forgetful-cofree adjunction.
In fact, if we restrict to the full subcategory $\coCAlg_k$ spanned by \emph{cocommutative} $k$-coalgebras, then the monoidal structure becomes cartesian closed, and thus can be viewed as an extension on the monoidal structure on $\Set$.

If $\V$ is a closed symmetric monoidal category, we say a $\V$-enriched category \emph{$\C$ is tensored over $\V$} \index{category tensored over a monoidal category} if there exists a functor $-\odot-\colon \V\times \C\rightarrow \C$ together with a natural isomorphism in $\V$
\[
\inthom_\C(V\odot A, B) \cong \inthom_\V(V, \inthom_\C(A,B))
\]
for all $V\in \V$, and $A,B\in \C$.
We say $\C$ is \emph{cotensored} (or \emph{powered}) \emph{over $\V$}\index{category cotensored over a monoidal category} if there exists a functor $-\pitchfork-\colon \V^\op\times \C\rightarrow \C$ together with a natural isomorphism in $\V$
\[
\inthom_\C(A, V\pitchfork B)\cong \inthom_\V(V, \inthom_\C(A,B))
\]
for all $V\in \V$, and $A,B\in \C$. 

We are now ready to state the main theorem in this section. 
The measuring coalgebra $\uni{A}{B}$ will provide the hom-objects of the enrichment $\inthom_{\Alg_k}(A,B)$ of algebras in coalgebras, with the universal composition that we described earlier:
\[
\uni{A_2}{A_3}\otimes_k \uni{A_1}{A_2}\rightarrow \uni{A_1}{A_3}. 
\]

\begin{theorem}\label{thm: alg enriched in coalg}
The category $\Alg_k$ of $k$-algebras is enriched in the closed symmetric monoidal category $(\coAlg_k, \otimes_k, k, \ubcoalg_k(-,-))$ of $k$-coalgebras, for which the hom-coalgebra between two algebras $A$ and $B$ is given by the universal measuring $\uni{A}{B}$. 
It is also tensored over $\coAlg_k$ via the measuring tensor and cotensored via the convolution algebra. In other words, we obtain the following natural isomorphisms in $\coAlg_k$:
\[
\underline{\coAlg}_k(C, \uni{A}{B})\cong \uni{C\triangleright A}{B} \cong \uni{A}{[C,B]}.
\]
\end{theorem}

\begin{proof}
 The convolution algebra $[-,-]\colon \coAlg_k^\op\times \Alg_k \rightarrow \Alg_k$ determines $\Alg_k$ as a right module over $\coAlg_k^\op$, in the sense that we have the following natural isomorphisms in $\Alg_k$ for all coalgebras $C,D$ and algebra $A$
 \[
 [C\otimes_k D, A]\cong [C, [D,A]] \quad \quad [k, A]\cong A.
 \]
 Thus $\Alg_k^\op$ is a left module over $\coAlg_k$ via $[-,-]^\op\colon \coAlg_k\times \Alg_k^\op\rightarrow \Alg_k^\op$.
 Moreover, for any algebra $B$, the induced functor $[-,B]^\op\colon \coAlg_k\rightarrow \Alg_k^\op$ has a right adjoint $\uni{-}{B}\colon \Alg_k^\op\rightarrow \coAlg_k$ as we have the natural isomorphism
 \[
 \coAlg_k(C, \uni{A}{B})\cong \Alg_k(A, [C, B]). 
 \]
 By \cite[3.1]{hylandetal}, we obtain that $\Alg_k^\op$ is enriched, tensored and cotensored over $\coAlg_k$.
 Thus $\Alg_k$ is enriched, tensored and cotensored over $\coAlg_k$.
\end{proof}

The above result first appeared in \cite{Sweedler} and was stated later in similar terms as us in \cite{Fox}.
The result has been extended in many different contexts, see \cite{anel2013sweedler,hylandetal,vasila,grignou,Per22,MRU22}.

\begin{remark}\label{remark: coalgebras measurings and universal internal hom}
While $\Alg_k$ is enriched in $\coAlg_k$, we do not have that $\coAlg_k$ is enriched in $\Alg_k$. However, because $\coAlg_k$ is closed symmetric monoidal, it is enriched in itself. 
This actually also defines a notion of measurings and partial homomorphism for coalgebras as we can now define. 

Let $(C, \Delta, \varepsilon)$, $(C_1, \Delta_1, \varepsilon_1)$ and $(C_2, \Delta_2, \varepsilon_2)$ be $k$-coalgebras.
A \emph{$C$-measuring} from $C_1$ to $C_2$ is a homomorphism of coalgebras $f\colon C\otimes_k C_1\rightarrow C_2$. 
We say a map $\widetilde{f}\colon C\rightarrow [C_1, C_2]$ is a \emph{$C$-indexed partial homomorphism from $C_1$ to $C_2$} if its adjunct $C\otimes_k C_1\rightarrow C_2$ is a $C$-measuring. In other words, if we denote $f_c\coloneqq \widetilde{f}(c)$ for all $c\in C$, this means for all $x\in C_1$:
\[
\sum_{(f_c(x))} {f_c(x)}_{(1)}\otimes {f_c(x)}_{(2)}=\sum_{(x)}\sum_{(c)} f_{c_{(1)}}(x_{(1)})\otimes f_{c_{(2)}}(x_{(2)})
\]
and
\[
\varepsilon_2(f_c(x))=\varepsilon(c)\varepsilon_1(x).
\]
We see that $f_c\colon C_1\rightarrow C_2$ is a homomorphism of coalgebras if $c$ is a grouplike element.

With these notions, we interpret the internal hom $\ubcoalg_k(C_1, C_2)$ as the universal measuring coalgebra from $C_1$ to  $C_2$, in the sense that, given any $C$-partial homomorphism $f\colon C\rightarrow [C_1, C_2]$,  there exists a unique homomorphism of coalgebras $u_f\colon C\rightarrow \ubcoalg_k(C_1, C_2)$ such that the following diagram commutes:
\[
\begin{tikzcd}
    C \ar{r}{f} \ar[dashed]{d}[swap]{\exists !\, u_f}& {[C_1, C_2]}\\
    \ubcoalg_k(C_1, C_2).\ar{ur}
\end{tikzcd}
\]
This is making precise the natural identification
\[
\ubcoalg_k(C\otimes_k C_1, C_2) \cong \ubcoalg_k(C, \ubcoalg_k(C_1, C_2))\cong \ubcoalg_k(C_1, \ubcoalg_k(C, C_2)).
\]
This is the coalgebraic analogue of \cref{thm: alg enriched in coalg}.
\end{remark}

\section{Data Types as (Co)algebras}
\label{sec: data types}

In this section, we explain how most (data) types in functional programming languages are obtained as initial algebras for endofunctors.

Here, we study the \emph{categorical semantics} \cite{scott1970outline} of functional programming languages. We will not make anything very precise on the programming side but will rather use that as inspiration. Precision will happen on the category theory side.

To simplify things, a functional programming language consists of (data) types, terms, and functions. Often encountered types are, for instance, the type of natural numbers, the type of booleans, and the type of strings; types are roughly analogous to sets. Like sets, types have constituents, their terms. For instance, $0$ is a term of the type of natural numbers, and \texttt{True} is a term of the type of booleans. In the programming language Haskell, this is written as $0 :: \mathtt{N}$ and $\mathtt{True} :: \mathtt{Bool}$. Just as between sets, we have functions between types. For instance, there is a function $\mathtt{Bool} \to \mathtt{N}$ that sends $\mathtt{True} \mapsto \mathtt{0}$ and $\mathtt{False} \mapsto \mathtt{1}$. In such a language, there is also a unit type $\texttt{unit}$ that has exactly one term so that functions $\mathtt{unit} \to T$ correspond to terms of $T$, and we thus think of functions as generalized terms.

In the \emph{categorical semantics} of such a programming language, we choose a category $\C$, and we `interpret' types of the programming language as objects of $\C$ and functions as morphisms of $\C$. Here, we will not make `interpretation' precise, but instead will gloss over that and make definitions directly in category theory. As mentioned, types usually bear much similarity to sets, so one often chooses $\C$ to be $\Set$, and this article is no exception.

\subsection{W-Types}\label{subsec: W-types}

In functional programming languages, it is customary to define types via their constructors. 
For instance, in Haskell, one defines the type of natural numbers by the following:
\begin{align*}
    \mathtt{{\color{Green} data \ } {\color{BrickRed} Nat \ } {\color{VioletRed} = \ } {\color{BrickRed} Zero\ } | {\ \color{BrickRed} Succ \ Nat}}.
\end{align*}
The natural numbers are usually defined in this way in functional programming these days, but the basic idea of this definition goes all the way back to Peano \cite{peano}.
Lawvere \cite{lawvere2004functorial} gave it its categorical instantiation that we use here.
We read the above definition as saying, first of all, that the type \texttt{Nat} has a term \texttt{Zero} :: \texttt{Nat} and a function $\mathtt{Succ} :: \mathtt{Nat} \to \mathtt{Nat}$ (i.e.~terms $\mathtt{Succ} n :: \mathtt{Nat}$ for every $n :: \mathtt{Nat}$).
Second of all, inherent in the \texttt{data} keyword, \texttt{Nat} is defined to be \emph{freely generated} by these two \emph{constructors}, \texttt{Zero} and \texttt{Succ}. A type being freely generated is about how the programming language allows one to define functions out of a type. For \texttt{Nat}, it suffices to define such a function on the constructors \texttt{Zero} and \texttt{Succ n}.

When taking the categorical semantics of such a programming language, the body of such a definition is interpreted as specifying an endofunctor; the interpretation of \texttt {Nat} is then taken to be the initial algebra of this endofunctor.

That is, for \texttt{Nat}, the specification above gives two constructors $\mathtt{Zero} :: \mathtt{Nat}$ (equivalently a function $\mathtt{Zero} \colon 1 \to \mathtt{Nat}$ where $1$ denotes the unit type) and $\mathtt{Succ} \colon \mathtt{Nat} \to \mathtt{Nat}$ (where $1$ is a terminal object). Categorically, the existence of such morphisms is equivalent to the existence of one morphism $\mathtt{Zero + Succ} \colon 1 + \mathtt{Nat} \to \mathtt{Nat}$. Thus, an object $X$ of a category $\C$ has the ``structure of \texttt {Nat}'' if it has a morphism $1 + X \to X$ (assuming that $\C$ has a terminal object $1$ and coproduct $1 + X$). We are led then to consider the endofunctor $1 + \id \colon \C \to \C$ that sends $X \mapsto 1 + X$, and algebras $(X, f \colon 1 + X \to X)$ of $1 + \id$.

\begin{definition}[algebra]
    Given an endofunctor $F$ on a category $\C$, an \emph{algebra} for $F$, or simply \emph{an $F$-algebra}\index{algebra!over an endofunctor}, is a pair $(A , f)$ where $A$ is an object of $\C$ together with a morphism $f\colon F A \to A$ in  $\C$, with no further assumptions. 
   An \emph{algebra homomorphism} $\eta \colon (A , f) \to (B, g)$ consists of a morphism $\eta\colon  A \to B$ such that the following diagram commutes in $\C$
    \[
         \begin{tikzcd}
             FA \ar[r,"F \eta"] \ar[d,"f "'] & FB \ar[d,"g"]
             \\ 
             A \ar[r,"\eta "] & B.
         \end{tikzcd}
    \]
    Let $\alg{\C}{F}$, or simply $\Alg$ if $\C$ and $F$ are understood,  denote the category of algebras of $F$. We denote by $\gls{AlgHom}$ the set of algebra homomorphisms $A\to B$.
\end{definition}

However, there are many sets with the structure of the natural numbers, but we want the specification above to characterize the natural numbers uniquely (up to isomorphism). We want to capture the idea that \texttt{Nat} is \emph{freely generated} by the constructors \texttt{Zero} and \texttt{Succ} -- that is, by the existence of an algebra structure. Categorically, this corresponds to taking the initial object of $\alg{\C}{F}$, the \emph{initial algebra} (an idea which goes back to \cite{Goguen-Thatcher-Wagner78}).

For instance, if we take the category $\Set$ of sets and the endofunctor $1 + \id$, then the initial object of $\alg{\Set}{1 + \id}$ is the usual natural numbers $\N$. Indeed, the natural numbers has an algebra structure: we take the function $1 \to \N$ to be the one corresponding to $0 \in \N$, and we take the function $\N \to \N$ to be the one taking $n \mapsto n + 1$ (where here, $n + 1$ refers to the usual addition within $\N$, not a coproduct with the terminal object). It is initial because given any other algebra $(A, f)$, we can construct a function $i\colon \N \to A$ by sending $0 \mapsto f(*)$ and inductively sending $n + 1 \mapsto f i (n)$. Indeed, the algebra structure on $A$ is \emph{exactly} the information needed to inductively construct a function $\N \to A$.

In programming languages such as Haskell, the main way of defining types is via their constructors. In this way, one can define any finite type (i.e., type corresponding to a finite set), lists with values of any other type, various varieties of trees, etc.

One widely considered class of such types are called \emph{W-types}, originally introduced with Martin--Löf type theory \cite{ml-wtypes}. Categorically, W-types correspond to initial algebras for polynomial endofunctors \cite{dybjer,mp}.

\begin{definition}
    [polynomial endofunctor, W-type]
    \label{def: poly endofunctor and w-type}
    Consider a category $\C$, objects $A, B$ of $\C$, and a morphism $f \colon A \to B$. The endofunctor $\gls{PolyEndo} \colon \C \to \C$ is the composition displayed below (if it exists):
    \[
         \begin{tikzcd}
          \C\ar{r}{-\times A}  &    \C / A \ar[r,"\prod_f"]  &  \C / B \ar[r,"\dom"] &  \C.
         \end{tikzcd}
    \]
    Here, $- \times A$ denotes the functor that takes an object $X \in \C$ to the object $\pi_A \colon X \times A \to A$ in the slice category $\C / A$. The functor $\prod_f \colon \C / A \to \C / B$ denotes the right adjoint to pullback along $f$, and $\dom$ is the functor that sends an object $X \to B$ of $\C / B$ to $X$.
 By \emph{polynomial endofunctor on $\C$} we mean any endofunctor on $\C$ obtained in this way.
Given a  $f \colon A \to B$, we denote by $W(f)$ the initial algebra of the endofunctor $\mathcal P_f$. A \emph{(categorical) W-type} \index{W-type} is an algebra of the form $W(f)$.
\end{definition}

\begin{example}
Consider polynomial endofunctors on $\Set$. Consider two sets $A$ and $B$ and a function $f \colon A \to B$. Note that for an object $X \in \C$, we have
\begin{align*}
   \mathcal P_f X &\coloneqq \dom \ \prod_f (X \times A) \\
    &\cong \sum_{b \in B} {\Set / B}(b , \prod_f (X \times A))\\ 
    &\cong \sum_{b \in B} {\Set / A}(f^{-1} b , X \times A) \\
    &\cong \sum_{b \in B} {\Set} (f^{-1} b , X).
\end{align*}
Letting $Fb \coloneqq f^{-1} b$, we find that $\mathcal P_f X \cong \sum_{b \in B} X^{Fb}$ (where $X^{Fb}$ denotes exponentiation, i.e., the internal hom, in $\Set$).  Now, we find that $\mathcal P_f$ does indeed look like a polynomial.
\end{example}

\begin{example}
    Consider the endofunctor $X \mapsto 1 + X$ on $\Set$. We claim that we can write this as a polynomial endofunctor. That is, we want to find an $f\colon A \to B$ such that $1 + X \cong \sum_{b \in B} X^{Fb}$. Since there are two summands in $1 + X$, we take $B$ to be the set $\{\mathtt{Zero},\mathtt{Succ}\}$. We want $X^{F(\mathtt{Zero})} \cong 1$, so we take $F(\mathtt{Zero}) \cong \emptyset$. We want $X^{F(\mathtt{Succ})} \cong X$, so we take $F(\mathtt{Succ}) \cong 1$. Thus, we take $f$ to be the inclusion $\{\mathtt{Succ}\} \to \{\mathtt{Zero},\mathtt{Succ}\}$.
\end{example}

\begin{example}
    \label{ex:lists-1}
    Given a type $\mathtt{a}$, we define the type of lists in $\mathtt{a}$ via the following specification in Haskell:
    \begin{align*}
        \mathtt{{\color{Green} data \ } {\color{BrickRed} List \ a \ } {\color{VioletRed} = \ } {\color{BrickRed} Nil\ } | {\ \color{BrickRed} Cons \ a \ (List \ a)}}.
    \end{align*}
    This has two constructors: $\mathtt{Nil} \colon 1 \to \mathtt{List \ a}$ and $\mathtt{Cons} \colon \mathtt{a} \times \mathtt{List \ a} \to \mathtt{List \ a}$, in categorical terms. This is equivalently a map $1 + \mathtt{a} \times \mathtt{List \ a} \to \mathtt{List \ a}$. That is, we want $\mathtt{List \ a}$ to be an algebra of the endofunctor $1 + \mathtt{a} \times - $.
    We consider this endofunctor on $\Set$. We have that:
    \begin{align*}
        1 + \mathtt{a} \times X &\cong X^0 + \mathtt{a} \times X^1 \\
        &\cong X^0 + \sum_{t \in \mathtt{a}} X^1 \\
        &\cong \sum_{t \in 1 + \mathtt{a}} X^{Ft},
    \end{align*}
    where $F \colon 1 + \mathtt{a} \to \Set$ sends the element of $1$ to $0$, the empty set, and any $t \in \mathtt{a}$ to $1$. Thus, this is the polynomial endofunctor associated to the inclusion $i: \mathtt{a} \hookrightarrow 1 + \mathtt{a}$; that is $\mathcal P _i \cong 1 + \mathtt{a} \times - $.
    That is, where one could say that our Haskell specification has two constructors, \texttt{Nil} and \texttt{Cons}, the W-type specification has $(1 + \mathtt{a})$-many constructors.

    We now claim that the initial algebra of this endofunctor is the set of lists with entries from $\mathtt{a}$.
    To be a bit more precise, it is $\sum_{n \in \N} \mathtt{a}^n$. This has an algebra structure: we take \texttt{Nil} to correspond to the single element of $\mathtt{a}^0$ in the disjoint union, and we take $\texttt{Cons} \colon\mathtt{a} \times \mathtt{List \ a} \to \mathtt{List \ a}$ to send $(a \in \mathtt{a}, \ell \in \mathtt a^n)$ to the element $a.\ell \in \mathtt a^{n+1}$ where $a.\ell (0) \coloneqq a$ and $a.\ell(m+1) \coloneqq \ell(m)$ for any $m \in \N$. It is initial because we can use any  structure $n + c \colon 1 + \mathtt{a} \times L \to L$ of any other algebra to ``inductively'' build a unique algebra homomorphism $u: \mathtt{List \ a} \to L$. 
    Indeed, we send the single element of the summand $a^0$ in $\mathtt{List \ a}$ to the element corresponding to $n : 1 \to L$,
    and then we inductively send an element $\texttt{Cons}(a,\ell) \in a^{n+1}$ to $c(a,u\ell)$. One can check that this is an algebra homomorphism (by construction) and is the unique such. 
\end{example}

\begin{theorem}[\cite{dybjer,mp}]
    Given any function $f\colon A \to B$ in $\Set$, the categorical W-type $W(f)$ exists \cite[Thm.~1]{dybjer}.
    In fact, all categorical W-types exist in any topos \cite[Prop.~3.6]{mp} with a natural numbers object.
\end{theorem}

There is an important construction of initial algebras that we do not explicitly use in this article, but that provides a lot of intuition.

\begin{theorem}[Adámek's fixed point theorem, \cite{adamek74}]
    \label{thm:adamek}
    Consider a category $\C$ with an initial object $0$ and colimits of sequences $\omega \to \C$, where $\omega$ denotes the countably infinite ordinal. Consider an endofunctor $F$ that preserves these colimits. Then the colimit of the following sequence is an initial algebra of $F$:
    \[ 0 \xrightarrow{!} F0 \xrightarrow{F!} F^2 0 \to \cdots. \]
\end{theorem}

One can check that polynomial endofunctors $\mathcal P_f$ preserve such colimits when all fibers of $f$ are finite, as in all of our examples. For instance, for the endofunctor $1 + \id$ on $\Set$, we can compute the initial algebra by taking the colimit of the above sequence, which becomes the following.
\[ 0 \hookrightarrow 1 \hookrightarrow 2 \hookrightarrow \cdots \]
Now, we can easily see that it (or at least its underlying set) is $\mathbb N$.

\subsection{M-Types}

In this previous section, we investigated the categorical semantics of types as initial algebras. This begs the question: is there a dual story, and what does it capture?

Indeed, there is a dual story about \emph{co}algebras, \emph{terminal} coalgebras, and M-types (the letter M being an upside-down W). In this story, we take the specifications of the types to be the same as in the preceding story. That is, we take the same endofunctors -- in particular, polynomial endofunctors -- and we consider terminal coalgebras of them.

Note that in some categories -- for instance the category of DCPOs, used to model languages like Haskell -- initial algebras and terminal coalgebras coincide \cite{sp}. When this happens, they are often called \emph{canonical fixed points}. In others -- for instance $\Set$, used to model languages like Coq and Agda -- they do not.

\begin{definition}[coalgebra]
    Consider an endofunctor $F$ on a category $\C$. A \emph{coalgebra} of $F$,\index{coalgebra!over an endofunctor} or simply an \emph{$F$-coalgebra}, is a pair $(C, f)$ where $C$ is an object of $\C$ and $f\colon C \to F C$ is a morphism in $\C$. A \emph{coalgebra homomorphism} $h\colon (C, f) \to (D, g)$ consists of a morphism $h\colon C \to D$ in $\C$ such that the following diagram commutes:
    \[
         \begin{tikzcd}
             C \ar[r,"h"] \ar[d,"f "'] & D \ar[d,"g"]
             \\ 
             FC \ar[r,"F h"] & FD.
         \end{tikzcd}
    \]
    The category of coalgebras, denoted $\coAlg_{\C}{F}$ or simply $\coAlg$ if $\C$ and $F$ are understood, has coalgebras of $F$ as objects and coalgebra homomorphisms as morphisms. 
    We denote by $\gls{coAlgHom}$ the set of $F$-coalgebra homomorphisms.
\end{definition}

\begin{definition}[M-type]
    Given a morphism $f\colon A \to B$ in a category $\C$, we denote by $M(f)$ the terminal coalgebra of $\mathcal P_f$. By \emph{(categorical) M-type} we mean any coalgebra of the form $M(f)$. \index{M-type}
\end{definition}

There is an important relationship between initial algebras and terminal coalgebras. First, the algebra structure morphism for an initial algebra is always an isomorphism (by the next proposition). 
This is why initial algebras and terminal coalgebras are often called (least or greatest) \emph{fixed points} for the endofunctor.
This isomorphism makes an initial algebra into a coalgebra, and there is then an induced morphism from the initial algebra to the terminal coalgebra. The dual story holds for terminal coalgebras, but the resulting morphism from the initial algebra to the terminal coalgebra is the same.

\begin{proposition}[\cite{lambek}]
    \label{prop:lambek}
    Consider an endofunctor on a category with an initial algebra $(A, \alpha)$ and a terminal coalgebra $(C,\chi)$. Then we find the following.
    \begin{enumerate}
        \item The morphisms $\alpha$ and $\chi$ are isomorphisms.
        \item The pair $(A, \alpha^{-1})$ is a coalgebra. Dually, $(C,\chi^{-1})$ is an algebra.
        \item There is a morphism $A \to C$.
    \end{enumerate}
\end{proposition}

Moreover, M-types are often an infinitary `completion' of a W-type, in the sense suggested by the following examples. Indeed, W-types can be thought of (and are often called) ``well-founded trees'' while M-types are ``non-well-founded trees''. For precise statements, see the discussion of this relationship within homotopy type theory in \cite{rech2017strictly}.

\begin{example}\label{example: terminal coalgebra of Nat}
    \label{ex: terminal coalgebra for 1 + -}
    Consider the endofunctor $1 + \id \colon\Set \to \Set$. We saw above that its initial algebra is $\N$.
We claim that its terminal coalgebra is $\gls{TermiCoalgNat}$. The underlying set of $\eN$ is $\N + \{\infty\}$. Its coalgebra structure $\eN \to 1 + \eN$ is given by sending $0$ to the unique element of the summand $1$ and any other $n$ to $n-1$, where $\infty - 1 \coloneqq \infty$.

    Given any other coalgebra $\chi\colon C \to 1 + C$, we let $\ind{c}$ denote the least $n$ such that $\chi^n (c)$ is the unique element of the summand $1$. One can check that this defines the only coalgebra homomorphism $\gls{TermiCoalgNatMap} \colon C \to \eN$.
The induced map $\N \to \eN$ from \cref{prop:lambek} is the inclusion.
\end{example}

\begin{example}
    Consider the endofunctor $1 + \mathtt{a} \times - \colon \Set \to \Set$. We saw in \cref{ex:lists-1} that its initial algebra is the set of (finite) lists with elements from $\mathtt{a}$.
We claim that the underlying set of its terminal coalgebra is the set of potentially infinite lists, often called \emph{streams}, with elements from $\mathtt{a}$, that is $\sum_{n \in \eN} \mathtt{a}^n$.
    The coalgebra structure $\sum_{n \in \eN} \mathtt{a}^n \to 1 + \mathtt{a} \times \sum_{n \in \eN} \mathtt{a}^n$ is given by sending the element of the summand $\mathtt{a}^0$ to the element of the summand $1$. We send an $\ell \in \mathtt{a}^\infty$ to the pair $(\ell(0),\ell)$. For other $n$, we send an element $\ell \in \mathtt{a}^n$ to the pair $(\ell(0), \ell' )$ where $\ell'(m) \coloneqq \ell(m-1)$.

    There is a unique coalgebra homomorphism from any other coalgebra $C$ to this one of streams. Intuitively, one sees the map $\chi\colon C \to 1 + \mathtt{a} \times C$ as a partially defined function $\chi\colon C \to \mathtt{a} \times C$. For any $c \in C$, $\chi(c)$, if defined, is a pair of a `head' in $\mathtt{a}$ and `tail' in $C$. By repeatedly applying $\chi$ to any $c$ until it is not defined anymore, one can construct a stream from $c$. One can check that this is the unique coalgebra homomorphism.
Just as in the above example, the induced map from the initial algebra to the terminal coalgebra is the natural inclusion of lists into streams.
\end{example}

We also have the dual to Adámek's fixed point theorem for initial algebras.

\begin{theorem}[Adámek's fixed point theorem, \cite{adamek74}]
    Consider a category $\C$ with a terminal object $1$ and limits of sequences $\omega \to \C$, where $\omega$ denotes the countably infinite ordinal. Consider an endofunctor $F$ that preserves these limits. Then the limit of the following sequence is a terminal coalgebra of $F$:
    \[ \cdots \to F^2 1 \xrightarrow{F !} F 1 \xrightarrow{!} 1. \] 
\end{theorem}

One can check that polynomial endofunctors $\mathcal P_f$ preserve connected limits. So for instance, for the endofunctor $1 + \id$ on $\Set$, we can compute the terminal coalgebra by taking the limit of the above sequence, which becomes
\[ \cdots \to 3 \to 2 \to 1 , \]
where each morphism identifies two elements of the domain.
Now, we can see that it (or at least its underlying set) is $\eN$.

\section{Measuring of Inductive Data Types}
\label{sec: measuring}

In this section, we now show that there is a version of Sweedler's theory of measuring coalgebras (\cref{sec: sweedler coalgebras}) for (co)algebras of endofunctors (\cref{sec: data types}).
So fix a symmetric monoidal closed category $(\C, \otimes, \gls{MonUnit}, [-,-])$
%\nomenclature[a2]{}{internal hom in a closed monoidal category} 
and a lax symmetric monoidal endofunctor $(F, \nabla, \eta)$ (defined below in \cref{definition: lax symmetric monoidal}) on $\C$.

\subsection{Lax Symmetric Monoidal Endofunctors}

In order to define the general notion of \emph{measuring} for algebras over an endofunctor $F$, we first need to require that $F$ is lax symmetric monoidal.

\begin{example}
    Our running example throughout this section will be $\Set$ with its cartesian monoidal structure, and the endofunctor $X \mapsto X + 1$.
\end{example}

\begin{definition}[lax symmetric monoidal endofunctor]
    \label{definition: lax symmetric monoidal}
 That $(F, \nabla, \eta)$ is a \emph{lax symmetric monoidal endofunctor}\index{lax symmetric monoidal endofunctor} means that $F$ is an endofunctor on $\C$ with
 \begin{description}
     \item[{(L1)}:] a natural transformation $\nabla_{X,Y}\colon F X \otimes FY\longrightarrow F(X\otimes Y)$, for all $X,Y\in \C$; and
     \item[{(L2)}:] a morphism $\eta\colon \I\rightarrow F\I $ in $\C$;
 \end{description}
 such that $(F, \nabla, \eta)$ is associative, unital, and symmetric, in the sense that the following diagrams commute.
 \begin{description}
    \item[(Associativity):] For all objects $X,Y, Z\in \C$, the following diagram commutes:
\[
\begin{tikzcd}[row sep= large]
 FX \otimes FY \otimes FZ \ar{r}{\nabla_{X,Y}\otimes \id}\ar{d}[swap]{\id\otimes \nabla_{Y,Z}} & [2.5em] F(X\otimes Y) \otimes FZ \ar{d}{\nabla_{X\otimes Y, Z}}\\
 FX\otimes F(Y\otimes Z) \ar{r}{\nabla_{X, Y\otimes Z}} & F(X\otimes Y \otimes Z).
   \end{tikzcd}
   \]
   
   \item[(Unitality):] For all objects $X\in \C$, the following diagrams commute:
\[
\begin{tikzcd}[column sep=scriptsize]
FX\ar[equals]{ddrr}\ar{r}{\cong} & [-1em] \I\otimes FX  \ar{r}{\eta\otimes \id} & F\I\otimes FX\ar{d}{\nabla_{\I, X}} & FX \ar{r}{\cong} \ar[equals]{ddrr} & [-1em]FX\otimes \I \ar{r}{\id\otimes \eta} & FX\otimes F\I \ar{d}{\nabla_{X, \I}}\\
 & & F(\I\otimes X) \ar{d}{\cong} & & & F(X\otimes \I)\ar{d}{\cong}\\
 [-1em]  & & FX & & & FX.
\end{tikzcd}
\]
\item[(Symmetry):] For all objects $X,Y\in \C$, the following diagram commutes
\[
\begin{tikzcd}
   FX \otimes FY \ar{r}{\nabla_{X,Y}} \ar{d}{\cong}[swap]{\tau_{FX,FY}}& F(X\otimes Y)\ar{d}{F(\tau_{X,Y})} [swap]{\cong}\\
   FY \otimes FX \ar{r}{\nabla_{Y,X}} & F(Y\otimes X),
\end{tikzcd}
\]
where $\tau_{X,Y}:X\otimes Y\rightarrow Y\otimes X$ is the symmetry natural isomorphism in $\C$.
\end{description}
\end{definition}

\begin{remark}
    One may also say \emph{symmetric lax monoidal} instead of lax  symmetric monoidal. 
    Notice the symmetry is not truly needed in our work but makes the statement of enrichment in \cref{thm:enriched} easier.
\end{remark}

\begin{example}
    For the endofunctor $\id + 1$, we define
    \begin{align*}
        \nabla_{X,Y}\colon (X+1)\times (Y+1) & \longrightarrow  (X\times Y)+1\\
        (x,y) & \longmapsto  \begin{cases}
            (x, y) & \text{if  }x\neq \tone,   y \neq \tone\\
            \tone & \text{otherwise.}
        \end{cases}
    \end{align*}
    for $x \in X, y \in Y, \tone \in 1$. 
    We define $\eta \colon 1 \to 1 + 1$ to be the inclusion into the first summand.
\end{example}

Note that since $F$ is lax monoidal, it is also \emph{lax closed}: that is, there is a map
\[
\widetilde{\nabla}_{X,Y}\colon F\left( [X,Y]\right) \stackrel{}\longrightarrow [FX, FY]
\]
natural in  $X,Y \in \C$.
As $-\otimes X\colon \C\rightarrow \C$ has the right adjoint $[X, -]\colon \C\rightarrow \C$, the above map is the adjunct  of the composition
\[
\begin{tikzcd}
F([X,Y])\otimes F(X) \ar{r}{\nabla_{[X,Y], X}}  & [2em] F([X,Y] \otimes X) \ar{r}{F(\ev_{X})} &[1em] F(Y),
\end{tikzcd}
\]
in which $\ev_{X}$ is the counit of the adjunction $- \otimes X \dashv [X, -]$.

\subsection{The Convolution Algebra}
Let $F\colon \C\rightarrow \C$ be any lax symmetric monoidal endofunctor.
Given $(C, \chi)$ an $F$-coalgebra and $(B, \beta)$ an $F$-algebra, then the object \gls{ConvAlg} in $\C$ is also an $F$-algebra via the \emph{convolution product} 
\[
\begin{tikzcd}
F([C, B]) \ar{r}{\widetilde{\nabla}_{C,B}} & {[FC, FB]} \ar{r}{\beta_*} & {[FC, B]} \ar{r}{\chi^*} & {[C,B]}.
\end{tikzcd}
\]
We call this $F$-algebra  $[C,B]$ the \emph{convolution algebra} of $C$ with  $B$.\index{convolution algebra!over an endofunctor}
The construction is natural in $C$ and $B$ and defines a functor
\[\coAlg^\op\times \Alg  \longrightarrow \Alg.\]

\begin{example}
The lax monoidal structure on $F$ determines an $F$-coalgebra structure on $\eta\colon \I\rightarrow  F(\I)$. 
One can check that the resulted convolution algebra $[\I, B]$ is isomorphic as an $F$-algebra to $B$.
\end{example}

\subsection{Measurings and Partial Algebra Homomorphisms}
\label{sec: Measurings and partial algebra homomorphisms}
Recall that as $\C$ is closed, we obtain natural bijections
\[
\C(C\otimes A, B) \cong \C(C, [A,B]) \cong \C(A, [C,B]),
\]
for any object $A,B, C$ in $\C$. 
Given a lax symmetric monoidal endofunctor $F\colon \C\rightarrow \C$, if we require $A$ and $B$ to be $F$-algebras, and  $C$ an $F$-coalgebra, then as $[C,B]$ is an $F$-algebra via the convolution algebra, we may consider the maps $A\rightarrow [C,B]$ in $\C$ to be $F$-algebra homomorphisms instead. By the identification above, they correspond to certain maps $C\otimes A\rightarrow B$ and $C\rightarrow [A,B]$ in $\C$ with additional conditions that we now describe.

\begin{definition}[measuring]\label{definition: measuring coalgebra}
Consider $F$-algebras $(A,\alpha)$ and $(B, \beta)$, and a $F$-coalgebra $(C, \chi)$. 
We call a map $\phi\colon C\otimes A\rightarrow B$ a \emph{measuring from $A$ to $B$}\index{measuring of algebras!over an endofunctor} if it makes the following diagram commute:
    \[
     \begin{tikzcd}
       & F(C)\otimes F(A) \ar{r}{\nabla_{C,A}} & F(C\otimes A) \ar{r}{F(\phi)} & F(B) \ar{dd}{\beta}\\
      C\otimes F(A) \ar{ur}{\chi\otimes \id} \ar{dr}[swap]{\id\otimes \alpha} \\
      & C\otimes A \ar{rr}{\phi} & & B.
     \end{tikzcd}
    \]
    We denote by \gls{MeasuringSet} the set of all measurings $C\otimes A\rightarrow B$.
\end{definition}

If $\phi\colon C\otimes A \rightarrow B$ is a measuring, $a\colon (A', \alpha') \rightarrow (A, \alpha)$ and $b\colon (B, \beta) \rightarrow (B', \beta')$ are algebra homomorphisms, and $c\colon (C', \chi')\rightarrow (C, \chi)$ is a coalgebra homomorphism, then one can check that the composite
\[
\begin{tikzcd}
   C'\otimes A' \ar{r}{c\otimes a} & C\otimes A \ar{r}{\phi} & B\ar{r}{b} & B'
\end{tikzcd}
\]
is a measuring. Therefore, the assignment $(C,A,B) \mapsto \mu_C(A,B)$ underlies a functor
\[
\mu \colon \bcoalg^\op \times \balg^\op \times \balg \longrightarrow \Set. 
\]
We shall see that this functor is representable in each of its variables under reasonable hypotheses. The convolution algebra provides a first representing object for $\mu_C(-,B)\colon \balg^\op \to \Set$. Indeed, we have the following bijection natural in $C,A,B$:
\[
\mu_C(A,B)\cong \balg(A, [C, B]).
\]
In other words, a measuring $\phi\colon C\otimes A\rightarrow B$ corresponds to an algebra homomorphism $\widehat{\phi}\colon A\rightarrow [C,B]$ under the bijection $\C(C\otimes A, B)\cong \C(A, \inthomp{\C}{C}{B})$. Indeed, notice that $\widehat{\phi}$ is a homomorphism if and only if the following diagram, the adjunct of the one appearing in \cref{definition: measuring coalgebra}, commutes:
 \[
     \begin{tikzcd}
       & F([C,B]) \ar{r}{\widetilde{\nabla}_{C,B}} &[2em] {[F(C), F(B)]} \ar{d}{\beta_*}\\
      F(A) \ar{ur}{F(\widehat{\phi})} \ar{dr}[swap]{\alpha} & & {[FC, B]}\ar{d}{{\chi}^*} \\
      & A \ar{r}{\widehat{\phi}} &  {[C,B].}
     \end{tikzcd}
    \]

\begin{example}
    The monoidal unit $\I$ of $\C$ is a coalgebra via the lax symmetric monoidal structure $\eta\colon\I\rightarrow F(\I)$. Thus morphisms $A \to B$ in $\C$ are in bijection with morphisms $\I \otimes A \to B$ in $\C$, and one can check that a morphism $A \to B$ in $\C$ is an algebra homomorphism if and only if $\I \otimes A \to B$ is a measuring. Thus, $\mu_\I( A, B)\cong\balg(A,B)$.
\end{example}

\begin{example}
    \label{ex: measuring for - + 1}
    Consider algebras $A$ and $B$ and a coalgebra $C$ for the endofunctor $X \mapsto 1 + X$.
    A measuring from $A$ to $B$ by $C$ is a function $f  \colon C \times A \rightarrow B$ such that the following three properties hold:
    \begin{description}
        \item[{(M1)}:] $f_c (0_A)=0_B$ for all $c \in C$;
        \item[{(M2)}:] $f_c(a+1)=0_B$ for all $\llbracket c \rrbracket=0$ and for all $a\in A$;
        \item[{(M3)}:] $f_c(a+1)=f_{c-1}(a)+1$ for $\llbracket c \rrbracket \geq 1$ and for all $a\in A$. 
    \end{description}
    Here, we write $f_c (a)$ for $f(c,a)$. Since $A$ has an algebra structure $A + 1 \to A$, it has a distinguished element, that we write above as $0_A$, and an endofunction, that we write above as $a + 1$; similarly for $B$. Since $C$ has a coalgebra structure, there is a unique map $\ind{-} : C \to \eN$ (recalling that $\eN$ is the terminal coalgebra from \cref{ex: terminal coalgebra for 1 + -}). The coalgebra structure map is a partial endofunction that we denote as $c - 1$ above.
\end{example}

Another formulation of measurings is possible. An algebra homomorphism $A\rightarrow [C, B]$ corresponds also to a certain map $C\rightarrow [A,B]$ in $\C$ that we now describe. 

\begin{definition}[partial homomorphism]
    \label{definition: abstract partial homomorphism}
    Given algebras $(A, \alpha)$ and $(B,\beta)$ and a coalgebra $(C,\chi)$, a \emph{$C$-indexed partial homomorphism} from $A$ to $B$\index{partial homomorphism indexed over a coalgebra!over an endofunctor} is a map $\widetilde{\varphi}\colon C\rightarrow [A,B]$ in $\C$ such that its adjunct $\widehat{\varphi}\colon A\rightarrow [C, B]$ is an algebra homomorphism i.e., 
    fits in the following commutative diagram
    \[
        \begin{tikzcd}
          & F(C) \ar{r}{F(\widetilde{\phi})} &[1em] F([A,B]) \ar{r}{\widetilde{\nabla}_{A,B}} & [1em] {[FA, FB]} \ar{dd}{\beta_*}\\
         C \ar{ur}{\chi} \ar{dr}[swap]{\widetilde{\phi}} \\
         & {[A,B]} \ar{rr}{\alpha^*} & & {[FA,B]}
        \end{tikzcd}
       \]
    where $\alpha^*$ denotes precomposition by $\alpha$ and $\beta_*$ denotes postcomposition by $\beta$. 
\end{definition}

By construction, a $C$-indexed partial homomorphism $\widetilde{\varphi}\colon C\rightarrow [A,B]$ corresponds to a $C$-measuring $\varphi\colon C\otimes A\rightarrow B$.

\begin{definition}[universal measuring coalgebra]\label{definition: universal measuring}
Let $A$ and $B$ be $F$-algebras. We define the category of measurings $\cM(A,B)$ from $A$ to $B$ to be the category whose objects are pairs $(C;\widetilde{\varphi})$ of an $F$-coalgebra $C$ and a $C$-indexed partial homomorphism $\widetilde{\varphi}\colon C\rightarrow [A, B]$ (or equivalently a measuring $\varphi\colon C \otimes A \to B$), and whose morphisms $(C;\widetilde{\varphi}) \to (D;\widetilde{\psi})$ are coalgebra homomorphisms $\gamma\colon C \to D$ such that the following diagram commutes
\[
 \begin{tikzcd}
C \ar{rr}{\gamma}\ar{dr}[swap]{\widetilde{\varphi}} & & D \ar{dl}{\widetilde{\psi}}\\
& {[A,B].}
 \end{tikzcd}
 \]
The \emph{universal measuring coalgebra from $A$ to $B$}\index{universal measuring coalgebra!over an endofunctor}, denoted $(\gls{UnivFMeasuring}, \widetilde{\ev})$, is the terminal object (if it exists) in $\cM(A,B)$. 
%\nomenclature{$\cM(A,B)$}{category of measurings between algebras}
This means given an $C$-indexed partial homomorphisms $\widetilde{\varphi}\colon C\rightarrow [A, B]$,  there exists a unique coalgebra homomorphism $u_\varphi\colon C\rightarrow \ubalg({A},{B})$ such that the following diagram commutes
\[
\begin{tikzcd}
    C \ar{r}{\widetilde{\varphi}} \ar[dashed]{d}[swap]{\exists! u_\varphi}& {[A, B]}\\
\ubalg(A,B). \ar[bend right]{ur}[swap]{\widetilde{\ev}}
\end{tikzcd}
\]
The construction is natural in $A$ and $B$ and defines a functor:
\[
\ubalg(-,-)\colon \Alg^\op\times \Alg \longrightarrow \coAlg.
\]
\end{definition}

If a universal measuring $(\ubalg(A,B), \ev)$ exists, then we obtain a representation $\ubalg(A,B)$ for the functor $\mu_-(A,B)$. That is, we have the following identification, natural in $C,A,B$:
\[
\Alg(A, [C, B])\cong \mu_C(A,B)\cong \bcoalg(C, \ubalg(A,B)).
\]
In the following section, we show that if $\C$ is closed and locally presentable and $F$ is accessible, then the universal measuring coalgebra always exists.

\subsection{Local Presentability and the Existence Theorem}

We will now usually require that $\C$ be \emph{locally presentable} and $F$ be \emph{accessible} \cite[Def.~1.17~\&~2.17]{presentable}. Then $\balg$ and $\bcoalg$ are also locally presentable, the forgetful functor $\balg \to \C$ has a left adjoint \gls{freeFAlg} \index{free algebra!over an endofunctor}, and the forgetful functor $\bcoalg \to \C$ has a right adjoint \gls{cofreeFcoAlg} \cite[Cor.~2.75~\&~Ex.~2.j]{presentable}\index{cofree coalgebra!over an endofunctor}. We will also use that these categories, as locally presentable categories, are complete and cocomplete.

\begin{example}
    $\Set$ is locally presentable and $\id + 1$ is accessible, as it preserves all filtered colimits (in fact all connected colimits).
\end{example}

\begin{example}
    The following endofunctors on a locally presentable, symmetric monoidal category $\C$ are always accessible and lax symmetric monoidal. 
    \begin{description}
           \item[$(\id)$:] The identity endofunctor $\id_\C$ has a trivial lax symmetric monoidal structure and preserves filtered colimits.
           \item[$(A)$:] Given a commutative monoid $(A, \eta, \nabla) \in \C$, the constant endofunctor that sends each object to $A$ has a lax symmetric monoidal structure. Indeed, $\eta$ and $\nabla$ are exactly the structure required for this functor to be lax symmetric monoidal. This preserves filtered colimits.
           \item[$(GF)$:] Suppose that $(F, \eta_F, \nabla_F)$ and $(G, \eta_G, \nabla_G)$ are lax symmetric monoidal endofunctors on $\C$. Then so is the composition $GF$. The map $\eta$ is the composition $\I \xrightarrow{\eta_G} G (\I) \xrightarrow{G \eta_F} GF(\I)$. The map $\nabla$ is the composition $GF(X) \otimes GF(Y) \xrightarrow{\nabla_G} G(F(X) \otimes F(Y)) \xrightarrow{G\nabla_F} GF(X \otimes Y)$. If $F$ and $G$ both preserved filtered colimits, then so does $GF$.
           \item[$(F \otimes G)$:] Suppose that $\C$ is closed and that $(F, \eta_F, \nabla_F)$ and $(G, \eta_G, \nabla_G)$ are lax symmetric monoidal endofunctors on $\C$. Then so is $F \otimes G$, by which we mean the functor that takes $X \mapsto FX \otimes GX$. The map $\eta$ is the composition $\I \to \I \otimes \I \xrightarrow{\eta_F \otimes \eta_G} (F \otimes G)(\I)$. The map $\nabla$ is the composition $(F \otimes G)(X) \otimes (F \otimes G)(Y) \cong (FX \otimes FY) \otimes (GX \otimes GY) \xrightarrow{\nabla_F \otimes \nabla_G} (F \otimes G)  (X \otimes Y)$. If $F$ and $G$ are accessible, then so is $F \otimes G$ (because $\otimes$ preserves colimits in each variable, as $\C$ is closed). 
           \item[$(F + G)$:] Suppose $\C$ is closed and that $(F, \eta_F, \nabla_F)$ is a lax symmetric monoidal endofunctor, and consider a pair $(G, \nabla_G)$ of an endofunctor $G\colon \C \to \C$ and transformation $GX \otimes GY \to G(X \otimes Y)$ which is associative and symmetric (i.e.~a lax symmetric monoidal endofunctor but without $\eta$), and suppose that $G$ is a \emph{module} over $F$, meaning that there is a natural transformation $GX \otimes FY \to G(X \otimes Y)$ satisfying the diagrams of \cite[Def.~39]{yetter}, where the concept of module over a monoidal functor is defined. Then $F + G$ is a lax symmetric monoidal functor, where $F+ G$ takes $X$ to $FX + GX$. We set $\eta$ to be the composition $\I \xrightarrow{\eta_F} F (\I) \rightarrow F(\I) + G(\I) \cong (F + G)(\I)$. The map $\nabla$ is given by the following composition
           \begin{align*}
           \hspace{-1em}
               & (F + G)(X) \otimes (F + G)(Y)  \\
               \cong \ &(FX \otimes FY) + (FX \otimes GY) + (GX \otimes FY) + (GX \otimes GY) \\
                \xrightarrow{\langle \nabla_F , \lambda, \rho, \nabla_G \rangle} \ &F(X \otimes Y) + G(X \otimes Y) 
                \\ \cong \ &(F+G)(X \otimes Y).
           \end{align*} 
           If $F$ and $G$ preserve filtered colimits, then so does $F + G$.
       \item[$(\id^A)$:] Suppose that $\C$ is cartesian (so we denote it $(\C,1 ,\times)$). Consider an exponentiable object $A$ of $\C$, meaning that the functor $\id \times A$ has a right adjoint, denoted $\id^A$. This has a lax symmetric monoidal structure. We let $1 \to 1^A$ and $X^A \times Y^A \to (X \times Y)^A$ be the isomorphisms we get from the fact that $-^A$, as a right adjoint, preserves limits. It is accessible as a right adjoint.
       \item[($W$-types):] Consider the cartesian monoidal structure on $\Set$. Consider a commutative monoid $C \in \Set$, which can be considered as a commutative monoidal discrete category, and a function $f\colon A \to C$ such that the preimage functor $f^{-1}\colon C \to \Set$ has an oplax symmetric monoidal structure. Then the polynomial endofunctor (\cref{def: poly endofunctor and w-type}) $\mathcal P_f \colon \Set \to \Set$ has a lax symmetric monoidal structure. The elements of $\mathcal P_f X$ are pairs $(c \in C, g \colon f^{-1} c \to X)$.  The natural transformation $\nabla_{X,Y} \colon \mathcal P_f X \times \mathcal P_f Y \to \mathcal P_f(X \times Y)$ takes a pair $((c, g \colon f^{-1} c \to X), (d, h \colon f^{-1} d \to X))$ to $(cd, (g \times h) \circ \Lambda_{c,d}\colon f^{-1} (cd) \to X \times Y)$ where $\Lambda_{c,d}\colon f^{-1} (cd) \to f^{-1} c \times f^{-1}d$ is part of the oplax monoidal structure of $f^{-1}$. The natural transformation $\epsilon\colon 1 \to \mathcal P_f 1 \cong C$ is the unit of the monoidal structure on $C$. This polynomial endofunctor is accessible since it is the composition of left and right adjoints.
       \item[(Discrete equational systems):] Suppose that the monoidal structure on $\C$ is cocartesian, so we denote it by $(\C,0, +)$, and suppose also that $\C$ also has binary products that preserve filtered colimits in each variable. Consider a \emph{discrete equational system} \cite{leinster} on $\C$, that is, a pair $(A,M)$ where $A$ is a set and $M$ is an object of $\C^{A \times A}$. The functor category $\C^A$ inherits a cocartesian monoidal structure $(\C^A ,0, +)$ pointwise. Define an endofunctor $M \otimes - \colon \C^A \to \C^A$ by setting
       \[ (M \otimes X)_a \coloneqq\sum_{b \in A} M_{b,a} \times X_b. \]
       This is lax symmetric monoidal. The morphisms $\eta\colon 0 \to M \otimes 0$ and $\nabla\colon M \otimes (X + Y) \to M \otimes X + M \otimes Y$ are induced by the universal properties of $0$ and $+$. It is also accessible since the coproducts and binary products preserve filtered colimits.
   \end{description}
   \end{example}

If $\C$ is locally presentable and $F$ is accessible, then given a coalgebra $(C,\chi)$ and algebras $(A,\alpha)$ and $(B, \beta)$, a map $\widetilde{\phi}\colon C\rightarrow [C,A]$ in $\C$ uniquely determines a coalgebra homomorphism $\overline{\varphi}\colon C \rightarrow \cofree([A,B])$. A map $\widetilde{\phi}\colon C\rightarrow [A,B]$ is a $C$-indexed partial homomorphism if and only if both composites from $(C, \chi_C)$ to $\cofree\big( \inthomp{\C}{FA}{B}\big)$ in the following diagram coincide:
\[
\begin{tikzcd}
    (C, \chi_C) \ar{r}{\overline{\varphi}} & \cofree\big( [A,B]\big) \ar[shift left]{r}{\Phi} \ar[shift right]{r}[swap]{\Psi} & \cofree\big( [FA, B]\big).
\end{tikzcd}
\]
The morphism $\Phi$ is the cofree coalgebra homomorphism induced by the map $\alpha^*\colon[A,B]\rightarrow [FA, B]$ in $\C$.
The morphism $\Psi$ is the adjunct under the cofree-forgetful adjunction of the following composite:
\[
\begin{tikzpicture}[baseline= (a).base]
\node[scale=0.99] (a) at (1,1){
\begin{tikzcd}
    [column sep=scriptsize]
   \cofree([A,B]) \ar{r}{\chi_\cofree} & [-0.5em] F \big(  \cofree([A,B]) \big) \ar{r}{F(\varepsilon)} & [-0.5em] F \big( [A,B]\big) \ar{r}{\widetilde{\nabla}_{A,B}} & {[FA, FB]} \ar{r}{\beta_*} & [-1em] {[FA, B]}.
\end{tikzcd}
};  
\end{tikzpicture}
\]
Here $\chi_\cofree$ is the coalgebraic structure on the cofree coalgebra, and $\varepsilon$ is the counit of the cofree-forgetful adjunction.

Now we can use this to guarantee the existence of a universal measuring.

\begin{theorem}
    \label{theorem: equalizer formula for universal measuring}
Suppose that $\C$ is a symmetric monoidal closed category that is locally presentable, and that $F\colon \C\rightarrow \C$ is a symmetric lax monoidal, accessible endofunctor.
Given $F$-algebras $A$ and $B$, then the universal measuring coalgebra $(\ubalg(A,B), \widetilde{\ev})$ exists and is determined as the following equalizer diagram in $\bcoalg$:
\[
\begin{tikzcd}
   \ubalg(A,B) \ar[dashed]{r}{\overline{\ev}} & \cofree\big( [A,B]\big) \ar[shift left]{r}{\Phi} \ar[shift right]{r}[swap]{\Psi} & \cofree\big( [FA, B]\big).
\end{tikzcd}
\]
\end{theorem}

\begin{proof}
Since every $C$-indexed partial homomorphism $\widetilde{\varphi}\colon C\rightarrow [A, B]$ is equivalent to a coalgebra homomorphism $C\rightarrow \cofree([A,B])$ that coequalizes the parallel morphisms described above, we see that the category $\cM(A,B)$ of measurings from $A$ to $B$ is equivalent to the category of cones over the aforementioned parallel morphisms.
Thus they have equivalent terminal objects $\ubalg(A,B)$.
    \end{proof}

\begin{definition}
    [dual (co)algebra]
    \label{ex: dual}
    Suppose that $\C$ is a symmetric monoidal closed category that is locally presentable, and that $F\colon \C\rightarrow \C$ is a symmetric lax monoidal, accessible endofunctor.
    The category $\balg$ has an initial object which we denote by $N$.
    
    Let $(-)^*\colon\bcoalg^\op\rightarrow \balg$ denote the functor $[-,N]$, and call $C^*$ the \emph{dual algebra of $C$} for any coalgebra $C$.

    Let $(-)^\circ\colon\balg^\op\rightarrow \bcoalg$ denote the functor $\ubalg(-,N)$, and call $A^\circ$ the \emph{dual coalgebra of $A$} for any algebra $A$.
    
    These functors form a dual adjunction in the sense that 
we obtain natural identifications for all $C\in \coAlg$ and $A\in \Alg$:
    \[
    \balg(A, C^*)\cong \bcoalg(C, A^\circ).
    \]
\end{definition}

\subsection{The Measuring Tensor}
We now provide the last representation of the measuring functor. 

\begin{definition}[measuring tensor]
Let $C$ be an $F$-coalgebra and $A$ an $F$-algebra. 
We define $\cM_C(A)$ to be the category whose objects consist of pairs $(B, \varphi)$ where $B$ is an $F$-algebra, and $\varphi\colon C\otimes A\rightarrow B$ is a $C$-measuring from $A$ to $B$.
A morphism $(B_1, \varphi)\rightarrow (B_2, \psi)$ in $\cM_C(A)$ consists of an algebra homomorphism $B_1\rightarrow B_2$ fitting in the commutative diagram in $\C$:
\[
\begin{tikzcd}
    & C\otimes A \ar{dl}[swap]{\varphi} \ar{dr}{\psi}\\
    B_1 \ar{rr} & & B_2.
\end{tikzcd}
\]
If it exists, we denote by $(C\triangleright A, u)$ the initial object of $\cM_C(A)$. The $F$-algebra $\gls{MeasTens}$ is called the \emph{measuring tensor of $C$ with $A$},\index{measuring tensor!over an endofunctor} and its measuring $u\colon C\otimes A\rightarrow C\triangleright A$ is the universal $C$-measuring from $A$.
In other words, given any $C$-measuring $\varphi\colon C\otimes A\rightarrow B$, there exists a unique algebra homomorphism $!\colon C\triangleright A\rightarrow B$ such that the following diagram commutes in $\C$:
\[
\begin{tikzcd}
    C\otimes A \ar{r}{\varphi} \ar{d}[swap]{u} & B\\
    C\triangleright A. \ar[dashed]{ur}[swap]{!}
\end{tikzcd}
\]
The construction is natural in $C$ and $A$ and thus determines a functor
\[
-\triangleright -\colon \coAlg\times  \Alg \longrightarrow \Alg.
\]
\end{definition}

The existence of the measuring tensor determines a natural identification
\[
\mu_C(A,B)\cong \balg(C\triangleright A, B).
\]
That is, the functor $\mu_C(A,-)\colon \balg \to \Set$ is represented by $C \triangleright A$.

We can prove the existence of the measuring tensor $C\triangleright A$ using a dual argument to the one used to prove the existence of the universal measuring coalgebra $\ubalg(A,B)$.
If $\C$ is locally presentable and $F$ is accessible, then for a coalgebra $(C,\chi)$, and algebras $(A,\alpha)$ and $(B,\beta)$, a map $\phi\colon C\otimes A\rightarrow B$ uniquely determines an algebra homomorphism $\phi'\colon\free(C\otimes A)\rightarrow (B,\beta)$.
Notice then that a map $\phi\colon C\otimes A\rightarrow B$ is a measuring if and only if both composites from $\free(C\otimes FA)$ to $(B, \beta)$ coincide in the following diagram: 
\[
\begin{tikzcd}
   \free(C\otimes FA) \ar[shift left]{r}{\Phi} \ar[shift right]{r}[swap]{\Psi} & \free(C\otimes A) \ar{r}{\phi'} & (B, \beta).
\end{tikzcd}
\]
The morphism $\Phi$ is the free map induced by $\id\otimes \alpha\colon C\otimes FA\rightarrow C\otimes A$.
The morphism $\Psi$ is the adjunct under the free-forgetful adjunction of the composition
\[
\begin{tikzcd}[column sep=scriptsize]
   C\otimes FA \ar{r}{\chi\otimes \id}& FC\otimes FA \ar{r}{\nabla_{C,A}} & F(C\otimes A) \ar{r}{F(i)} & F(\free(C\otimes A)) \ar{r}{\alpha_\free} & \free(C\otimes A),
\end{tikzcd}
\]
in which $i$ is the unit of the free-forgetful adjunction and $\alpha_\free$ is the algebra structure on the free algebra $\free(C\otimes A)$. 
Then we see that the category $\cM_C(A)$ is equivalent to the category of cocones over the parallel morphisms described above. 
We have now shown the following.

\begin{theorem}
    \label{thm: copower}
Suppose that $\C$ symmetric monoidal closed category that is locally presentable, and that $F\colon \C\rightarrow \C$ is an accessible endofunctor.
Consider an $F$-coalgebra $C$ and an $F$-algebra $A$.
The measuring tensor $(C\triangleright A, u)$ exists and is determined as the following coequalizer in $\balg$:
\[
\begin{tikzcd}
   \free(C\otimes FA)\ar[shift left]{r}{\Phi} \ar[shift right]{r}[swap]{\Psi} & \free(C\otimes A) \ar[dashed]{r}{u'} & C\triangleright A.
\end{tikzcd}
\]
\end{theorem}

\begin{remark}\label{remark: universal measuring as an adjoint}
The convolution algebra also provides an alternative characterization of the algebra $C\triangleright A$ and coalgebra $\ubalg(A,B)$.
As limits in $\balg$ and colimits in $\bcoalg$ are determined in $\C$ \cite{varieties} and the internal hom $[-,-]\colon\C^\op\times \C\rightarrow \C$ preserves limits, the functor $[-,-]\colon\bcoalg^\op\times \balg\rightarrow \balg$ also preserves limits. 
Moreover, fixing a coalgebra $C$, the induced functor $[C, -]\colon\balg\rightarrow \balg$ is accessible since filtered colimits in $\balg$ are computed in $\C$ (see \cite[5.6]{varieties}). 
Therefore, by the adjoint functor theorem \cite[1.66]{presentable}, the functor $[C,-]$ is a right adjoint. Its left adjoint is precisely $C\triangleright -\colon\balg\rightarrow \balg$. Indeed, for any algebras $A$ and $B$, we obtain the following bijection, natural in $C,A,B$.
\[
\balg \big( C\triangleright A, B\big) \cong \balg \big(A, [C,B]\big).
\]
Notice we can also determine the universal measuring coalgebra by using the adjoint functors.
Fixing now an algebra $B$, the opposite functor $[-,B]^\op\colon\bcoalg\rightarrow \balg^\op$ preserves colimits, where the domain is locally presentable and the codomain is essentially locally small. By the adjoint functor theorem \cite[1.66]{presentable}, this functor is a left adjoint. 
Its right adjoint is precisely the functor $\ubalg(-, B)\colon\balg^\op\rightarrow \bcoalg$. Indeed, for any algebra $A$ and $B$ and any coalgebra $C$, we have the following bijection, natural in $C,A,B$:
\begin{align*}
\bcoalg\big( C, \ubalg(A,B) \big) \cong \balg \big( A, [C,B]\big).
\end{align*}
\end{remark}

Combining the identifications, we see that the measuring functor is representable in each factor:
\begin{align}\label{eq: universal prop of measuring}
    \mu_C(A,B) \cong \bcoalg\big( C, \ubalg(A,B) \big)\cong \balg \big( A, [C,B]\big) \cong \balg \big( C\triangleright A, B\big).
\end{align}
In other words, for any algebra $A$ and $B$ and any coalgebra $C$, the following data are equivalent.
\vspace{0.2em}
\begin{center}
    \begin{tabular}{|c|c|c|c|c|}
    \hline  $C\otimes A\rightarrow B$ &  $C\rightarrow [A,B]$ & $C\rightarrow \ubalg(A,B)$ & $A\rightarrow [C,B]$ &  $C\triangleright A\rightarrow B$ \\
      measuring \vspace{-0.08em} & partial \vspace{-0.08em}& coalgebra  \vspace{-0.08em}& algebra\vspace{-0.08em} & algebra \vspace{-0.08em} \\ 
        &homomorphism  & homomorphism & homomorphism & homomorphism \\\hline 
    \end{tabular}
\end{center}
\vspace{0.1em}

\subsection{Interlude: Calculations Involving Preinitial Algebras and Subterminal Coalgebras}

In this section, we take a brief interlude to develop some theory and examples around preinitial algebras and subterminal coalgebras. These provide good examples with which to work out the theory: calculations involving them are both feasible and illuminating.

In some sense, the initial algebra is the `nicest' algebra, and the terminal coalgebra is the `nicest' coalgebra. But these are a bit too nice and so do not reveal much about the theory. Here, we take one step further into the world of (co)algebras by considering \emph{quotients} of the initial algebra, and \emph{subobjects} of the terminal coalgebra.

The reader who wants to read about the general theory before embarking on understanding these examples can skip to the next section, \cref{sec: Measuring as an enrichment}.

\begin{definition}
    [preinitial algebra, subterminal coalgebra]
    Say that an algebra $A$ is \emph{preinitial} \index{preinitial algebra} if the universal map $I \to A$ from the initial algebra $I$ is an epimorphism (in $\balg$). Dually, a coalgebra $C$ is \emph{subterminal} \index{subterminal coalgebra} if the universal map $C \to T$ to the terminal coalgebra $C$ is a monomorphism (in $\bcoalg$).
\end{definition}

Now we establish recognition theorems for preinitial algebras and subterminal coalgebras.

    \begin{proposition}\label{lem: quotient algebras}
        Consider a category $\C$, an endofunctor $F\colon \C \to \C$, and an algebra $(A,\alpha)$ of $F$.
        \begin{enumerate}
            \item Given a quotient object $Q$ of $A$ for which $\alpha$ restricts to a map $\alpha_Q \colon FQ \to Q$, then the pair $(Q, \alpha_Q)$ is a quotient object of $(A, \alpha)$.
            \item If the forgetful functor $\Alg \to \C$ preserves epimorphisms, then for every quotient object $(Q, \alpha_Q)$ of $(A, \alpha)$, $Q$ is a quotient object of $A$.
        \end{enumerate}
        Thus, if the forgetful functor preserves epimorphisms, it induces an isomorphism between the poset of quotient algebras of $(A,\alpha_A)$ and the poset of quotient objects $Q$ of $A$ for which $\alpha$ restricts to a morphism $\alpha_Q \colon FQ \to Q$.
    \end{proposition}

    \begin{proof}
        For (1), asking that $\alpha$ restricts to a map $\alpha_Q$ is exactly asking that $(Q, \alpha_Q)$ is an algebra of $F$ and that the quotient map constitutes a total algebra homomorphism $q\colon (A, \alpha) \to (Q, \alpha_Q)$. 
        Since the forgetful functor $\Alg \to \C$ is faithful, it reflects epimorphisms. The second statement (2) is immediate.
    \end{proof}

    \begin{example}\label{ex: quotient algebras for id + 1}
        Consider the endofunctor $\id + 1$ on $\Set$. Since the forgetful functor $\Alg \to \C$ preserves epimorphisms, there is an isomorphism between the poset of quotient algebras of an algebra $(A,\alpha)$ and the poset of quotient sets $Q$ of $A$ for which $\alpha$ restricts to a morphism $\alpha_Q \colon FQ \to Q$.
    Thus, nontrivial preinitial algebras include those of the form
$\standalg{n} \coloneqq (\{0,1,...,n\}, \alpha_\standalg{n})$ for any $n \in \N$ where $\alpha_\standalg{n}$ is the algebra structure that $\{0,1,...,n\}$ inherits as the quotient of $\N$ in $\Set$ that identifies all $m \geq n$.
    \end{example}

    \begin{definition}
        [index]
        Consider a category $\C$ with a terminal object $\term$, an endofunctor $F\colon \C \to \C$, and suppose that the category of coalgebras of $F$ has a terminal coalgebra $T$.
        \begin{itemize}
            \item  Call a point $\term \to T$ an \emph{index}. For any index $n$, let $\bcoalg^n$ denote the category whose objects are pairs $((X, \chi_X), x)$ where $(X, \chi_X)$ is a coalgebra and $x\colon \term \to X$ such that $\ind{x} = n$, using the notation of \cref{example: terminal coalgebra of Nat}.
            \item Say that a coalgebra is \emph{freely generated by a point of index $n$} if it is an initial object in $\bcoalg^n$. Say that a coalgebra is \emph{generated by a point of index $n$} if it is the vertex of a cone on the identity endofunctor on $\bcoalg^n$.
        \end{itemize}   
    \end{definition}

    \begin{definition}[well-equipped]
        Say that $F \colon\C \to \C $ is \emph{well-equipped} \index{well-equipped endofunctor} if $\C$ is well-pointed and for every index $n$, there is a coalgebra generated by a point of index $n$.
    \end{definition}

    \begin{proposition}\label{lem:char-sub}
        Consider a category $\C$, an endofunctor $F$, and a coalgebra $(C, \chi)$.
        \begin{enumerate}
            \item Given a subobject $S$ of $C$ for which $\chi$ restricts to a morphism $\chi_S\colon S \to FS$, the pair $(S, \chi_S)$ is a subobject of $(C, \chi)$.
            \item If $F$ is well-equipped, then for every subcoalgebra $(S, \chi_S)$ of $(C, \chi)$, we have that $S$ is a subobject of $C$.
        \end{enumerate}
        Thus, if $F$ is well-equipped, there is an isomorphism between the poset of subobjects of $(C, \chi_C)$ and the poset of subobjects of $C$ for which $\chi_C$ restricts to a morphism $\chi_S\colon FS \to S$.
    \end{proposition}
    \begin{proof}
    For (1), consider $\iota \colon S \rightarrowtail C$. Asking that $\chi$ restricts to a morphism $\chi_S$ means exactly that $\iota$ constitutes a total algebra homomorphism $\iota \colon (S, \chi_C) \to (C, \chi)$. This is a monomorphism since the forgetful functor $U\colon \Alg \to \C$ is faithful and thus reflects monomorphisms.

    For (2), consider a monomorphism $\iota\colon(S, \chi_S) \rightarrowtail (C, \chi)$. In order to show that $\iota \colon S \to C$ is also a monomorphism, consider two points $x, y\colon\term \to S$ in $\C$ such that $\iota x = \iota y$. Since $\C$ is assumed to be well-pointed, it suffices to show that $x = y$. Note that $\ind{x} = \ind{\iota x} = \ind{\iota y} =  \ind{y}$, and let $((I, \chi_I), \top_n)$ denote a coalgebra generated by a point of this index. Then we get two maps $\overline x, \overline y\colon(I, \chi_I) \to (S, \chi_S)$ and a map $\overline{\iota x} = \overline{\iota y} \colon (I, \chi_I) \to (C, \chi)$ such that $\iota \overline x = \overline{\iota x} = \overline{\iota y} = \iota \overline y$. Since $\iota$ is a monomorphism, $\overline x = \overline y$, but then $x = \overline x \top_n = \overline y \top_n = y\colon * \to \C$.
    \end{proof}

    \begin{example}\label{ex: subcoalgebras for id + 1}
        Consider the endofunctor $\id + 1$ on $\Set$. The nontrivial subsets of $\eN$ that are closed under the coalgebra morphism are $\{0,1,...,n\}$, $\{\infty\}$, $\N$ (and unions of these); let $\gls{subterm1}, \I, \gls{subterm2}$ denote the corresponding subterminal coalgebras. By the preceding proposition, these are the generating sub-coalgebras of $\eN$.
        The indices are the elements of $\eN$. Each coalgebra $\standcoalg{n}$ is freely generated by a point of index $n$ for finite $n$, and $\I$ is freely generated by a point of index $\infty$.
    \end{example}

Now, we can use preinitial algebras and terminal coalgebras of the endofunctor $1 + \id$ on $\Set$ to calculate some measurings.

\begin{example}
    \label{ex: partial induction}
    Consider algebras $A$ and $B$ of the endofunctor $1 + \id$ on $\Set$.
    Suppose that $A$ is preinitial, so that in particular every element of $A$ is of the form $0_A$ or $a + 1$.
    If we want to construct an algebra homomorphism $f\colon A \to B$, we can proceed by induction. First, we send $0_A \mapsto 0_B$. Next, we can try to send $1_A \mapsto 1_B$ (where $-_A$ denotes the algebra morphism $\N \to A$). Note that this only works if $1_A$ is different from $0_A$ or if $0_B = 1_B$. If this works, we can try to continue inductively, sending $n_A \mapsto n_B$. If this procedure succeeds, then we will have constructed an algebra homomorphism $f\colon A \to B$ by induction.
    
    We can turn this construction into a sequence of functions that we define in the following way.
    \begin{description}
        \item[{Initial step (P1)}:] Define $f_0\colon A\rightarrow B$ by $f_0(a)\coloneqq 0_B$ for all $a\in A$. 
        \item[{Inductive step}:] Define $f_{c+1} \colon A\rightarrow B$ by:
        \begin{description}
            \item[{(P2)}:] $f_{c+1}(0_A)\coloneqq 0_B$;
            \item[{(P3)}:] $f_{c+1}(a+1)\coloneqq f_{c}(a)+1$.
        \end{description}
    \end{description}
    If we have defined $f_c$ for all $c \in \N$, then we will say that we have defined an \emph{$\infty$-partial homomorphism}. Note that in this case, one can then form a total algebra homomorphism which we will denote $f_\infty$.
Otherwise, if we have only defined $f_c$ for all $c \in \{0, ..., n\}$, we will say that we have defined an \emph{$n$-partial homomorphism}.
    Now compare this induction with the definition of measuring (\cref{ex: measuring for - + 1}).
    There is a measuring from $A$ to $B$ by $\standcoalg{n}$ if the induction above creates an $n$-partial homomorphism, and in this case the functions of the form $f_c$ constructed are the same as those specified in \cref{ex: measuring for - + 1}.

    If we denote $\Alg_n(A,B)\coloneqq \mu_{\n^\circ}(A,B)\cong \coAlg(\n^\circ, \ubalg(A,B))$ the set of all $\n^\circ$-measurings from $A$ to $B$, the natural inclusions $\n^\circ\hookrightarrow (\n+1)^\circ$ are homomorphisms of coalgebras, so after applying the functor $\coAlg(-, \ubalg(A,B))\colon \coAlg^\op\rightarrow \Set$, we obtain a tower
     \[
\begin{tikzcd}[column sep=small]
    \Alg_\infty(A,B)\ar{r} & \cdots \ar{r} &
    \Alg_2(A,B) \ar{r}&
    \Alg_1(A,B) \ar{r}&
    \Alg_0(A,B),
\end{tikzcd}    
\] 
whose limit is $\Alg_\infty(A,B)\coloneqq \coAlg(\coaN, \Alg(A,B))$ as the right adjoint functor $\coAlg(-, \ubalg(A,B))$ sends colimit in $\coAlg$ to a limit in $\Set$,  and the colimit the coalgebra homomorphisms $\n^{\circ}\hookrightarrow (\n+1)^\circ$ is precisely the coalgebra $\coaN$.
Our construction above of $f_\infty$ precisely builds, layer by layer in the tower, an element of $\Alg_\infty(A,B)$.
    
    There is a measuring by $\coaN$ if the induction never fails, and again the functions $f_c$ from the induction above and \cref{ex: measuring for - + 1} coincide. Now, note that exhibiting a measuring by $\eN$ amounts to exhibiting a measuring by $\coaN$ together with a total algebra homomorphism $f_\infty$. For such an $A$, then, exhibiting a measuring by $\coaN$ is (logically) equivalent to exhibiting one by $\eN$.
\end{example}

Now, we show that the universal measuring coalgebra from $A$ to $B$ is subterminal and easy to calculate if $A$ is preinitial.

\begin{proposition}\label{lem:subobject}
    Consider two algebras $A,B$ such that $A$ is preinitial. Then $\ubalg(A,B)$ is a subterminal coalgebra, and it is the maximum subterminal coalgebra $C$ such that $\balg(A, [C,B])$ is nonempty.
\end{proposition}

\begin{proof}
    We first claim that $\balg(A, [C,B])$ is empty or a singleton for any coalgebra $C$.
    Since $!_A\colon \N \to A$ is an epimorphism, the function
    \[\balg(A, [C,B]) \xrightarrow{!^*_A}  \balg(\N, [C,B]) \]
    is an injection, and since $\N$ is initial, $\balg(\N, [C,B]) \cong *$.
    Thus, $\balg(A, [C,B])$ is a singleton, in the case where the only map $\N \to [C,B]$ factors through $A$, and otherwise is empty.

    We now claim that the canonical map $\ind{-}\colon\ubalg(A,B) \to \eN$ is a monomorphism. Suppose there are maps $a,b\colon C \to \ubalg(A,B)$ from a coalgebra $C$ such that $\ind{a} = \ind{b}$. Recall that the coalgebra $\ubalg(A,B)$ is defined by the following universal property for any coalgebra $C$
    \[ \bcoalg(C,\ubalg(A,B)) \cong \balg(A, [C,B]).\]
    Thus, $\bcoalg(C,\ubalg(A,B))$ is a singleton, so $a = b$, and $\ind{-}$ is monic.
\end{proof}

\begin{example}\label{ex:universal measuring for preinitial}
    Again, consider algebras $A$ and $B$ of the endofunctor $1 + \id$ on $\Set$ where $A$ is preinitial.
    In this case, the universal measuring is a subterminal coalgebra by the above. If the induction of \cref{ex: partial induction} succeeds in constructing $f_i$ for $i \leq n$ but fails to construct $f_{n+1}$, then the maximum subterminal coalgebra that measures $A \to B$ is $\standcoalg{n}$, so this is the universal measuring. And if the induction of \cref{ex: partial induction} creates an $\infty$-partial homomorphism, i.e., a total homomorphism, then the maximum subterminal coalgebra that measures $A$ to $B$ is $\eN$ itself, so this is the universal measuring.
    We will also show this fact more directly (i.e., without reference to \cref{lem:subobject}) below.
\end{example}

\begin{example}
    \label{ex: compute partial alg hom}
We compute $\ubalg(\standalg{n},B)$ for any algebra $B$ of $1 + \id$ using \cref{eq: universal prop of measuring}.
We first observe the following for any algebra $Z$:
\[ \balg(\standalg{n} , Z) \cong 
\begin{cases} 
    * & \text{if } n_{Z} = (n+1)_{Z} \\ \emptyset & \text{otherwise}.
\end{cases}
\] 
Since we are considering $Z\coloneqq [C,B]$, we need to understand when $n_{[C,B]} = (n+1)_{[C,B]}$. By definition, $0_{[C,A]}$ is the constant function at $0_A$. Then $1_{[C,A]}$ is the function that takes every $c \in C$ of index $0$ to $0_B$, and every other $c \in C$ to $1_B$. Inductively, we can show that $n_{[C,B]}(c) = \min(\llbracket c \rrbracket , n)_{B}$.
Thus, $n_{[C,B]} = (n+1)_{[C,B]}$ means that $\min(\llbracket c \rrbracket , n)_{B} = \min(\llbracket c \rrbracket , n + 1)_{B}$ for all $c \in C$, and this holds if and only if $\llbracket c \rrbracket \leq n$ for all $c \in C$ or $n_B = (n+1)_B$. Now we have the following:
\begin{align}\label{eq: conv alg calc}
    \bcoalg(C ,  \ubalg(\standalg{n},B)) \cong \balg(\standalg{n} , [C,B]) \cong
\begin{cases} 
    * & \text{if } \llbracket c \rrbracket \leq n \text{ for all } c \in C \\ 
    * & \text{if } n_B = (n+1)_B \\
    \emptyset & \text{otherwise.} 
\end{cases}
\end{align}
In the case that $n_B = (n+1)_B$, we find that $\ubalg(\standalg{n},B)$ has the universal property of the terminal object, $\eN$.
Now suppose that $n_B \neq (n+1)_B$. Since $\bcoalg(C, \standcoalg{n}) = *$ if and only if $\llbracket c \rrbracket \leq n$ for all $c \in C$, $\ubalg(\standalg{n},B)$ has the universal property of $\standcoalg{n}$. 

Now we have calculated the following 
\[ \ubalg(\standalg{n},B) = \begin{cases}
    \eN \text{ if } n_B = (n+1)_B \\
    \standcoalg{n} \text{ otherwise.}
\end{cases} \]
This aligns with our expectations, since there is a total homomorphism $\standalg{n} \to B$ if $n_B = (n+1)_B$ but there is only an $n$-partial homomorphism $\standalg{n} \to B$ otherwise.
Finally, note that taking $B \coloneqq\N$, we have calculated $\ubalg(\standalg{n},\N)$, the \emph{dual} (\cref{ex: dual}) of $\standalg{n}$, to be $\standcoalg{n}$.
\end{example}

\subsection{Measuring as an Enrichment}
\label{sec: Measuring as an enrichment}

We now come to the main punchline of the general theory presented in this paper: that $\ubalg(-,-)$ gives the category of algebras an enrichment in coalgebras. First, we describe how to compose measurings.

\begin{proposition}\label{prop:sym-mon}
    Let $\C$ be a (symmetric) monoidal category.
    Let $F\colon \C\rightarrow \C$ be a lax (symmetric) monoidal functor.
    Then the category $\coAlg$ has a (symmetric) monoidal structure for which the forgetful functor $\coAlg\rightarrow \C$ is strong (symmetric) monoidal.
    Moreover, if $\C$ is locally presentable and closed, and $F$ is accessible, then $\coAlg$ is also closed, with internal hom denoted $\ubcoalg(-,-)$.
\end{proposition}

\begin{proof}
    Suppose $(C, \chi_C)$ and $(D, \chi_D)$ are coalgebras. Then $C\otimes D$ has the following coalgebra structure:
    \[
\begin{tikzcd}
{C \otimes D} \ar{r}{\chi_C \otimes \chi_D} & [1em] {F(C) \otimes F(D)} \ar{r}{\nabla_{C,D}} & {F(C\otimes D).}
\end{tikzcd}
\]
The morphism $\eta\colon \I\rightarrow F(\I)$ provides the coalgebraic structure on $\I$.

To see that $(\bcoalg, \otimes, (\I, \eta))$ is a symmetric monoidal category, we first need to check that $(C, \chi_C)\otimes (\I, \eta)\cong (C, \chi_C)$. On objects, we have $C\otimes \I\cong C$ as $\I$ is a monoidal unit. The coalgebraic structures correspond as we have the following commutative diagram by unitality of $(F, \nabla, \eta)$:
\[
\begin{tikzcd}
   C \ar{r}{\cong}\ar{d}[swap]{\chi} & C\otimes \I \ar{dr}{\chi\otimes \eta} \\
   F(C) \ar[bend right=2em, equals]{rrrr} \ar{r}{\cong} & F(C)\otimes \I \ar{r}[swap]{\id\otimes \eta}& F(C)\otimes F(\I) \ar{r}{\nabla} & F(C\otimes \I)\ar{r}{\cong}  & F(C).
\end{tikzcd}
\]
Similarly, we have a natural isomorphism $ (\I, \eta)\otimes (C, \chi_C)\cong (C, \chi_C)$ from the identification $\I\otimes C\cong C$. 
Moreover, we also have the following natural isomorphism
\[
((C, \chi)\otimes (C', \chi')) \otimes (C'', \chi'') \cong (C, \chi)\otimes ((C', \chi') \otimes (C'', \chi'')),
\]
lifted from $(C\otimes C')\otimes C''\cong C\otimes (C'\otimes C'')$.
Lastly, the symmetry $C\otimes C'\cong C'\otimes C$ lifts to $\bcoalg$. Therefore as $\C$ is symmetric monoidal, the above natural identifications assemble to a symmetric monoidal structure on $\bcoalg$ as the pentagon, triangle, hexagon and symmetric identities remain valid in $\bcoalg$.

Suppose now that $\C$ is locally presentable and closed, and $F$ is accessible. Then for all objects $X\in \C$, we have that $X\otimes -\colon \C\rightarrow \C$ is a left adjoint functor,  with right adjoint given by its the internal hom $[X, -]\colon \C\rightarrow \C$. 
Since $\coAlg$ is also locally presentable and colimits are computed in $\C$, we get that for all $C\in \coAlg$, the induced functor $C\otimes -\colon \coAlg\rightarrow \coAlg$ also preserves all colimits and thus must be a left adjoint as $\coAlg$ is locally presentable. 
We denote by $\ubcoalg(C,-)\colon \coAlg\rightarrow \coAlg$ its right adjoint. The construction is natural in $C$ and thus defines the desired closed structure on $\coAlg$.
\end{proof}

We can build explicitly the internal hom $\ubcoalg(-,-)$ of $\coAlg$ as follows.
Given coalgebras $(C, \chi_C)$ and $(D,\chi_D)$, the coalgebra $\gls{CoAlgHom}$ is the following equalizer in $\coAlg$:
\[
\begin{tikzcd}
    \ubcoalg(C, D) \ar[dashed]{r} & \cofree([C, D]) \ar[shift left]{r}{\Phi} \ar[shift right]{r}[swap]{\Psi} & \cofree([C, FD]).
\end{tikzcd}    
\]
The morphism $\Phi$ is induced by the map $(\chi_D)_*\colon [C, D]\to [C, FD]$ in $\C$.
The morphism $\Psi$ is induced by the forgetful-cofree adjunction on a morphism the following morphism:
\[
\begin{tikzcd}
    [column sep=scriptsize]
 \cofree([C,D]) \ar{r}{\chi_\cofree} & F\cofree([C, D]) \ar{r}{F\varepsilon} & F{[C, D]}  \ar{r}{\widetilde{\nabla}_{C, D}} & {[FC, FD]} \ar{r}{\chi_C^*} & {[C ,FD]}. 
\end{tikzcd}    
\]

Now we can state and prove our main theorem.

\begin{theorem}
    \label{thm:enriched}
Suppose that $\C$ is a locally presentable, symmetric monoidal closed category. Let $F\colon \C \to \C $ be an accessible lax symmetric monoidal endofunctor.
Then the category $\balg$ is enriched, tensored, and cotensored over the closed symmetric monoidal category $\bcoalg$ respectively via
\[
\balg^\op \times \balg\xrightarrow{\ubalg(-,-)} \bcoalg, \quad \bcoalg\times \balg \xrightarrow{-\triangleright-} \balg, \quad \bcoalg^\op\times \balg\xrightarrow{[-,-]} \balg.
\]
In other words, for any coalgebra $C$, and algebras $A$ and $B$, we obtain the natural identifications in $\coAlg$:
\[
  \ubcoalg(C, \ubalg(A,B))\cong \ubalg(C \triangleright A, B) \cong \ubalg(A, [C, B]).  
\]
\end{theorem}

\begin{proof}
    The convolution algebra functor $[-,-]\colon\bcoalg^\op\times \balg\rightarrow \balg$ presents $\balg$ as a right module over $\bcoalg^\op$. 
    Thus $\balg^\op$ is a left module over $\bcoalg$ via \[[-,-]^\op\colon \bcoalg\times \balg^\op \rightarrow \balg^\op.\] 
    Moreover, for any $A\in \balg$, the functor $[-,B]^\op\colon\bcoalg\rightarrow \balg^\op$ 
    has a right adjoint $\ubalg(-, B)\colon\balg^\op\rightarrow \bcoalg$ as shown in \cref{remark: universal measuring as an adjoint}.
    Thus, by \cite[3.1]{hylandetal}, the category $\balg^\op$ is enriched over $\bcoalg$ with tensor $[-,-]^\op$. 
    Therefore $\balg$ is enriched over $\bcoalg$ with power $[-,-]$. 
\end{proof}

\begin{remark}
    While $\Alg$ is enriched in $\coAlg$, we cannot dualize our arguments to show that $\coAlg$ is enriched in $\Alg$.
    This is because if $\C$ is locally presentable, then $\C^\op$ is generally not locally presentable \cite[1.64]{presentable}. 
    Instead, just as in \cref{remark: coalgebras measurings and universal internal hom}, the internal hom in $\coAlg$ provides that  $\coAlg$ is enriched, tensored, and cotensored over itself. Given $F$-coalgebras $(C, \chi)$, $(C_1, \chi_1)$ and $(C_2, \chi_2)$, we say a morphism $\widetilde{\varphi}\colon C\rightarrow [C_1, C_2]$ is a\emph{ $C$-indexed partial homomorphism} if its adjunct $\varphi\colon C\otimes C_1\rightarrow C_2$ is a homomorphism of coalgebras (which can be renamed as a \emph{$C$-measuring from $C_1$ to $C_2$}).
    In other words, this means we obtain the commutative diagram in $\C$
    \[
        \begin{tikzcd}
          & F(C) \ar{r}{F(\widetilde{\phi})} &[1em] F([C_1,C_2]) \ar{r}{\widetilde{\nabla}_{C_1,C_2}} & [1em] {[FC_1, FC_2]} \ar{dd}{\chi_1^*}\\
         C \ar{ur}{\chi} \ar{dr}[swap]{\widetilde{\phi}} \\
         & {[C_1,C_2]} \ar{rr}{(\chi_2)_*} & & {[C_1,FC_2]}.
        \end{tikzcd}
       \]
    Then the internal hom $\ubcoalg(C_1, C_2)$ can be viewed as the universal measuring coalgebra in the sense that given any $C$-indexed partial homomorphism $\widetilde{\varphi}\colon C\rightarrow [C_1, C_2]$ (or equivalently a measuring $\varphi\colon C\otimes C_1\rightarrow C_2$), there exists a unique coalgebra homomorphism $u_\varphi\colon C\rightarrow \ubcoalg(C_1, C_2)$ such that the diagram commutes:
    \[
    \begin{tikzcd}
        C \ar{r}{\widetilde{\varphi}} \ar[dashed]{d}[swap]{\exists !\, u_\varphi} & {[C_1, C_2]}\\
        \ubcoalg(C_1, C_2). \ar{ur}[swap]{\widetilde{\ev}}
    \end{tikzcd}
    \]
    This is making precise the coalgebraic variant of \cref{thm:enriched}: there is a natural identification in $\coAlg$
    \[
    \ubcoalg(C, \ubcoalg(C_1, C_2)) \cong \ubcoalg(C\otimes C_1, C_2)\cong \ubcoalg(C_1, \ubcoalg(C, C_2)).
    \]
    The last identification just signifies there is a symmetry in the coalgebraic case: a $C$-measuring $C\otimes C_1\rightarrow C_2$ can be viewed as a $C_1$-measuring $C_1\otimes C\rightarrow C_2$.
\end{remark}

\begin{example}\label{ex: partial induction part deux}
    We can strengthen \cref{ex: partial induction} as follows. 
     Given algebras $A$ and $B$ over endofunctor $-+1$, denote by $\ubalg_n(A,B)$ the coalgebra $\ubcoalg(\n^\circ, \ubalg(A,B))$.
     The inclusions $\n^\circ \hookrightarrow (\n+1)^\circ$ define a tower in $\coAlg$:
      \[
\begin{tikzcd}[column sep=small]
    \ubalg_\infty(A,B)\ar{r} & \cdots \ar{r} &
    \ubalg_2(A,B) \ar{r}&
    \ubalg_1(A,B) \ar{r}&
    \ubalg_0(A,B).
\end{tikzcd}    
\]
Just as in \cref{ex: partial induction}, we compute that its limit is given by $\ubalg_\infty(A,B)\coloneqq\ubcoalg(\coaN, \ubalg(A,B))$.
The coalgebras $\ubalg_n(A,B)$ should be considered as the universal measuring of $n$-partial homomorphisms, while $\ubalg_\infty(A,B)$ is \emph{the universal measuring of $\infty$-partial homomorphisms}: these are functions $f_\infty\colon A\rightarrow B$ for which $f_\infty((n+1)_A)=f_\infty(n_A)+1$, and thus are equivalent to total algebra homomorphisms if $A$ is preinitial.
\end{example}

\begin{remark}\label{remark: generalized tower}
 The previous example suggests the following generalization.
 Denote by $1$ the terminal object of a locally presentable, symmetric monoidal closed category $\C$.
Let $F$ be an accessible lax symmetric monoidal functor.
The unique map $!\colon F1\rightarrow 1$ in $\C$ determines an algebra structure on $1$.
In fact, $1$ is also the terminal object in $\Alg$. 
The induced map $F!\colon F^21\rightarrow F1$ defines also an algebra structure on $F1$ such that the map $!\colon F1\rightarrow 1$ is homomorphism of $F$-algebras.
More generally, we obtain algebra homomorphisms $F^{n}!\colon F^{n+1}1\rightarrow F^{n}1$ for all $n\geq 0$. 
We have successfully built a tower in $\Alg$:
\[
\begin{tikzcd}
\cdots \ar{r} & F^31\ar{r}{F^2!} \ar{r} & F^2 1 \ar{r}{F!} & F1 \ar{r}{!} & 1,
\end{tikzcd}    
\]
the same one that appeared in Ad\'amek's fixed point theorem (\cref{thm:adamek}), but considered in the category $\Alg$ instead of $\C$.
We now apply the dual coalgebra functor $(-)^\circ\colon\Alg^\op\rightarrow \coAlg$ from \cref{ex: dual}.  
We obtain a filtered diagram in $\coAlg$:
\begin{equation}\label{eq: dual of cofree Adamek}
    \begin{tikzcd}
1^\circ \ar{r}{!^\circ} & (F1)^\circ \ar{r}{F!^\circ} & (F^21)^\circ \ar{r}{F^2!^\circ}& (F^3 1)^\circ \ar{r}  & \cdots.  
    \end{tikzcd}
\end{equation}
Denote by $\Limi$ the colimit of the above diagram in $\coAlg$ (which is always computed in $\C$).
For all $n\geq 0$, given $F$-algebras $A$ and $B$, we define the \emph{$n$-partial measuring coalgebra from $A$ to $B$} to be the coalgebra determined as:
\[
\ubalg_n(A, B)\coloneqq \ubcoalg((F^n 1)^\circ, \ubalg(A, B)).    
\]   
The \emph{$\infty$-partial measuring coalgebra from $A$ to $B$} is the coalgebra defined as
\[
    \ubalg_\infty(A, B)\coloneqq \ubcoalg(\Limi, \ubalg(A, B)).   
\]
If we apply the right adjoint functor $\ubcoalg(-, \ubalg(A, B))\colon \coAlg^\op\to \coAlg$ to the filtered diagram of \cref{eq: dual of cofree Adamek}, we obtain the tower in $\coAlg$:
 \[
\begin{tikzcd}[column sep=small]
    \ubalg_\infty(A,B)\ar{r} & \cdots \ar{r} &
    \ubalg_2(A,B) \ar{r}&
    \ubalg_1(A,B) \ar{r}&
    \ubalg_0(A,B).
\end{tikzcd}    
\]
As right adjoints preserve limits, we indeed obtain that $\ubalg_\infty(A,B)$ is the limit of the tower:
\begin{align*}
    \ubalg_\infty(A,B) & := \ubcoalg(\Limi, \ubalg(A, B))\\
    & = \ubcoalg(\mathrm{colim}_n\, (F^n1)^\circ, \ubalg(A, B)) \\
    & \cong \lim_{n}\ubcoalg((F^n1)^\circ, \ubalg(A,B))\\
    & =\lim_n \ubalg_n(A,B).    
\end{align*}
\end{remark}

On some occasions, the category of coalgebras of $F$ can be interesting while its category of algebras is less so. 
For instance, given an alphabet $\Sigma$, coalgebras over the endofunctor $F(X)=2\times X^\Sigma$ in $\Set$ are automata but the initial algebra remains $\emptyset$. To remedy this, we can extend our main result into the following theorem. 

\begin{theorem}
    \label{thm: most general}
    Let $\C$ be a locally presentably and closed symmetric monoidal category.
    Let $F,G\colon \C\to \C$ be accessible endofunctors.
    Assume that $F$ is lax symmetric monoidal, and that $G$ is an $F$-module, in the sense that there is a transformation
    \[ FX \otimes GY \to G( X \otimes Y) \]
    natural in $X,Y$ and satisfying the diagrams of \cite[Def.~39]{yetter}.
Then the category $\balg_{G}$ is enriched, tensored and cotensored over $\bcoalg_F$.
\end{theorem}

\begin{proof}
    We obtain a convolution algebra $\bcoalg_F^\op \times \balg_{G} \to \balg_{G}$ by taking a coalgebra $(C,\chi)$ and an algebra $(A, \alpha)$ to the following composite
\[
\begin{tikzcd}
    G[C,A] \ar{r} & {[FC, GA]} \ar{r}{\chi^*} \ar{r} & {[C, FA]} \ar{r}{\alpha_*} & {[C,A]},
\end{tikzcd}    
\]
    where
    the unmarked arrow is the adjunct of the following composite
    \[
    \begin{tikzcd} 
    FC\otimes G([C,A]) \ar{r} & G(C\otimes [C,A]) \ar{r}{G(\ev)} & GA
\end{tikzcd}  
    \]
    where $\ev$ denotes evaluation. The action on morphisms is given by the naturality of constituent arrows of this constructed algebra structure. The functoriality is inherited from the functoriality of $[-, -]$. 

    The diagrams of \cite[Def.~39]{yetter} exactly correspond to the opposite 
    \[[-,-]^\op \colon \bcoalg_F \times \balg_{G}^\op \to \balg_{G}^\op,\]
     this functor being a left action in the sense of \cite{hylandetal}.
    Since $[-,A]^\op$ preserves colimits, so does $[-,A]^\op \colon \bcoalg_F \to \balg_{G}^\op$. Since $\bcoalg_F$ is locally presentable, and $\Alg_G^\op$ is essentially small, the functor has a right adjoint.
    Thus, by \cite[Prop.~3.1]{hylandetal}, we see that $\balg_{G}^\op$ is enriched, tensored and cotensored in $\bcoalg_F$. Thus, $\balg_{G}$ is enriched, tensored and cotensored in $\bcoalg_F$.
\end{proof}

\begin{corollary}\label{theorem: mixed enriched GF over F}
    Suppose that $\C$ is a locally presentable, closed symmetric monoidal category. Suppose $F$ is an accessible lax symmetric monoidal functor.
Let $G\colon\C\rightarrow \C$ be a $\C$-enriched functor. 
Then $\balg_{GF}$ is enriched and cotensored over $\bcoalg_F$.
\end{corollary}

\begin{proof}
    To apply previous result, we only need to show there exists an $F$-module structure on $GF$, that is, we need a certain natural map 
    \[ FX \otimes GFY \to GF( X \otimes Y). \]
    The $\C$-enrichment of $G$ produces a map $X \otimes GY \rightarrow G(X\otimes Y)$ for any objects $X$ and $Y$ in $\C$. Indeed, this is the adjunct of
    \[
    \begin{tikzcd}
        X \ar{r} & {[Y, X\otimes Y]} \ar{r} & {[GY, G(X\otimes Y)]},
    \end{tikzcd}    
    \]
    where the first map is the coevaluation of the internal hom, and the second map is induced by the $\C$-enrichment. We apply this map to $FX$ and $FY$ to get the following composition
    \[ FX \otimes GFY \rightarrow G(FX\otimes FY) \xrightarrow{G \nabla} GF (X \otimes Y).\]
    One can check that this makes the diagrams of \cite[Def.~39]{yetter} commute.
\end{proof}

\begin{example}
If $F(X)=2\times X^\Sigma$, we could consider $G=\id+1$, and thus $\balg_{GF}$ has a nonempty initial object and is enriched in automata.
\end{example}

The enrichment of algebras in coalgebras specifies a pairing of coalgebras 
\[
 \ubalg(A_2,A_3)\otimes \ubalg(A_1,A_2) \longrightarrow \ubalg(A_1,A_3),
\]
regarded as an enriched composition, for any algebras $A_1$, $A_2$ and $A_3$.
In more details, the above coalgebra homomorphism is induced by the measuring of
\[
\begin{tikzcd}
  \big(  \ubalg(A_2,A_3)\otimes \ubalg(A_1, A_2)\big) \otimes A_1 \ar{r}{\id\otimes \ev_{A_1,A_2}} & [2em]\ubalg(A_2,A_3)\otimes A_2 \ar{r}{\ev_{A_2,A_3}} & A_3.
\end{tikzcd}
\]
In other words, the enrichment is recording precisely that we can compose a measuring $C\otimes A_1\rightarrow A_2$ with $D\otimes A_2\rightarrow A_3$ to obtain a measuring $(D\otimes C)\otimes A_1\rightarrow A_3$. 

\begin{remark}
\Cref{thm:enriched} and \cref{theorem: mixed enriched GF over F} appeared in \cite{north2023coinductive}. This current paper contextualizes and provides more details to the original work. 
\Cref{thm: most general} is new and is an extension of the previous results, and moreover, the tower appearing in \cref{remark: generalized tower} provides a new formalization of partial induction that did not appear in \cite{north2023coinductive}. 
\Cref{subsect: general theory of c-initial} below provides also new results that did not appear in \cite{north2023coinductive}. 
\end{remark}

\begin{remark}
    Recall that if $F$ is an endofunctor in $\mathcal{C}$, we can define the free monad $\widetilde{F}$ in $\C$ on $F$, and  the (Eilenberg--Moore) category of $\widetilde{F}$-algebras in $\C$ is equivalent to the category of $F$-algebras $\Alg_F$ in $\C$. A similar dual statement exists for coalgebras over endofunctors.
    Algebras over monads are enriched in coalgebras over comonads given some compatibility \cite[Thm.~6]{grignou}, and our results fit into this framework.
    Given a field $k$, then $k$-algebras are algebras over the free $k$-algebra monad and $k$-coalgebras are coalgebras over the cofree $k$-coalgebra comonad, and thus the results of \cref{sec: sweedler coalgebras} also fit in this framework.
\end{remark}

\subsection{General \texorpdfstring{$C$}{C}-Initial Objects}\label{subsect: general theory of c-initial}

Within the theory of categorical W-types, the category of algebras is mostly used to find the initial algebra, the W-type. Here, we explore how the extra structure that we have found in categories of algebras can be used to generalize the notion of initial algebras, inspired by the generalization of (co)limits in categories by the notion of weighted (co)limits in enriched categories.

We can use the extra structure in the enriched category of algebras to specify more algebras than we could in the unenriched category of algebras.

\begin{definition}
    [$C$-initial algebra]
    \label{def: second initial algebra}
    Suppose that $\C$ is locally presentable and closed and that $F$ is accessible.
Given a coalgebra $C$,
    we say an algebra $A$ is a \emph{$C$-initial algebra} \index{$C$-initial algebra} if there exists a unique map $C\rightarrow \ubalg(A,X)$, for all algebras $X$. That is, $\mu_C (A,X) \cong *$.
The \emph{terminal} $C$-initial algebra \index{terminal $C$-initial algebra} is the terminal object, if it exists, in the subcategory of $\balg$ spanned by the $C$-initial algebras.
\end{definition}

\begin{remark}
    One may wonder what would happen if for a \emph{set} $S$, we defined an $S$-initial algebra to be an algebra $A$ such that there is a unique \emph{function} $S \to \balg(A,X)$ for all $X \in \balg$. But every algebra is an $\emptyset$-initial algebra, and an $S$-initial algebra is an initial algebra for any $S \neq \emptyset$ (because functions $S \to T$ are unique only when $S = \emptyset$ or $T \cong 1$). Thus, we need to consider $\ubalg(A,X)$ and not just $\balg(A,X)$ to obtain interesting $C$-initial algebras.
\end{remark}

\begin{example}
    We have shown in \cref{ex: compute partial alg hom} that $\standalg{n}$ is an $\standcoalg{n}$-initial algebra.
 Since $\balg(A,X) \cong \bcoalg(\I, \ubalg(A,X))$, the initial algebra $\N$ is the $\I$-initial algebra. In fact, since $\ubalg(\N,X) \cong \eN$ for all $X$ by \cref{lem:subobject} or by a similar computation to \cref{ex: compute partial alg hom}, we find that $\N$ is a $C$-initial algebra for any subterminal coalgebra (i.e., $\emptyset, \standcoalg{n}, \coaN, \eN$).
\end{example}

\begin{example}
    Since the only $\I$-initial algebra is $\N$, it is also the terminal $\I$-initial algebra.
\end{example}

The following result helps us calculate $C$-initial algebras in examples.

\begin{proposition}
    \label{prop:map from initial alg to dual}
    Suppose that $\C$ is locally presentable and closed and that $F$ is accessible.
    There is a unique map from any $C$-initial algebra to $C^*$.
\end{proposition}

\begin{proof}
    We consider an object $\lim^C\id_\ubalg$ with the following universal property:
\[ \balg(A,{\lim}^C\id_\ubalg) \cong \lim_{X \in \balg} \ \bcoalg (C, \ubalg(A,X)).  \]
(This is the limit of $\id_\ubalg$ weighted by the constant functor at $C$.)
Since the hom $\bcoalg (C, \ubalg(A,-))$ preserves limits and $\lim_{X \in \balg} X \cong \N$, we have
\begin{align*}
    \lim_{X \in \balg} \ \bcoalg (C, \ubalg(A,X)) &\cong \bcoalg (C, \ubalg(A,\lim_{X \in \balg} X)) \\
    &\cong \bcoalg (C, \ubalg(A,\N)) \\
    &\cong \balg(A, [C,\N]) \\
    &= \balg(A, C^*).
\end{align*}
Thus, $C^*$ has the universal property of $\lim^C\id_\ubalg$.

Note that for any $C$-initial algebra $A$, we have $\lim_{X \in \balg} \ \bcoalg (C, \ubalg(A,X))$ is the terminal object. Thus, by this universal property of $C^*$, there is a unique map $A \to C^*$.
\end{proof}

In the following two examples, we calculate that $\standalg{n}$ is the terminal $\standcoalg{n}$-initial algebra.

\begin{example}
    Let $C:= \standcoalg{n}$. Then elements of $(\standcoalg{n})^*$ are sequences of $n+1$ natural numbers.
    The successor of a sequence $(a_i)_{i=0}^n$ is $(b_i)_{i=0}^n$ where $b_0 = 0$ and $b_{i+1} = a_i +1$. Notice that the successor of $(b_i)_{i=0}^n$ is $(c_i)_{i=0}^n$ where $c_0 = 0$, $c_1 = 1$ and otherwise $c_{i+2} = a_i +2$. Thus, we can inductively show that the $(n+1)$-st successor of any element of $[\standcoalg{n}, \N]$ is the sequence $(i)_{i=0}^n$, and the successor of this sequence is itself. 

    We claim that the unique morphism $!_{[\standcoalg{n}, \N]} : \N \to [\standcoalg{n}, \N]$ factors through $\standalg{n}$. We have $m_{[\standcoalg{n}, \N]} = (\min( i,m))_{i=0}^n$. Thus, the restriction of the map $!_{[\standcoalg{n}, \N]}$ to $\{0,...,n\} \subset \N$ is injective, and $n_{[\standcoalg{n}, \N]} = m_{[\standcoalg{n}, \N]}$ for all $m \geq n$.
\end{example}

\begin{example}\label{example: terminal C-initial algebra}
    Now we can show that $\standalg{n}$ is the terminal $\standcoalg{n}$-initial algebra. In this calculation, we use the law of excluded middle for the only time in this paper.
    Consider an $\standcoalg{n}$-initial algebra $A$.

    First, we show that every $a \in A$ is either the basepoint or a successor. So suppose that there is an element $a \in A$ that is not a basepoint or successor, and consider an algebra $B$ with more than one element. Then for any $b \in B$ and any measure $f: \standcoalg{n} \to A \to B$, we can form a measure $\widetilde f: \standcoalg{n} \to A \to B$ such that $\widetilde f_n(a) = b$ and $\widetilde f$ agrees with $f$ everywhere else, since \cref{def: partial homomorphism} imposes no requirements on $\widetilde f_n(a)$. Thus, there are multiple measures $\standcoalg{n} \to A \to B$, equivalently total algebra homomorphisms $\standcoalg{n} \to \ubalg(A, B)$, so we find a contradiction.

    Now, we consider the unique map $A \to [\standcoalg{n},\N]$ and claim that this factors through the injection $\standalg{n} \to [\standcoalg{n},\N]$, so that there is a unique $A \to \standalg{n}$. Since every element of $A$ is either a basepoint or a successor, every element of $A$ is either of the form $n_A$ or has infinitely many predecessors. The elements of the form $n_A$ are mapped those to of the form $n_{[\standcoalg{n},\N]}$, and the elements who have infinitely many predecessors can only be mapped to the `top element' $n_{[\standcoalg{n},\N]} = (i)_{i=0}^n$, since this is the only element which has an $m$-th predecessor for any $m \in \N$. Thus, the unique $A \to [\standcoalg{n},\N]$ indeed factors through $\standalg{n}$.
\end{example}

\section{Examples}
\label{sec: examples}

In the previous section we have built a general theory showing that the category of algebras of an endofunctor $F\colon \C \to \C$ is enriched in the category of coalgebras of that same endofunctor.
We also saw several objects representing the measuring functor $\m\colon \CoAlg^\op \times \Alg^\op \times \Alg \to \Set$ in each of its arguments.
Lastly, we assigned significance to $C$-initial algebras, which are algebras that have the same desirable properties as initial algebras, but are more abundant.

In this section, we would like to bring the theory to life by studying some examples.
Fixing an endofunctor $F$ in each example, we will consider its initial algebra and terminal coalgebra, the representing objects of $\m$ and $C$-initial algebras.
Starting with some straightforward examples to get a feel, we will work towards those which have 
those that serve as semantics for something in computer science.
For most of the examples we will only give the results and the necessary ingredients to obtain them.
However, there are two selected examples where we will go into more detail, showing intermediary steps and confirming the results.
The first of the two selected examples is the monoid type, which is the first example where we see the monoidal structure in action.
The second is the list type, which is the first example which has semantics commonly used in computer science.

In all our examples we consider the monoidal category $(\Set, \times, 1)$ which is locally presentable and closed.
The endofunctors $F\colon \Set \to \Set$ in the examples will all be accessible, and we will provide the monoidal structure of the endofunctor explicitly in each example.
Once an endofunctor is fixed we will denote the category of $F$-algebras by $\Alg$ and the category of $F$-coalgebras by $\CoAlg$.
Recall that due to $\Set$ being locally presentable and $F$ being accessible, the left adjoint of the forgetful functor $U\colon \Alg \to \Set$ exists.
For the same reasons the right adjoint to the forgetful functor $U\colon \CoAlg \to \Set$.
The free and cofree functors will be denoted by $\Fr$ and $\Cof$ respectively, and will be used throughout this section.
Finally, we will usually denote $F$-algebras by $(A,\alpha), (B,\beta)\in \Alg$ and $F$-coalgebras by $(C,\chi), (D,\delta) \in \CoAlg$.

Throughout, given a set map $\varphi\colon C\times A\rightarrow  B$, we may write $\varphi_c(a)$ for $\varphi(c,a)$ for $c\in C$ and $a\in A$. Given sets $S$ and $T$, we write $[S, T]$ for the set of all functions $S\rightarrow T$. We shall only focus on describing the measurings $\varphi\colon C\times A\to B$, and the associated $C$-indexed partial homomorphism $\widehat{\varphi}\colon C\to [A,B]$ is defined as $\widehat{\varphi}(c)=\varphi_c$, and each $\varphi_c\colon A\to B$ is regarded as a function $A\to B$ that only partially respect the algebra structures on $A$ and $B$, for all $c\in C$.

\subsection{The Unit Type}\label{subsec:unit}
First off, we wish to start with a simple example.
We will see that this endofunctor has the singleton $1$ as its initial algebra.
Hence, this example corresponds to the unit type in functional programming, the type with only one value.
Let $F \colon \Set \to \Set$ be the functor given by
\begin{align*}
    F \colon \Set &\longrightarrow \Set \\
    X&\longmapsto 1.
\end{align*}
It has a trivial lax monoidal structure.

The category of $F$-algebras, denoted $\Alg$, has elements
$
    \alpha\colon 1 \to A,
    * \mapsto a_0
$
denoted $(A,a_0)$ and morphisms
$
    f \colon (A,a_0) \to (B,b_0)
$
satisfying $f(a_0) = b_0$.
It can be regarded as the category of pointed sets, or as the coslice category $1/\Set$.
If it is clear we are considering a pointed set, we will sometimes simply write $A \in \Alg$.
The category of $F$-coalgebras, denoted $\CoAlg$, has elements
$
    \chi\colon C \to 1
$
and morphisms
$
    f \colon C \to D,
$
and so is isomorphic to the category of sets.

\subsubsection{Initial and Terminal Objects}
The initial algebra is given by $\id \colon (1,*) \to (1,*)$, since for every algebra $(A,a_0)$ there exists exactly one pointed morphism out of $\id \colon 1 \to 1$.
For any preinitial algebra $(I,i_0)$ the unique morphism $(1,*) \to (I,i_0)$ must be epic.
This can only be the case whenever the underlying function $1 \to I$ is surjective.
The only set which satisfies this condition is $1$, so we conclude there are no preinitial algebras besides the initial algebra.

The terminal coalgebra is given by 1 since the $\CoAlg$ is isomorphic to $\Set$. 
A subterminal coalgebra is any set $X$ such that $X \rightarrowtail 1$ is injective.
The only sets satisfying this are $\emptyset$ and $1$, so we conclude these underlie the only subterminal coalgebras.

\subsubsection{Measurings}

%\begin{definition}
Let $(A,a_0), (B,b_0) \in \Alg$ and $(C, \chi) \in \CoAlg$.
A {measuring from $A$ to $B$ by $C$} is a function $\phi\colon C \times A \to B$, such that \[\phi_c(a_0) = b_0\]for all $c \in C$.
%\end{definition}
We find that in this case, a measuring $\phi\colon C \times A \to B$ is nothing but a family of algebra morphisms $A \to B$ indexed by $C$.

\subsubsection{Free and Cofree Functors}
In order to compute the representing objects of $\mu$, we wish to make use of the constructions provided in \cref{sec: measuring}.
These constructions make use of the free and cofree functor.

The free functor is given by
\begin{align*}
    \Fr\colon \Set &\longrightarrow \Alg \\
    X &\longmapsto (X + 1, \operatorname{inr})\\
    f &\longmapsto f + \id_1
\end{align*}
where $\operatorname{inr}$ is the inclusion $1 \to X + 1$.
The unit and counit are given by
$
    \eta^{\Fr}_X \colon X \to X + 1, x \mapsto x
$
for all set $X$,
and
$
    \epsilon^{\Fr}_A \colon A + 1 \to A ,
    * \mapsto a_0,
    a \mapsto a
$, for all pointed set $(A, a_0)$.
One can easily check these satisfy the triangle identities and hence form an adjunction.

The forgetful functor $U\colon \CoAlg \to \Set$ is an isomorphism.
Hence the cofree functor is the identity as well.

\subsubsection{Representing Objects}
Now we have everything in place to compute the representing objects of $\m$.
The first representing object is the universal measuring coalgebra $\underline{\Alg}(A,B)$.
Its underlying set is given by $[A\setminus\{a_0\}, B]$ with the measuring structure
\begin{align*}
    \ev \colon [A\setminus\{a_0\}, B] \times A &\to B \\
    (f,a) &\mapsto 
    \begin{cases}
        f(a) &\text{ if } a \neq a_0\\
        b_0 &\text{ if } a = a_0.
    \end{cases}
\end{align*}
Indeed, given a measuring $\phi\colon C \times A \to B$, there is exactly one coalgebra morphism $f\colon C \to \underline{\Alg}(A,B)$ which results in a morphism of measurings, namely
\begin{align*}
    f \colon C &\to \underline{\Alg}(A,B) \\
    c &\mapsto {\phi_c}_{|_{A\setminus\{a_0\}}}.
\end{align*}
Conversely, given a coalgebra morphism $f\colon C \to \underline{\Alg}(A,B)$, we get a measuring $\ev \circ (f \times \id) \colon C\times A \to B$.

The second representing object is the measuring tensor $C \triangleright A$, and has as underlying set
$$
C \times A / (c,a_0) \sim (c', a_0),
$$
for all $c,c' \in C$, with the equivalence class of $(c,a_0)$ giving the algebra structure. 
A measuring $\phi$ always uniquely factorizes since $\phi(c,a_0) = \phi(c',a_0) = b_0$ for all $c,c' \in C$.
We can now state that $\m_C(A,-)$ is represented by $C \triangleright A$.

Lastly, given $C \in \CoAlg, B \in \Alg$, we can turn $[C,B]$ into an algebra by taking the constant function $\const_{b_0} \colon C \to B, c \mapsto b_0$
to be the preferred point in $[C,B]$. 
This is the {convolution algebra}.

\subsubsection{$C$-Initial Algebras}
If $C = \emptyset$, every algebra is $C$-initial, since $\emptyset$ is the initial object in $\Set$.
For $C \neq \emptyset$, being $C$-initial implies the underlying set of $\underline{\Alg}(A,X)$ given by $[A\setminus\{a_0\}, X]$ contains only a single element.
This is because for any other $1 \neq D \in \CoAlg$, there exists more than one function $C \to D$. 
For $[A\setminus\{a_0\}, X]$ to contain only one function for all $X$, it must be that $A\setminus\{a_0\} = \emptyset$, hence $A \cong 1$.
The only algebra satisfying this is the initial algebra $\id \colon 1 \to 1$.
We conclude for $C \neq \emptyset$, the only $C$-initial algebra is $\id \colon 1 \to 1$, i.e., the terminal object with its unique point as basepoint.

\subsection{The Empty Type}\label{subsec:empty}
Perhaps an even more fundamental example is that of the empty type.
The empty type is the type with no values, which means we wish to consider a functor which has the empty set as initial algebra.
Let $F \colon \Set \to \Set$ be the functor given by
\begin{align*}
    F \colon \Set &\longrightarrow \Set \\
    X &\longmapsto X.
\end{align*}
Its lax monoidal structure is given given by identities.

The category of $F$-algebras, denoted $\Alg$ has elements
$
    \alpha \colon A \to A
$
denoted $(A,\alpha)$. 
Morphisms $f\colon(A,\alpha) \to (B,\beta)$ are given by functions $ f \colon A \to B $
which make the following diagram commute:
\[\begin{tikzcd}
A & B \\
A & B.
            \arrow["f", from=1-1, to=1-2] 
\arrow["\alpha"',from=1-1, to=2-1] \arrow["\beta",from=1-2, to=2-2]
            \arrow["f", from=2-1, to=2-2]
\end{tikzcd}\]

The category of $F$-coalgebras, denoted $\CoAlg$ has elements
$
    \chi \colon C \to C
$
denoted $(C,\chi)$.
Morphisms $f\colon(C,\chi) \to (D,\delta)$ are given by functions $ f : C \to D $
which make the following diagram commute:
\[\begin{tikzcd}
C & D \\
C & D.
            \arrow["f", from=1-1, to=1-2] 
\arrow["\chi",from=2-1, to=1-1] \arrow["\delta"',from=2-2, to=1-2]
            \arrow["f", from=2-1, to=2-2]
\end{tikzcd}\]
We observe the category $\CoAlg$ is isomorphic to the category of $\Alg$ by definition. 
Nevertheless, we will distinguish between these categories to see which role they play in the upcoming constructions.

\subsubsection{Initial and Terminal Objects}
The initial algebra is given by $\id\colon \emptyset \to \emptyset$, since for every algebra $(A,\alpha)$ there exists exactly morphism out of $\id \colon \emptyset \to \emptyset$.
The only preinitial algebra is given by the initial algebra itself, since there exist no epimorphisms out of the empty set other than the identity.

The terminal coalgebra is given by $\id \colon 1 \to 1$ since there exists exactly one morphism into $\id \colon 1 \to 1$ for every coalgebra $(C,\chi)$.
A subterminal coalgebra is then any coalgebra $(C,\chi)$ such that $(C,\chi) \rightarrowtail (1,\id)$ is injective.
The only coalgebras satisfying this are $(\emptyset,\id)$ and $(1,\id)$ and we conclude $\emptyset$ is the only subterminal coalgebra besides the terminal algebra.

\subsubsection{Measurings}
% \begin{definition}
%     [measuring]
Let $(A,\alpha), (B,\beta) \in \Alg$ and $(C,\chi) \in \CoAlg$.
A {measuring from $A$ to $B$ by $C$} is a function $\phi\colon C \times A \to B$, such that \[\phi_c(\alpha(a)) = \beta(\phi_{\chi(c)}(a))\] for all $(c,a) \in C \times A$.
% \end{definition}
Another way to state the condition on measurings is to say the diagram
\[\begin{tikzcd}
A & [2em]B \\
A & B
            \arrow["\phi_{\chi(c)}", from=1-1, to=1-2] 
\arrow["\alpha"',from=1-1, to=2-1] \arrow["\beta",from=1-2, to=2-2]
            \arrow["\phi_c"', from=2-1, to=2-2]
\end{tikzcd}\]
must commute for all $c \in C$. 

\begin{example}
    Consider the algebras $(\N, (\cdot 4)), (\N, (\cdot 8))$ and the coalgebra $(\N, (\cdot 2))$.
    A measuring from $(\N, (\cdot 4))$ to $(\N, (\cdot 8))$ by $(\N, (\cdot 2))$ is given by
    \begin{align*}
        \phi \colon \N \times \N &\longrightarrow \N \\
        (i,j) &\longmapsto i\cdot j^2.
    \end{align*}
    To verify this is a measuring, we must check
    $
    \phi(i,4j) = 8 \cdot \phi(2i, j),
    $
    which is the case since
    $$
    \phi(i,4j) = i \cdot (4j)^2 = 16ij^2 = 8 \cdot 2i \cdot j^2  = 8 \cdot \phi(2 \cdot i, j).
    $$
\end{example}

% The \emph{category of measurings from $A$ to $B$} has as elements measurings $\phi : C \times A \to B$, 
% and as morphisms $f:\phi \to \psi$ functions $f : C \to D$ such that
% \[\begin{tikzcd}
% 	{C\times A} & B \\
% 	{D\times A}
% 	\arrow["{{f \times \id}}"', from=1-1, to=2-1]
% 	\arrow["\phi", from=1-1, to=1-2]
% 	\arrow["\psi"', from=2-1, to=1-2]
% \end{tikzcd}\]

\subsubsection{Free and Cofree Functors}
In order to compute the representing objects in the upcoming section, we need the left and right adjoint of the forgetful functors.
They are the free and cofree functor and are respectively given by
\begin{align*}
    \Fr \colon \Set &\longrightarrow \Alg \\
    X &\longmapsto
    \begin{pmatrix}
        \coprod_{i \in \N} X_i &\xrightarrow{\Sh} \coprod_{i \in \N} X_i\\
        X_i \ni x &\mapsto x \in X_{i+1}
    \end{pmatrix},
\end{align*}
where $X_i = X$,
and
\begin{align*}
    \Cof \colon \Set &\longrightarrow \CoAlg \\
    X &\longmapsto
    \begin{pmatrix}
        \prod_{i \in \N} X_i &\xrightarrow{\CoSh} \prod_{i \in \N} X_i\\
        (x_0, x_1, x_2, \dots ) &\mapsto (x_1, x_2, \dots )
    \end{pmatrix},
\end{align*}
where $X_i = X$.
Their behavior on functions is in both cases given by applying the function component-wise.

To verify these are adjoints, one must verify
$
[X,A] \cong \Alg(\Fr(X),A)
$
and
$
[C,X] \cong \Alg(C,\Cof(X))
$.
This can be done by constructing explicit natural bijections.
The first bijection is given by
$
    [X,A] \to \Alg(\Fr(X),A), f \mapsto \coprod_{i \in \N} \alpha^i \circ f
$
with inverse
$
    \Alg(\Fr(X),A) \to [X,A] ,  \coprod_{i \in \N} f_i \mapsto f_0.
$
The second bijection is given by
$
    [C,X] \to \CoAlg(C,\Cof(X)) , f \mapsto \langle f \circ \chi^i \rangle_{i \geq 0}
$
with inverse
$
    \CoAlg(C,\Cof(X)) \to [C,X]  , \langle f_i \rangle_{i \geq 0} \mapsto f_0.
$
One can check those bijections are indeed natural.

\subsubsection{Representing Objects}
The universal measuring coalgebra $\underline{\Alg}(A,B)$ has as underlying set
$$
\left.\left\lbrace (f_i)_{i\geq 0}=(f_0, f_1, \dots) \in \prod_{i \in \N} [A,B] \, \right| \, \beta \circ f_{i+1} = f_i \circ \alpha\right\rbrace.
$$
Its coalgebra structure is given by shifting the entries to the left, or more precisely by the restriction of $\CoSh$ to $\underline{\Alg}(A,B)$. 
The terminal object in the category of measurings from $A$ to $B$ is then given by
\begin{align*}
    \ev \colon \underline{\Alg}(A,B) \times A &\longrightarrow B \\
    ( (f_i)_{i \geq 0},a) &\longmapsto f_0(a).
\end{align*}
Indeed, given a measuring $\phi \colon C \times A \to B$, there is exactly one coalgebra morphism $f\colon C \to \underline{\Alg}(A,B)$ which results in a morphism of measurings, namely
$
    f \colon C \to \underline{\Alg}(A,B), c \mapsto (\phi_{\chi^i(c)})_{i\geq 0}.
$
Conversely, given a morphism $f\colon C \to \underline{\Alg}(A,B)$, composing with $\ev$ yields a measuring $\ev \circ (f \times \id) \colon C\times A \to B$.
This yields a bijective correspondence between morphisms $f\colon C \to \underline{\Alg}(A,B)$ and measurings $\phi \colon C\times A \to B$.

The measuring tensor $C \triangleright A$ has as underlying set
$$
\coprod_{i\in \N} (C \times A)_i
\big/
\sim,
$$
where the equivalence relation $\sim$ is given by
$
(\chi(c),a)_{i+1}\sim (c,\alpha(a))_{i}.
$
Its algebra structure is given by shifting to the right, namely by the function $\Sh$. 
A measuring $\phi$ corresponds to the algebra morphism
$
    \widetilde{\phi} \colon C \triangleright A \to B ,
    (c,a)_i \mapsto \phi_{\chi^i(c)}(a) 
$
and an algebra morphism $\widetilde{\phi} \colon C \triangleright A \to B$ corresponds to a measuring $\phi = \widetilde{\phi}|_{(C\times A)_0}$.
We can then state $\m_C(A,-)$ is represented by $C \triangleright A$.

Finally, given $C \in \CoAlg, B \in \Alg$, we define the convolution algebra $[C,B]$ to be the algebra structure given by
\begin{align*}
    \beta_*\chi^* \colon [C,B] &\to [C,B]\\
    f &\mapsto \beta \circ f \circ \chi.
\end{align*}

\subsubsection{$C$-Initial Algebras}
If $C = \emptyset$ every algebra is $C$-initial, with $(1,\id)$ being the terminal $\emptyset$-initial algebra.
For $C \neq \emptyset$ the only $C$-initial algebra is $(\emptyset, \id)$, since every other algebra would give too much freedom in constructing maps from $C$ to $\Alg(A,X)$ for there to be a unique morphism.
This means $(\emptyset, \id)$ is also the terminal $C$-initial algebra for all $C \in \CoAlg$.

\subsection{The Monoid Type}\label{subsec:monoid}
In this example, we will begin to see the flexibility measurings provide.
This is also the motivation for discussing more details regarding the representing objects than previously done.
We will be considering the type which takes values in a fixed monoid $(M, \bullet, e)$.
This type is not as frequently used in computer science, but it does give us insight into the general theory.
To study this type, we must find a functor which has it as an initial algebra.

Let $F \colon \Set \to \Set$ be the functor given by
\begin{align*}
    F \colon \Set &\longrightarrow \Set \\
    X &\longmapsto M\\
    f &\longmapsto \id_M
\end{align*}
where $(M, \bullet, e)$ is some fixed monoid in $\Set$.
The lax monoidal structure is given by monoid structure on $M$, namely
$\nabla_{A,B} \colon M \times M \to M$
    where $\nabla_{A,B}(m,m')= m \bullet m',
$
and
$\eta \colon 1 \to M$, where $\eta(*)=e$.
Since we will need it later, we also remark the lax closed structure is given by
$
    \widetilde{\nabla}_{A,B} \colon M \to [M,M]$ where $\widetilde{\nabla}_{A,B}(m)=r_m$,    
where $r_m \colon M \to M$ is defined as $r_m(m')=m'\bullet m$, i.e., it is the function which multiplies an element of $M$ by $m \in M$ from the right.

The category $\Alg$ has elements $\alpha \colon M \to A$, with morphisms $f\colon(A,\alpha) \to (B,\beta)$ given by commuting squares (or triangles if you prefer)
\[\begin{tikzcd}
	M & M \\
	A & B.
	\arrow["\alpha"', from=1-1, to=2-1]
	\arrow["f", from=2-1, to=2-2]
	\arrow["\beta", from=1-2, to=2-2]
	\arrow["{\id_M}", from=1-1, to=1-2]
\end{tikzcd}\]
The category $\CoAlg$ has elements $\chi \colon C \to M$, with morphisms $f\colon(C,\chi) \to (D,\delta)$ given by commuting squares
\[\begin{tikzcd}
	M & M \\
	C & D.
	\arrow["f", from=2-1, to=2-2]
	\arrow["{\id_M}", from=1-1, to=1-2]
	\arrow["\chi", from=2-1, to=1-1]
	\arrow["\delta"', from=2-2, to=1-2]
\end{tikzcd}\]
Notice how these categories correspond to under category $M / \Set$ and the over category $\Set / M$ respectively.

\subsubsection{Initial and Terminal Objects}
The initial algebra and the terminal coalgebra are both given by
$$
\id_M \colon M \to M.
$$
Any algebra for which the function $\alpha \colon M \to A$ is surjective gives a preinitial algebra.
Similarly, any coalgebra for which the function $\chi \colon C \to M$ is injective will give a subterminal coalgebra.

\subsubsection{Measuring}
% \begin{definition}
%     [measuring]
Let $(A,\alpha), (B,\beta)\in \Alg$ and $(C,\chi) \in \CoAlg$.
    A {measuring from $A$ to $B$ by $C$} is a function $\phi\colon C \times A \to B$, such that \[\phi_c(\alpha(m)) = \beta(\chi(c) \bullet m)\] for all $c \in C$ and $m \in M$.
% \end{definition}
Comparing this to the condition on a regular algebra morphism $f\colon A \to B$, given by
$
f(\alpha(m)) = \beta(m),
$
we see a coalgebra $C$ introduces some sort of twisting.
The above condition on $\phi$ is equivalent to asking the following diagram commutes
\[\begin{tikzcd}
	& {M \times M} & M \\
	{C\times M} \\
	& {C\times A} & B.
	\arrow["{\chi \times \id_M}", from=2-1, to=1-2]
	\arrow["{\id_C\times \alpha}", from=2-1, to=3-2]
	\arrow["\phi", from=3-2, to=3-3]
	\arrow["\beta", from=1-3, to=3-3]
	\arrow["\bullet", from=1-2, to=1-3]
\end{tikzcd}\]

\begin{example}
    Consider the algebra $(M, \id_M)$ and the coalgebra $(M,\id_M)$.
    A trivial example of a measuring is the monoid multiplication $\bullet \colon M \times M \to M$.
\end{example}

\begin{example}
    Let $A \in \Set$ be arbitrary and consider the algebra structure on $[A,M]$ given by
    \begin{align*}
        \alpha \colon M &\to [A,M] \\
        m &\mapsto \const_m
    \end{align*}
    and the coalgebra $(M,\id_M)$.
    A measuring from $([A,M], \alpha)$ to itself by $(M,\id_M)$ is given by
    \begin{align*}
        \phi \colon M \times [A,M] &\longrightarrow [A,M]\\
        (m, f) &\longmapsto (a \mapsto f(a) \bullet m).
    \end{align*}
    Indeed, we can verify
    $
    \phi(m', \const_m) = \const_{m' \bullet m} = \alpha(m' \bullet m).
    $
\end{example}

\begin{example}
    Let $\gls{List}$ denote the set of finite lists with entries in the monoid $M$.
    We shall be slightly informal in this example, but the 
 the reader can see more precise statements in \cref{subsec:list}.
    We can give it an algebra structure
    $
        \alpha \colon M \to M^*
    $
    by mapping $m$ to the list $[m]$ that contains the single element $m$.
    We can also give it a coalgebra structure
    $
        \chi \colon M^* \to M
    $ 
    by sending a list to the product of all entries in the list.
    If acquainted with functional programming, one might recognize the definition of $\fold$ in the coalgebra structure.

    Let $\gls{EmptList}$ denote the empty list and $\gls{ConcList}$ denote the concatenation of $m \in M$ and $\ell \in M^*$. A measuring from $(M^*,\alpha)$ to itself by $(M^*,\chi)$ is given by
    \begin{align*}
        \phi \colon M^* \times M^* &\longrightarrow M^*\\
        (\ell', [\:]) &\longmapsto [\:]\\
        (\ell', m:\ell) &\longmapsto (\chi(\ell') \bullet m):\ell,
    \end{align*}
    collapsing the entire list $\ell'$ onto the first element of the list $m:\ell$ using the monoidal structure on $M$.
    To verify this is a measuring we check
    $
    \phi_{\ell'}(\alpha(m)) = \phi_{\ell'}([m]) = [\chi(\ell')\bullet m] = \alpha(\chi(\ell')\bullet m).
    $

    Note this is not the only measuring we could have defined.
    Another measuring is given by
    \begin{align*}
        \phi' \colon M^* \times M^* &\longrightarrow M^*\\
        \ell', [\:] &\longmapsto [\:]\\
        \ell', m:\ell &\longmapsto [\chi(\ell') \bullet m],
    \end{align*}
    simply ignoring the last part of the list $m:\ell$.
    Again we can verify this is a measuring, using the exact same equation as before to find
    $
    \phi_{\ell'}(\alpha(m)) = \phi_{\ell'}([m]) = [\chi(\ell')\bullet m] = \alpha(\chi(\ell')\bullet m).
    $
\end{example}

\subsubsection{Free and Cofree Functors}
% Again, we wish to make use of the left and right adjoint of the forgetful functors.
The left adjoint free functor is given by
\begin{align*}
    \Fr \colon \Set &\longrightarrow \Alg \\
    X &\longmapsto (X + M, \operatorname{\alpha_{\Fr}})\\
    f &\longmapsto f + \id_M.
\end{align*}
where $\alpha_{\Fr} = \inr$, the inclusion of $M$ into $X+M$.
The unit and counit are given by
$
    \eta^{\Fr}_X \colon X \to X + M$ where $\eta^{\Fr}_X(x)=x$, for any set $X$, 
and
$
    \epsilon^{\Fr}_A \colon A + M \to A$
    is defined by $\epsilon^{\Fr}_A(m)=\alpha(m)$ and $\epsilon^{\Fr}_A(a)=a$, for all $m\in M$ and $a\in A$, where $(A, \alpha)$ is an algebra.
One can easily check these satisfy the triangle identities and hence form an adjunction.

The right adjoint cofree functor is given by
\begin{align*}
    \Cof\colon \Set &\longrightarrow \CoAlg \\
    X &\longmapsto (X \times M, \chi_{\Cof})\\
    f &\longmapsto (f\times \id_M).
\end{align*}
where $\chi_{\Cof} = \pr_M$.
The unit and counit are given by
$
    \eta^{\Cof}_C \colon C \to C \times M$ where $\eta^{\Cof}_C(c)= (c, \chi(c))
$ for all coalgebra $(C, \chi)$,
and
$ \epsilon^{\Cof}_X \colon X\times M \to X$ where $\epsilon^{\Cof}_X(x,m)=x$ for all set $X$.
Again, one can check the triangle identities to verify the adjunction.

\subsubsection{Representing Objects}
The universal measuring coalgebra $\underline{\Alg}(A,B)$ is given by the following equalizer
\[\begin{tikzcd}
	{\underline{\Alg}(A,B)} & {[A,B] \times M} & [2em]{[M,B]\times M}
	\arrow[dashed, "\operatorname{eq}", from=1-1, to=1-2]
	\arrow["{[\alpha, B] \times \id_M}", shift left, from=1-2, to=1-3]
	\arrow["\Phi"', shift right, from=1-2, to=1-3]
\end{tikzcd}\]
where $\Phi$ is the transpose of
$$
\widetilde{\Phi} \colon [A,B] \times M \xto{\chi_{\Cof}} M \xto{F(\epsilon^{\Cof})} M \xto{\widetilde{\nabla}_{A,B}} [M,M] \xto{\beta_*} [M,B].
$$
If we compute the composition we find
$
    \widetilde{\Phi} \colon [A,B] \times M  \to [M,B]$ is defined explicitly by $\widetilde{\Phi}(f,m)=r_m$,
where we recall $r_m \colon M \to M$ was multiplication by $m$ on the right. 
Taking its transpose $\Phi = \Cof(\widetilde{\Phi}) \circ \eta$ we find the explicit formula for $\Phi$ to be 
\begin{align*}
    \Phi \colon [A,B] \times M &\longrightarrow [M,B]\times M \\
    (f,m) &\longmapsto (m' \mapsto \beta(m' \bullet m), m).
\end{align*}
Now that we know explicitly which maps we want to equalize, we can compute
$$
\underline{\Alg}(A,B) = \{ (f,m) \in [A,B] \times M \mid (f \circ \alpha)(m') = \beta(m \bullet m')  \}.
$$
with coalgebra structure given by the projection
$
    \pr_M \colon \underline{\Alg}(A,B)\to M$ where $\pr_M(f,m)=m$.
We would like to verify $\underline{\Alg}(A,B)$ is indeed the universal measuring.
In other words, we would like to construct a natural bijection
$$
\Psi \colon \m_C(A,B) \xto{\cong} \CoAlg(C,\underline{\Alg}(A,B)).
$$
This bijection is given by sending a measuring $\phi \in \m_C(A,B)$ to
$$
\Phi(\phi)(c) = (a \mapsto \phi(c,a), \chi(c)).
$$
This is a well-defined coalgebra morphism since
$
\chi(c) = \pr_M(\Phi(\phi)(c)).
$
Its inverse $\Psi\inv$ sends a coalgebra morphism $f \colon C \to \underline{\Alg}(A,B)$ to the measuring
\begin{align*}
    \Psi\inv(f) \colon C\times A &\to B \\
    (c,a) &\mapsto (\pr_{[A,B]} \circ f)(c)(a).
\end{align*}
We must verify this is a measuring, which we can do using that $f$ is a coalgebra morphism.
Observe
$
(\pr_{[A,B]} \circ f)(c)(\alpha(m)) = \beta((\pr_M \circ f)(c) \bullet m) = \beta(\chi(c), m).
$
It is easy to verify that $\Psi$ and $\Psi\inv$ are inverses of each other and that the bijection is natural.

The measuring tensor $C \triangleright A$ is given by the coequalizer
\[\begin{tikzcd}
	{(C \times M) + M} & [3em] {(C \times A) + M} & {C \triangleright A},
	\arrow["\Psi"', shift right, from=1-1, to=1-2]
	\arrow["{\id_C \times \alpha + \id_M}", shift left, from=1-1, to=1-2]
	\arrow[dashed, "\operatorname{coeq}", from=1-2, to=1-3]
\end{tikzcd}\]
where $\Psi$ is the transpose of
$$
\widetilde{\Psi} \colon C \times M \xto{\chi \times \id_M} M \times M \xto{\nabla_{A,B}} M \xto{F(\eta^{\Fr})} M \xto{\alpha_{\Fr}} (C \times A) + M.
$$
If we compute the composition we find
$
    \widetilde{\Phi} \colon C \times M  \to (C \times A) + M,
    (c,m) \mapsto \chi(c) \bullet m.
$
Taking its transpose $\Psi = \epsilon \circ \Fr(\widetilde{\Phi})$ we find the explicit formula for $\Phi$ to be 
\begin{align*}
    \Psi \colon {(C \times M) + M} &\longrightarrow {(C \times A) + M} \\
    (c,m) &\longmapsto \chi(c) \bullet m\\
    m' &\longmapsto m'.
\end{align*}
Now that we know explicitly which maps we want to coequalize, we can compute
$$
C \triangleright A \cong
\big( (C \times A) + M \big) / (c,\alpha(m)) \sim (\chi(c) \bullet m).
$$
with algebra structure given by
$
    M \to C \triangleright A ,
    m \mapsto [m].
$
A slightly more intuitive may be to write $C \triangleright A$ as the pushout
\[\begin{tikzcd}
	{C \times M} & M \\
	{C \times A} & {C \triangleright A}
	\arrow["{\id_C \times \alpha}"', from=1-1, to=2-1]
	\arrow["{ f }", from=1-1, to=1-2]
	\arrow[from=2-1, to=2-2]
	\arrow[from=1-2, to=2-2]
	\arrow["\lrcorner"{anchor=center, pos=0.125, rotate=180}, draw=none, from=2-2, to=1-1]
\end{tikzcd}\]
where $f \colon C \times M \to M, (c,m) \mapsto \chi(c) \bullet m$.
To verify $C \triangleright A$ is measuring tensor, we construct a natural bijection
$$
\Xi \colon \m_C(A,B) \xto{\cong} \Alg(C \triangleright A,B).
$$
This bijection is given by sending a measuring $\phi \in \m_C(A,B)$ to
\begin{align*}
    \Xi(\phi) \colon C \triangleright A &\longrightarrow B \\
    [c,a] &\longmapsto \phi(c,a)\\
    [m] &\longmapsto \beta(m).
\end{align*}
We must check this is well-defined, which is the case since
$
\Xi(\phi)([c,\alpha(m)]) = \phi(c,\alpha(m)) = \beta(\chi(c) \bullet m) = \Xi(\phi)([\chi(c) \bullet m])
$
and the equivalence relation defining $C \triangleright A$ is generated by $(c,\alpha(m)) \sim \chi(c) \bullet m$.
It is a well-defined algebra morphism since $[m] = \beta(m)$ by definition.
Its inverse $\Xi\inv$ sends an algebra morphism $f \colon C \triangleright A \to B$ to the measuring
\begin{align*}
    \Xi\inv(f) \colon C\times A &\to B \\
    (c,a) &\mapsto f([c,a]).
\end{align*}
We must verify this is a measuring, which we can do using that $f$ is an algebra morphism.
Observe
$
f([c,\alpha(m)]) = f([\chi(c) \bullet m]) = \beta([\chi(c) \bullet m]),
$
which shows $\Xi\inv(f)$ is a measuring.
It is easy to verify the $\Xi$ and $\Xi\inv$ are inverses of each other and that the bijection is natural.

Lastly, the convolution algebra $[C,B]$ has algebra structure given by the composition
$$
M \xto{\widetilde{\nabla}_{C,B}} [M,M] \xto{\beta_*\circ \chi^*} [C,B].
$$
It sends $m\in M$ to $(c \mapsto \beta(\chi(c)\bullet m))$.

\subsubsection{$C$-Initial Algebras}
One of the characterizations of a $C$-initial algebra $A$ is that $C \triangleright A$ is isomorphic to the initial object.
In our case this would mean that
$$
M \cong \left.\Big( (C \times A) + M \Big) \right/ {}_{\displaystyle(c,\alpha(m)) \sim (\chi(c) \bullet m)},
$$
which is the case if and only if $\alpha \colon M \to A$ is surjective.
We conclude any preinitial algebra $\alpha \colon M \twoheadrightarrow A$ is also $C$-initial for any coalgebra $C$.
This also implies the terminal algebra $M \to 1$ is the terminal $C$-initial algebra for all coalgebras $C$.

\subsection{The Natural Numbers Type}\label{subsec:naturals}
In \cref{sec: measuring}, the main example provided was the example of the natural numbers type.
In this section we will repeat and elaborate on that example.
Finally, we will add a new constructive proof regarding $C$-initial algebras.
To study this type, we must find a functor which has it as an initial algebra.

Consider the functor
\begin{align*}
    F \colon \Set &\longrightarrow \Set \\
    X &\longmapsto 1+X\\
    f &\longmapsto \id_1 + f.
\end{align*}
The lax monoidal structure is given by
\begin{align*}
    \nabla_{X,Y} \colon 1 + X \times 1 + Y &\longrightarrow 1 + (X\times Y)\\
    (*,y) &\longmapsto *\\
    (x,*) &\longmapsto *\\
    (x,y) &\longmapsto (x,y)
\end{align*}
and
$
\inr \colon 1 \to 1 + 1.
$

The category of $F$-algebras has elements
$
    \alpha \colon 1+ A \to A
$
denoted $(A,\alpha)$.
Sometimes we will denote $\alpha(*)$ by $0_A$, and more generally $\alpha^n(0_A)$ will be denoted as $n_A$.
Morphisms $f\colon (A,\alpha) \to (B,\beta)$ are given by functions $ f \colon A \to B $
which make the following diagram commute:
\[\begin{tikzcd}
1+ A & 1+ B \\
A & B.
            \arrow["\id + f ", from=1-1, to=1-2] 
\arrow["\alpha",from=1-1, to=2-1] \arrow["\beta",from=1-2, to=2-2]
            \arrow["f", from=2-1, to=2-2]
\end{tikzcd}\]
The category of $F$-coalgebras has elements
$
    \chi \colon C \to 1+ C,
$
denoted $(C,\chi)$.
Morphisms $f\colon(C,\chi) \to (D,\delta)$ are given by functions $ f \colon C \to D $
which make the following diagram commute
\[\begin{tikzcd}
1+C & 1 +D \\
C & D.
            \arrow["\id_1 + f ", from=1-1, to=1-2] 
\arrow["\chi",from=2-1, to=1-1] \arrow["\delta",from=2-2, to=1-2]
            \arrow["f", from=2-1, to=2-2]
\end{tikzcd}\]

\subsubsection{Initial and Terminal Objects}
The initial algebra is the given by the set of natural numbers $\N$.
Its algebra structure is given by succession, as we recalled in \cref{subsec: W-types}.
To be more precise, it is given by
\begin{align*}
    1 + \N &\to \N \\
    * &\mapsto 0\\
    n &\mapsto n+1.
\end{align*}
We denote by $-_{A} \colon \N \to A$ the unique algebra morphism.
% To verify $\N$ is indeed the initial algebra, let $(A, \alpha)$ be an arbitrary algebra.
% The unique algebra morphism $\N \to A$, denoted $-_{A} \colon \N \to A$ is forced by the commutative square
% \[\begin{tikzcd}
% 1+ \N & 1+A \\
% \N & A
%             \arrow["\id + -_{A}", from=1-1, to=1-2] 
% \arrow[from=1-1, to=2-1] \arrow["\alpha",from=1-2, to=2-2]
%             \arrow["-_{A}", from=2-1, to=2-2],
% \end{tikzcd}\]
% and defined as
% \begin{align*}
%     -_{A}\colon \N &\to A \\
%     i &\mapsto \alpha^{i}(0_A).
% \end{align*}
The sets of the form $\n \coloneqq \{0,1,\dots, n\}$ are preinitial algebras with algebra structure given by
\begin{align*}
    1 + \n &\to \n \\
    * &\mapsto 0\\
    i &\mapsto \min(i,n).
\end{align*} 
Recall from \cref{example: terminal coalgebra of Nat} that the terminal coalgebra is given by $\eN \coloneqq \N + \{\infty\}$.
Its coalgebra structure is given by
\begin{align*}
    \eN &\to 1+ \eN \\
    i &\mapsto
    \begin{cases}
        * &\text{ if }  i = 0\\
        i-1 &\text{ otherwise}
    \end{cases}\\
    \infty &\mapsto \infty.
\end{align*}
% To verify this is indeed the terminal coalgebra, let $(C, \chi)$ be an arbitrary coalgebra.
The unique coalgebra morphism $C \to \eN$ is denoted $\ind{-}\colon C \to \eN$.  
% is forced by the commutative square
% \[\begin{tikzcd}
% 1+ C & 1+ \eN \\
% C & \eN
%             \arrow["\id + \ind{-}", from=1-1, to=1-2] 
% \arrow["\chi", from=2-1, to=1-1] \arrow["",from=2-2, to=1-2]
%             \arrow["\ind{-}", from=2-1, to=2-2],
% \end{tikzcd}\]
% and defined as
% \[
%     \ind{-}\colon C \to \eN ,
%     c \mapsto 
%     \begin{cases}
%         0  &\text{if } \chi(c) = *\\
%         1 + \ind{c'} & \text{if } \chi(c) = c'\\
%     \end{cases}.
% \]
Subterminal coalgebras are given by subsets of $\eN$ which are closed under the coalgebra structure: $\emptyset$, $\mathbb{n}^\circ$, $\coaN$, $\eN$ as well as the union of $\{\infty\}$ with any of the these sets.
The coalgebra structure is simply a restriction of the coalgebra structure of $\eN$.
Finally, we would like to point out that $\n^\circ$ has a universal property.
Since $\n^\circ$ is subterminal, there exists at most one morphism $C \to \n^\circ$ into it for every $C \in \CoAlg$.
This morphism exists if and only if $\ind{c} \leq n$ for all $c \in C$, which is the universal property of $\n^\circ$.

\subsubsection{Measuring}
Let $(A,\alpha), (B,\beta) \in \Alg$ and $(C,\chi) \in \CoAlg$.
A {measuring from $A$ to $B$ by $C$} is a function $\phi \colon C \times A \to B$ satisfying
\begin{enumerate}
    \item $\phi_c(0_A) = 0_B$ for all $c \in C$
    \item $\phi_c(\alpha(a)) = 0_B$ if $\chi(c) = *$
    \item $\phi_c(\alpha(a)) = \beta(\phi_{c'}(a))$ if $\chi(c) = c'$.
\end{enumerate}
Intuitively, 
%a measuring is a partial algebra morphism from $A$ to $B$, where 
the coalgebra $C$ determines up to which point we are computing the algebra morphism.

% In the examples below, we wish to explore the new possibilities measurings have opened up for us, as well as show the limitations of measurings.

\begin{example}
Notice that the algebra structures involved can be quite limiting.
We claim there does not exist a total algebra morphism $f\colon\n \to \N$.
Denoting the algebra structure of $\n$ by $\alpha_n \colon 1 + \n \to \n$, we can create a contradiction by noting that for $n \in \n$, we must have
$
f(n) = f(\alpha_n(n)) = 1 + f(n)
$
which cannot be the case.
However, we are able to define a measuring from $\n$ to $\N$.
We define the measuring
\begin{align*}
    \phi \colon {\n}^\circ \times \n &\to \N \\
    (i,j) &\mapsto \min(i,j).
\end{align*}
We note that checking this is a measuring can be done by comparing the definition of measuring and that of $\phi$.
Moreover, comparing with the definition also reveals this morphism is unique.
\end{example}

% We would also like to provide a non-example to see a situation in which the measuring does not exist.
\begin{example}
    Say we would like to define a measuring $\phi\colon \n^\circ \times (\n-1) \to \N$.
    By definition $\phi$ must satisfy
    \begin{align*}
        \phi(0,j) &= 0 && \text{ for all } j \in \n-1\\
        \phi(i,0) &= 0 && \text{ for all } i \in \n^\circ\\
        \phi(i,\alpha_{n-1}(j)) &= 1 + \phi(i-1, j) && \text{ for all } i \neq 0, j \in \n-1. 
    \end{align*}
    This forces the definition of $\phi$ to be $\phi(i,j) = \min(i,j)$ for all $0 \leq i,j \leq n-1$.
    However, we run into contradictions when computing $\phi(n,n-1)$.
    Since $\alpha_{n-1}(n-2) = \alpha_{n-1}(n-1) = n-1$, we must have that
    $$
    \phi(n,n-1) = \phi(n,\alpha_{n-1}(n-2)) = 1 + \phi(n-1, n-2) = 1 + \min(n-1,n-2) = n-1,
    $$
    but also
    $$
    \phi(n,n-1) = \phi(n,\alpha_{n-1}(n-1)) = 1 + \phi(n-1, n-1) = 1 + \min(n-1,n-1) = n.
    $$
    We see there does not exist a measuring $\phi \colon \n^\circ \times (\n-1) \to \N$.
    This is due to the fact that we tried to take the inductive definition of $\phi$ one step too far.
\end{example}

\subsubsection{Free and Cofree Functors}
% Again, we aim to use the free and cofree functors to construct the representing objects.
The free functor is defined by mapping $X \in \Set$ to the algebra
\begin{align*}
    \alpha\colon 1+\N \times (X + \{0_X\}) &\to \N \times (X + \{0_X\}) \\
    * &\mapsto (0, 0_X)\\
    (i,0_X) &\mapsto (i+1,0_X)\\
    (i,x) &\mapsto (i+1,x).
\end{align*}
The behavior of $\Fr$ on functions $f\colon X \to Y$ is given by
\begin{align*}
    \Fr(f) \colon \N \times (X + \{0_X\}) &\to \N \times (Y + \{0_Y\}) \\
    (i,x) &\mapsto (i,f(x))\\
    (i, 0_X) &\mapsto (i, 0_Y).
\end{align*}
Checking this is a left adjoint to the forgetful functor can be done by constructing the natural bijection
$$
[X,A] \cong \Alg(\Fr(X),A),
$$
for any set $X$ and algebra $(A, \alpha)$
The bijection is given by
$
    f \mapsto \widetilde f,
$
where $\widetilde{f}$ is defined inductively as
\begin{align*}
    \widetilde f \colon \N \times (X + \{0_X\}) &\to B \\
    (0, 0_X) &\mapsto 0_B\\
    (0, x) &\mapsto f(x)\\
    (i, 0_X) &\mapsto \beta(\widetilde f(i-1,0_X))\\
    (i, x) &\mapsto \beta(\widetilde f(i-1,x)).
\end{align*}
The inverse of the bijection $[X,A] \to \Alg(\Fr(X),A)$ is given by restriction to $X$.
One can check this bijection is natural.

The cofree functor is defined by mapping $X \in \Set$ to the following coalgebra:
\begin{align*}
    \delta \colon X \times X^{\leq \infty} &\to 1 + X \times X^{\leq \infty}\\
    (x_0, [\:]) &\mapsto *\\
    (x_0, x:\ell) &\mapsto (x,\ell)
\end{align*}
Here, $X^{\leq \infty}$ is the set of finite or infinite lists with entries in $X$ (see \cref{subsubsec:initial and terminal objects} for details).
The behavior of $\Cof$ on functions $f\colon X \to Y$ is forced and given by applying $f$ to all elements of $X$ in any element of $X \times X^{\leq \infty}$.
Checking this is a right adjoint to the forgetful functor can be done by constructing the natural bijection
$$
[C,X] \cong \CoAlg(C, \Cof(X)),
$$
The bijection is given by
$
    [C,X] \to \CoAlg(C,\Cof(X)),
    f \mapsto \widetilde f,
$
where
\begin{align*}
    \widetilde f = \langle \widetilde f_0,\widetilde f_1 \rangle \colon C &\longrightarrow X \times X^{\leq \infty} \\
    c &\longmapsto
    \begin{cases}
        (f(c), f(c'):\widetilde f_1(c')) \text{ if } \chi(c) = c'\\
        (f(c), [\:]) \text{ if } \chi(c) = *.
    \end{cases}
\end{align*}
Notice the function $\widetilde f_0 = f$, which also gives the inverse for the bijection.
Verifying $\widetilde f$ is a coalgebra morphism is left to the reader.

\subsubsection{Representing Objects}
For more details, we refer the reader to the upcoming section about lists, since it is a direct generalization of what we are considering here.

First we wish to construct the universal measuring coalgebra $\underline{\Alg}(A,B)$.
The underlying set of $\underline{\Alg}(A,B)$ is a subset of $[A,B] \times [A,B]^{\leq \infty}$, given by
\begin{align*}
  \{
(f_0, f_{i+1})_{i\in I}
\mid
f_{i}(0_A) = 0_B, f_i (\alpha(a))  &= \beta( f_{i+1}(a) ),\\
&f_i (\alpha(a)) = \beta( f_{i+1}(a) ) \text{ for all } a \in A
\},  
\end{align*}
where by $f_{\max}$ we mean the last element of the stream $(f_{i+1})_{i\in I}$, if it exists.
Unpacking the above, we find elements of $\underline{\Alg}(A,B)$ are non empty streams $(f_0, f_{i+1})_{i\in I}$ such that
$
f_i(\alpha(a)) = \beta(f_{i+1}(a)),
$
and the last element of the stream (if it exists) should be the function mapping all elements of the form $\alpha(a)$ to the element $0_B \in B$.
The coalgebra structure of $\underline{\Alg}(A,B)$ is given by shifting to the left, 
$
    (f_0, (f_{i+1})_i) \mapsto (f_1, (f_{i+2})_{i}),
    (f_0, [\:]) \mapsto *.
$
Lastly, we give the evaluation map
\begin{align*}
    \ev \colon \underline{\Alg}(A,B) \times A &\to B\\
    ((f_0, (f_{i+1})_{i\in I}) , a) &\mapsto f_0(a).
\end{align*}
The evaluation map is measuring from $A$ to $B$ by $\underline{\Alg}(A,B)$ by construction.
To verify $\m_C(A,B) \cong \CoAlg(C, \underline{\Alg}(A,B))$, we refer the reader to the next section.

Next we would like to determine the measuring tensor $C \triangleright A$. 
By \cref{thm: measuring tensor as coequalizer}, $\m_C(A,B) \cong \Alg(C\triangleright A, B)$, where $C \triangleright A$ has as underlying set
$$
\Fr(C\times A)/\sim  \:\:= \N \times ( (C \times A) + \{e_{C\times A}\} ) / \sim. 
$$
Here the equivalence relation $\sim$ is generated by
\begin{align*}
    (i, c, 0_A) &\sim (i, e_{C\times A}) \\
    (i, c, \alpha(a))  &\sim (i, e_{C\times A}) \text{ if } \chi(c) = *\\
    (i, c, \alpha(a)) &\sim (i + 1, c', a) \text{ if } \chi(c) = c'.
\end{align*}
The algebra structure is induced by the function
\begin{align*}
    1 + \Fr(C\times A)&\to \Fr(C\times A)\\
    * &\mapsto (0, e_{C\times A})\\
    (i,e_{C\times A}) &\mapsto (i+1 ,e_{C\times A})\\
    (i,c,a ) &\mapsto (i+1 ,c,a)
\end{align*}
which respects the quotient map $\Fr(C\times A) \to \Fr(C\times A) / \sim$.
To verify $\m_C(A,B) \cong \Alg(C\triangleright A, B)$ we again refer the reader to the upcoming section.

Finally, we define the convolution algebra $[C,B]$. Its algebra structure is given by
\begin{align*}
    1 + [C,B] &\to [C,B] \\
    * &\mapsto (c \mapsto 0_B)\\
    f &\mapsto 
    \left(
    c \mapsto 
    \begin{cases}
        \beta(f(c')) \text{ if } \chi(c) =c'\\
        0_A \text{ if } \chi(c) = *
    \end{cases}
    \right).
\end{align*}

We would like to calculate $\underline{\Alg}(\n, B)$ for arbitrary $B \in \Alg$.
Our aim is to leverage the isomorphism
$$
\CoAlg(C, \underline{\Alg}(\n, B)) \cong \m_C(\n, B) \cong \Alg(\n, [C,B]).
$$
To do so, we make an observation about $\Alg(\n, Z)$ for an arbitrary algebra $(Z,\zeta) \in \Alg$.
Since $\n$ is preinitial, any morphism out of $\n$ is unique.
So, $\Alg(\n, Z)$ has at most one element.
The question now becomes if there is a condition on $Z$ for $\Alg(\n, Z)$ to be inhabited.
We claim this condition is
\begin{equation}\label{eq:natcondition}
    \zeta(i_Z) = \zeta((\min(n-1, i))_{Z}) \text { for all } i \in \N.
\end{equation}
Now that we have an indication on whether $\Alg(\n, Z)$ is inhabited for arbitrary $Z \in \Alg$, we can focus on $Z = [C,B]$.
Observing the algebra structure on $[C,B]$, we see the morphism $-_{[C,B]}$ is given by
\begin{align*}
    -_{[C,B]} : \N &\to [C,B] \\
    0 &\mapsto (c \mapsto 0_B)\\
    i+1 &\mapsto 
    \left(
        c \mapsto
        \begin{cases}
            \beta(i_{[C,B]}(c')) \text{ if } \chi(c) = c' \\
            0_B \text{ if } \chi(c) = *
        \end{cases}
    \right).
\end{align*}
Denoting the algebra structure on $[C,B]$ by $\alpha$, we would like to know under which conditions on $C$ and $B$ $\alpha(i_{[C,B]}) = \alpha((\min(n-1)(i))_{[C,B]})$.
The above condition can be satisfied if $\ind{c} \leq n$ for all $c\in C$ or if condition \ref{eq:natcondition} holds for $B$.
%This may need a bit more clarification...
So, we have the following:
$$
\Alg(\n, [C,B]) \cong 
\begin{cases}
    \{*\} \text{ if } \beta(i_B) = \beta((\min(n-1)(i))_B) \text{ for all } i \in \N\\
    \{*\} \text{ if } \ind{c} \leq n \text{ for all } c \in C\\
    \emptyset \text{ otherwise}.
\end{cases}
$$
Now we can make use of the isomorphism $\CoAlg(C, \underline{\Alg}(\n, B)) \cong \Alg(\n, [C,B])$
We observe that in the case of $\beta(i_B) = \beta((\min(n-1)(i))_B)$,  $\underline{\Alg}(\n, B)$ has the universal property of the terminal coalgebra $\eN$.
Whenever this is not the case, $\underline{\Alg}(\n, B)$ has the universal property of $\n^\circ \in \CoAlg$.
%The universal property of X^*_n \in \CoAlg might need some explanation.
We conclude
$$
\underline{\Alg}(\n, B) = 
\begin{cases}
    \N_{\infty} \text{ if }  \beta(i_B) = \beta((\min(n-1)(i))_B) \text{ for all } i \in \N\\
    \n^\circ \text{ otherwise}.
\end{cases}
$$

\subsubsection{$C$-Initial Algebras}
In this section we would like to give a proof of the fact that $\n$ is the terminal $\n^\circ$-initial algebra.
The proof is different from the simpler proof of \cref{example: terminal C-initial algebra}, with the main advantage that it does not use the law of excluded middle.
Where in the previous examples we could make a complete classification of $C$-initial algebras, here we cannot.
The algebra structures involved are more complex, and hence we will restrict ourselves to the most useful cases.

We first show there exists a morphism $A \to \n$ for any $\n^\circ$-initial algebra $A$ using induction.
After that, we will show the morphism is unique.
We start with a lemma which allows us to use induction later on.
\begin{lemma}
    Let $A$ be a $\n^\circ$-initial algebra. Then $A$ is also $(\n-1)^\circ$ initial.
\end{lemma}
\begin{proof}
    Consider the coalgebra morphism $m \colon (\n-1)^\circ \to \n^\circ, i \mapsto i$.
    This induces a morphism
    \begin{align*}
        m\triangleright A \colon (\n-1)^\circ \triangleright A &\to  \n^\circ \triangleright A \\
        [n,i,a] &\mapsto [n,i,a].
    \end{align*}
    This morphism is monomorphic by definition.
    Since $A$ is $\n^\circ$-initial, we know $\n \triangleright A$ is an initial object.
    This means $m\triangleright A $ is a monomorphism into the initial object, hence an isomorphism.
\end{proof}
As an immediate consequence we have the following corollary.
\begin{corollary}
    Let $A$ be a $\n^\circ$-initial algebra. Then $A$ is also $\k^\circ$ initial for all $k \leq n$.
\end{corollary}

The next lemma is a technical lemma which we will be able to leverage during the induction step.
\begin{lemma}
    Let $A$ be a $\n^\circ$-initial algebra and let $\phi \colon \n^\circ \times A \to \N$ be the unique measuring from $A$ to $\N$ by $\n^\circ$.
    Then for all $i \in \n^\circ$ and $0 \leq j \leq n-i$ we have
    $$
    \phi(i,a) = \phi(i, \phi(i+j, a)_A).
    $$
\end{lemma}
\begin{proof}
    Let $k\leq n$.
    Define the coalgebra morphism
    \begin{align*}
        p_j \colon \k^\circ &\to \n^\circ \times \n^\circ \\
        i &\mapsto (i,i+j)
    \end{align*}
    for all $0 \leq j \leq n-k$.
    Consider the composition of measurings $\phi \circ (\id \times -_{A} \circ \phi) \in \m_{\n^\circ \times \n^\circ}(A,\N)$.
    We can precompose this measuring with the coalgebra morphism $p_j$ to obtain a measuring
    $$
    \k^\circ \times A \xrightarrow{p_j \times \id_A} \n^\circ \times \n^\circ \times A \xrightarrow{\phi \circ (\id \times -_{A} \circ \phi)} \N.
    $$
    We also have the coalgebra morphism $f\colon \k^\circ  \to \n^\circ, i \mapsto i$ which we can precompose with $\phi$.
    This gives us a measuring
    $$
    \k^\circ \times A \xrightarrow{f \times \id_A} \n^\circ \times A \xrightarrow{\phi} \N.
    $$  
    Since $A$ is also $\k^\circ$-initial, we know these measurings must coincide.
    Hence we can state
    $
    \phi(i,a) = \phi(i,-_{A} \circ \phi(i+j,a))
    $
    for all $i \in \k^\circ$ and $0 \leq j \leq n-k$.
    The only restriction placed on $k$ was that $k \leq n$.
    Iterating over all $0 \leq k \leq n$, we arrive at the desired result
    $$
    \phi(i,a) = \phi(i,-_{A}\circ \phi(i+j,a))
    $$
    for all $0 \leq i \leq n$ and $0 \leq j \leq n-i$.
\end{proof}

\begin{corollary}\label{lem:bound}
    Let $A$ be $\n^\circ$-initial and let $\phi$ be the unique measuring to $\N$.
    For all $i \in \n^\circ$ and $a \in A$, $\phi(i,a) \leq i$.
\end{corollary}
\begin{proof}
    Since $\phi(i,a) = \phi(i, (\phi(i, a))_A)$ by the previous lemma, we know
    $
    [0,i,a] = [0, i, (\phi(i, a))_A]
    $
    in $\n^\circ \triangleright A$.
    We also have $[\phi(i,a),z] = [0,i,a]$, and by definition of the equivalence relation $\sim$, we have
    $
    [0, i, (\phi(i, a))_A] = [\min(i, \phi(i, a)), z].
    $
    From this we conclude
    $$
    \phi(i,a) = \min(i, \phi(i, a)),
    $$
    hence $\phi(i,a) \leq i$ for all $i \in \n^\circ$.
\end{proof}
Now we are ready to define the a family of functions which will culminate in an algebra morphism $A \to \n$ for any $\n^\circ$-initial algebra $A$.

\begin{proposition}\label{prop:Atonexists}
Let $A$ be an  $\n^\circ$-initial algebra.
Let $\varphi\colon \n^\circ\times A\rightarrow \N$ be the unique measuring. The functions $\phi_k \colon A \to \k$ 
are total algebra morphisms for all $0 \leq k \leq n$.
\end{proposition}

\begin{proof}
    We proceed by induction over $k$.
    For $k = 0$, we have that $\k \cong 1$, the terminal object in the category of algebras.
    Hence $\phi_0 : A \to 1$ is an algebra morphism.
    For the inductive step, assume $\phi_{k-1} : A \to \k-1$ is an algebra morphism.
    We wish to show $\alpha_k(\phi_k(a)) = \phi_k(\alpha(a))$ for all $a \in A$.
    If $\phi_k(a) = 0$, we know
    \begin{align*}
         \phi(k,\alpha(a)) & = 1 + \phi(k-1,a) \\
         &= 1 + \phi(k-1, (\phi(k,a))_A) \\
         & = 1 + \phi(k-1, 0_A) \\
         &  = 1 = \alpha_k(0) \\
         &= \alpha_k(\phi(k,a)).
    \end{align*}
    If $\phi_k(a) \geq 0$ we can make the following deduction
    \begin{align*}
        \alpha_k(\phi_k(a)) 
        &= \alpha_k(\phi(k, a))\\
        &= 1 + \alpha_{k-1}(\phi(k,a) - 1)\\
        &= 1 + \alpha_{k-1}(\phi(k,(\phi(k,a)) - 1)_A)\\
        &= 1 + \alpha_{k-1}(\phi(k-1,(\phi(k,a) - 1)_A))\\
        &= 1 + \phi(k-1,  \phi(k,a)_A)\\
        &= 1 + \phi(k-1, a)\\
        &= \phi(k,\alpha(a)).
    \end{align*}
\end{proof}

In particular, this lemma shows $\phi_n \colon A \to \n$ is an algebra morphism.
It still remains to show this algebra morphism is unique.
\begin{lemma}\label{lem:Atonunique}
    For any $\n^\circ$-initial algebra $A$, there exist at most one algebra morphism $A \to \n$.
\end{lemma}
\begin{proof}
    By \cref{prop:map from initial alg to dual}, there exists a unique morphism $A \to [\n^\circ, \N]$ for any $\n^\circ$-initial algebra $A$.
    Since $\n$ is $\n^\circ$-initial we know there exists a unique morphism $m \colon \n \rightarrowtail [\n^\circ, \N]$.
    Upon closer inspection $m$ turns out to be a monomorphism.
    Given any two morphisms $f,g \colon A \to \n$, we can draw the following diagram
    \[\begin{tikzcd}
        A & {[\n^\circ, \N]} \\
        & \n
        \arrow["{!}", from=1-1, to=1-2]
        \arrow["f"', shift right, from=1-1, to=2-2]
        \arrow["g", shift left, from=1-1, to=2-2]
        \arrow["m"', tail, from=2-2, to=1-2].
    \end{tikzcd}\]
    Since the morphism $A \to [\n^\circ, \N]$ is unique, we know the composites $m \circ f = m \circ g$, and by $m$ being mono we conclude $f = g$.
    Hence, there can be at most one algebra morphism from a $\n^\circ$-initial algebra $A$ to $\n$.
\end{proof}

Putting all the above together, we arrive at the following result.

\begin{theorem}
    The algebra $\n$ is the terminal $\n^\circ$-initial algebra.
\end{theorem}

\begin{proof}
    Given an $\n^\circ$ initial algebra $A$, by Proposition \ref{prop:Atonexists} we obtain an algebra morphism $\phi_n \colon A \to \n$.
    By Lemma \ref{lem:Atonunique} it is unique.
    We conclude $\n$ is the terminal $\n^\circ$-initial algebra.
\end{proof}

\subsection{The List Type} \label{subsec:list}
In \cref{subsec:monoid} we have seen how introducing a monoidal structure allows measurings to introduce some ``twist''.
On the other hand in \cref{subsec:naturals} we have seen how measurings are able to control ``shape'' in a sense.
In this section we are combining the two, which comes together nicely in the list type.
Given a monoid $(M, \bullet, e)$, we can consider the type which contains all possible finite lists with elements in $M$.
Lists are used ubiquitously throughout computer science and exhibits the flexibility gained when considering not just algebra morphisms, but measurings as well.

To study the list type we must find a functor which has it as an initial algebra.
This functor exists and is defined as follows.
Let $(M, \bullet, e)$ be a commutative monoid in $\Set$ and consider the functor
\begin{align*}
    F \colon \Set &\longrightarrow \Set \\
    X &\longmapsto 1+X\times M\\
    f &\longmapsto \id_1 + f \times \id_M.
\end{align*}
The lax symmetric monoidal structure is given by
\begin{align*}
    \nabla_{X,Y} \colon (1 + M \times X) \times (1 + M \times Y) &\longrightarrow 1 + M \times X\times Y\\
    (*,(m,y)) &\longmapsto *\\
    ((m,x),*) &\longmapsto *\\
    ((m,x),(m',y)) &\longmapsto (m \bullet m' , x,y)
\end{align*}
and
$
    \eta \colon 1 \to 1 + M ,
    * \mapsto e.
$
Since we will use it explicitly later, we remark the lax closed structure is given by
\begin{align*}
    \widetilde{\nabla}_{X,Y}\colon 1+ M \times [X,Y] &\longrightarrow [1+ M \times X, 1+ M \times Y] \\
    * &\longmapsto \const_*\\
    (m, f) &\longmapsto \widetilde{\nabla}_{X,Y}(m,f),
\end{align*}
where
\begin{align*}
    \widetilde{\nabla}_{X,Y}(m,f) \colon 1 + M\times X &\longrightarrow 1 + M \times Y \\
    * &\longmapsto *\\
    (m',x) &\longmapsto m\bullet m', f(x).
\end{align*}

The category of $F$-algebras has elements
$
    \alpha \colon 1+M\times A \to A
$
denoted $(A,\alpha)$. We will write $\alpha(*) = e_A$, thinking of $\alpha(*)$ as an ``empty element''.
Sometimes when considering an algebra we will speak of a ``list-like algebra''.
Morphisms $f\colon(A,\alpha) \to (B,\beta)$ are given by functions $ f \colon A \to B $
which make the following diagram commute:
\[\begin{tikzcd}
1+M\times A & 1+M\times B \\
A & B.
            \arrow["F(f)", from=1-1, to=1-2] 
\arrow["\alpha"',from=1-1, to=2-1] \arrow["\beta",from=1-2, to=2-2]
            \arrow["f", from=2-1, to=2-2]
\end{tikzcd}\]
The category of $F$-coalgebras has elements
$
    \chi \colon C \to 1+M\times C,
$
denoted $(C,\chi)$.
Sometimes when considering a coalgebra we will speak of a ``stream-like coalgebra''.
Morphisms $f\colon (C,\chi) \to (D,\delta)$ are given by functions $ f \colon C \to D $
which make the following diagram commute:
\[\begin{tikzcd}
1+M\times C & 1+M\times D \\
C & D.
            \arrow["F(f)", from=1-1, to=1-2] 
\arrow["\chi",from=2-1, to=1-1] \arrow["\delta"',from=2-2, to=1-2]
            \arrow["f", from=2-1, to=2-2]
\end{tikzcd}\]

\subsubsection{Initial and Terminal Objects}
\label{subsubsec:initial and terminal objects}
We denote the set of lists of length $n$ with elements in $M$ by $M^{n}$.
The underlying set of the initial algebra is given by
$$
\gls{List} \coloneqq\coprod_{i\in \N} M^{i}.
$$
Again borrowing notation from functional programming, we will write $\gls{EmptList} \in M^0$ for the empty list and given a list $\ell \in M^*$ and an element $m \in M$, appending $m$ to the list $\ell$ is denoted by $\gls{ConcList}$.
Moreover, we will call the element which was appended to the list most recently the first element, and conversely the element which was appended to the list earliest the last element.
The algebra structure of $M^*$ is given by 
\begin{align*}
    1 + M \times M^* &\longrightarrow M^* \\
    * &\longmapsto [\:]\\
    (m,\ell) &\longmapsto m:\ell.
\end{align*}
This also explains the term ``list-like algebra'', since every algebra $\alpha \colon 1 + M \times A \to A$ has an empty element $e_A = \alpha(*)$ and a way of appending an element $m \in M$ to a ``list-like'' $a \in A$.
To verify $M^*$ is indeed the initial algebra, let $(A, \alpha)$ be an arbitrary algebra.
The unique algebra morphism $M^* \to A$, denoted $\fromI{A} \colon M^* \to A$ is forced by the commutative square
\[\begin{tikzcd}
1+M \times M^* & 1+M \times A \\
M^* & A
            \arrow["F(f)", from=1-1, to=1-2] 
\arrow[from=1-1, to=2-1] \arrow["\alpha",from=1-2, to=2-2]
            \arrow["\fromI{A}", from=2-1, to=2-2],
\end{tikzcd}\]
and defined as
\begin{align*}
    \fromI{A}\colon M^* &\to A \\
    [\:] &\mapsto e_A\\
    m:\ell &\mapsto \alpha(m,f(\ell)).
\end{align*}

Preinitial algebras correspond to equivalence relations $\sim$ on $M^*$ which satisfy
\[
\ell \sim \ell' \Rightarrow m:\ell \sim m:\ell' \text{ for all } m \in M,
\]
by \cref{lem: quotient algebras}.
An example is the equivalence relation which identifies all lists with length greater than $n$ with the list of their last $n$ elements.
This results in a preinitial algebra which we will denote by $\gls{Listn} \coloneqq \coprod_{0\leq i\leq n} M^{i}$.
Its algebra structure is given by
\begin{align*}
    1 + M \times M^{\leq n} &\to M^{\leq n} \\
    * &\mapsto [\:]\\
    (m,\ell) &\mapsto \take(n)(m:\ell),
\end{align*}
where $\take$ takes the first $n$ elements of a list.
Another example is the equivalence relation generated by
$$
m:\ell \sim \ell \text{ if } m \not\in M',
$$
where $M' \subseteq M$.
This results in the preinitial algebra $(M')^*$ for any $M' \subseteq M$.
The unique algebra morphism $M^* \to (M')^*$ corresponds to filtering the elements of a list.

One can check that in the above cases the unique morphisms $M^* \to M^{\leq n}$ and $M^* \to (M')^*$ are epic.
The above two can be combined to yield filtered lists of length at most $n$, leading to even more preinitial algebras.
However, keep in mind this list is not exhaustive, and that there are even more preinitial algebras.

The terminal coalgebra is given by lists in $M$ of finite and infinite length and is denoted $M^{\leq \infty}$.
Formally, it is given by $M^{\leq \infty} = M^* + (\prod_{n \in \N}M)$.
To verify this is indeed the terminal coalgebra, let $(C, \chi)$ be an arbitrary coalgebra.
The unique algebra morphism $C \to M^{\leq \infty}$, denoted $\toT{C}\colon C \to M^{\leq \infty}$, is forced by the commutative square
\[\begin{tikzcd}
1+M \times C & 1+M \times M^{\leq \infty} \\
C & M^{\leq \infty}
            \arrow["F(f)", from=1-1, to=1-2] 
\arrow["\chi", from=2-1, to=1-1] \arrow["",from=2-2, to=1-2]
            \arrow["\toT{C}", from=2-1, to=2-2],
\end{tikzcd}\]
and defined as
\begin{align*}
    \toT{C}\colon C &\longrightarrow M^{\leq \infty} \\
    c &\longmapsto 
    \begin{cases}
        [\:] \text{ if } \chi(c) = *\\
        m:f(c') \text{ if } \chi(c) = (m,c').\\
    \end{cases}
\end{align*}

Subterminal coalgebras correspond to subsets $X \subseteq M^{\leq \infty}$ satisfying
$
m:\ell \in X \Rightarrow \ell \in X
$
by \cref{lem:char-sub}.
As with with the case of the natural numbers, some of the subterminal coalgebras share the same underlying set as preinitial algebras.
We will use $(-)^\circ$ to denote the coalgebra whenever it might be ambiguous.
Some examples are
$\emptyset,
{(M^{\leq 0})}^\circ \cong 1,
{(M^{\leq n})}^\circ,
{M^*}^\circ,
\prod_{n \in \N} M$ and 
$(M')^{\leq \infty}, M' \subseteq M$,
with their coalgebra structure inherited from the coalgebra structure on $M^{\leq \infty}$.
One can verify the all maps from subterminal coalgebra into $M^{\leq \infty}$ are injective.

\subsubsection{Functions as Algebra Morphisms}\label{basiclistfunctions}
In this section we would like to demonstrate the power of the algebraic approach to types.
We wish to give the definitions of some basic functions frequently seen in computer science as algebra morphisms.

\begin{example}
The first function is the length function $\len$, which takes a list and returns its length.
We want to word it in the language of algebras, so we first define an algebra
\begin{align*}
    1+M \times \N &\to \N\\
    * &\mapsto 0\\
    (m, n) &\mapsto n+1.
\end{align*}
The function $\gls{len}$ on lists $\ell \in M^*$ is then given by the unique algebra morphism $M^* \to \N$.
\end{example}

\begin{example}
The second function considered takes two lists and concatenates them.
Classically, this would be denoted as a function $M^* \times M^* \to M^*$, but to frame it as an algebra morphism, we consider it as a function $M^* \to [M^*,M^*]$.
To this end, consider the the algebra
\begin{align*}
    \alpha \colon 1+M \times [M^*, M^*] &\to [M^*, M^*]\\
    * &\mapsto \id\\
    (m,f) &\mapsto (\ell \mapsto m:f(\ell)).
\end{align*}
The concatenation function is denoted by $(++)$ and is defined uniquely by the following diagram
\[\begin{tikzcd}
1+M\times M^* & [2em]1+M \times [M^*, M^*] \\
M^* & [M^*, M^*]
            \arrow["F(\:++\:)", from=1-1, to=1-2] 
\arrow["",from=1-1, to=2-1] \arrow["\alpha",from=1-2, to=2-2]
            \arrow["(++)", from=2-1, to=2-2] 
            [M^*, M^*]
\end{tikzcd}\]
where uniqueness follows from the fact $M^*$ is the initial algebra.
Given $\ell \in M^*$, the function $(++)(\ell) \in [M^*,M^*]$ is the function which takes a list $\ell'$ and concatenates it with $\ell$,
resulting in the list $(++)(\ell)(\ell') = \ell ++\: \ell'$.
\end{example}

\begin{example}
The next function is $\head$, which is a partially defined function which attempts to extract the first element of a list-like.
We define the algebra
\begin{align*}
    \alpha \colon 1 + M \times (M + \{\bot\})&\to M + \{\bot\}\\
    * &\mapsto \bot\\
    (m,\ell) &\mapsto m.
\end{align*}
The function $\head$ is then defined as the following algebra morphism
\begin{align*}
  \head \colon M^* & \to (M + \{\bot\}, \alpha) \\
    [\:] &\mapsto \bot \\
    m:\ell  & \mapsto m.  
\end{align*}
\end{example}

\begin{example}
The next function is $\take(n)$, which is a function that takes the first $n \in \eN$ elements of a list.
We have already defined this function ``by hand'' when defining the algebra structure on $M^{\leq n}$, but we would like view it from our algebraic perspective.
As maybe suspected, it is given by the unique function
\begin{align*}
    \take(n) \colon M^* &\to  M^{\leq n}\\
    \ell &\mapsto \take(n)(\ell).
\end{align*}
\end{example}

\begin{example}
The $\filter$ function takes a predicate $p \in P_M \coloneqq [M, \{\top, \bot\}]$ and a list $\ell \in M^*$ and returns only the elements in $\ell$ which satisfy $p$.
Classically, this would be written as $\filter\colon M^* \times P_M \to M^*$, but in our case we want to consider it as a function
$\filter\colon M^* \to [P_M, M^*]$.
First, we define an algebra structure on $ [P_M, M^*]$ by
\begin{align*}
    \beta \colon 1+M\times  [P_M, M^*] &\to  [P_M, M^*] \\
    * &\mapsto \const_{[\:]}\\
    (m,f) &\mapsto \left(
        p \mapsto
        \begin{cases}
            m:f(p) \text{ if } p(m)\\
            f(p) \text{ otherwise}
        \end{cases}
        \right).
\end{align*}
The function $\filter$ is given by the unique algebra morphism $M^* \to  [P_M, M^*]$.
Another way to construct the filter function using a predicate $p \in P_M$ is to compute the set $M_p = \{m \in M \mid p(m) = \top\}$.
Since $M_p \subset M$, $M_p^*$ has the algebra structure corresponding to filtering.
The function $\filter(p)$ is given by the unique function $\filter(p) \colon M^* \to M_p^*$.
\end{example}

Notice that all functions above can be described as folds.
Constructing the target algebra $(A,\alpha)$ as done above corresponds to finding the accumulator and the binary operator of the fold. 

\subsubsection{Measuring}
Let $(A,\alpha), (B,\beta) \in \Alg$ and $(C,\chi) \in \CoAlg$.
A {measuring from $A$ to $B$ by $C$} is a function $\phi\colon C \times A \to B$ satisfying
\begin{enumerate}
    \item $\phi_c(e_A) = e_B$ for all $c \in C$;
    \item $\phi_c(\alpha(m,a)) = e_B$ if $\chi(c) = *$;
    \item $\phi_c(\alpha(m,a)) = \beta(m' \bullet m, \phi_{c'}(a))$ if $\chi(c) = (m',c')$.
\end{enumerate}
Intuitively, a measuring is a function that takes a stream-like coalgebra and a list-like algebra and combines their elements using the monoidal structure on $M$, then wraps them up in a new list-like.
We see the coalgebra introduces some element wise ``twist'', as well as limits the ``shape'' of the list-like.

\begin{example}\label{ex:meauringtoX^*}
Give any commutative monoid $(M,e,\bullet)$, we remark there does not exist a total algebra morphism $f\colon M^{\leq n} \to M^*$, since for a list $\ell \in M^{\leq n}$ of length $n$, we must have that
$$
m : f(\take(n-1)(\ell)) = f(m:\take(n-1)(\ell)) = f(m:\ell) = m:f(\ell),
$$
which cannot be the case.
Note that in the above computation we have used $\take(n)(m:\ell) = m:\take(n-1)(\ell)$.
The problem here is that we are disregarding the last element of the list, where an algebra morphism into $M^*$ does need that information.

However, we are able to define a measuring from $M^{\leq n}$ to $M^*$.
We define the measuring
\begin{align*}
    \phi \colon ({M^{\leq n}})^\circ \times M^{\leq n} &\to M^* \\
    ([\:], \ell) &\mapsto [\:]\\
    (\ell', [\:]) &\mapsto [\:]\\
    (m':\ell',m:\ell) &\mapsto (m' \bullet m) : \phi(\ell', \ell).
\end{align*}
We note that checking this is a measuring can be done by comparing the definition of measuring and that of $\phi$.
This example is valid because the coalgebra ${(M^{\leq n})}^\circ$ limits up to which point we are considering elements of the algebra $M^{\leq n}$.
The problem we had earlier, caused by disregarding the last element of a list due to the maximum length of the list, is mitigated by the introduction of the coalgebra ${(M^{\leq n})}^\circ$.
Since the elements of ${(M^{\leq n})}^\circ$ all have length at most $n$, we can truly disregard the last element of a list whenever its length exceeds $n$.
\end{example}

\begin{example}
Next, we investigate another interesting preinitial algebra given by $(M')^*$ for some subset $M' \subseteq M$.
We would like to find a measuring $\phi\colon C \times (M')^* \to B$ for some coalgebra $C \in \CoAlg$ and some algebra $B \in \Alg$.
This measuring must satisfy
$$
\phi_c(m:\ell) = 
\begin{cases}
    \beta((m\bullet m'), \phi_{c'}(\ell)) \text{ if } \chi(c) = (m',c') \\
    e_B \text{ otherwise}.
\end{cases}
$$
By the coalgebra structure on $(M')^*$, we have that for any $m \not\in M'$, $m:\ell = \ell$.
In the case that $\chi(c) = (m',c')$ and $m \not\in M'$, the above condition becomes
$
\phi_c(\ell) = \phi_c(m:\ell) = \beta((m\bullet m'), \phi_{c'}(\ell)),
$
giving us a contradiction since we need $\phi_c(\ell) = \beta((m\bullet m'), \phi_{c'}(\ell))$ for all $m \not\in M'$.
The problem here is a loss of information. 
This loss stems from the fact that $m:m:\dots:m:\ell = \ell$ whenever $m \not\in M'$,
whereas the measuring does need to take into account the added elements $m$.
\end{example}

\begin{example}
    Using measurings, we are also able to define a more flexible version of the $\map$ function.
    Traditionally, $\map$ takes a function $f$ and a list and applies $f$ to all elements of the list element wise.
    Using measurings, we can actually specify a list of functions which gets applied to elements of a list in a pairwise fashion.
    For each $m \in M$, define the function $r_m \colon M \to M$ which multiplies an element of $M$ by $m$ from the right.
    We define the algebra
    \begin{align*}
        1+ M \times [M,M]^{\leq \infty} &\longrightarrow [M,M]^{\leq \infty}\\
        * &\longmapsto r_e\\
        (m,\ell) &\longmapsto r_m : \ell.
    \end{align*}
    We can now define the measuring
    \begin{align*}
        \phi \colon {M^*}^\circ \times  [M,M]^{\leq \infty} &\to M^* \\
        ([\:], \ell) &\mapsto [\:]\\
        (\ell, [\:]) &\mapsto [\:]\\
        (m:\ell, f:k) &\mapsto f(m):\phi(\ell, k).
    \end{align*}
    We verify this is a measuring by checking
    $
    \phi(m':\ell, r_{m}:k) = r_{m}(m'):\phi(\ell, k) = (m'\bullet m) :\phi(\ell,k).
    $
    Note that we really need to define an algebra structure on $[M,M]^{\leq \infty}$ instead of a coalgebra structure, since there is no obvious way to extract $m \in M$ from a function $f \colon M \to M$ in such a way that the conditions on a measuring are satisfied.
\end{example}

\subsubsection{Free and Cofree Functors}
The free functor is defined by mapping $X \in \Set$ to
\begin{align*}
    \alpha_{\Fr}\colon 1+M \times M^* \times (X + \{e_X\}) &\to M^* \times (X + \{e_X\}) \\
    * &\mapsto ([\:], e_X)\\
    (m,\ell,e_X) &\mapsto (m:\ell,e_X)\\
    (m,\ell,x) &\mapsto (m:\ell,x).
\end{align*}
The behavior of $\Fr$ on functions $f\colon X \to Y$ is forced and given by
\begin{align*}
    \Fr(f)\colon M^* \times (X + \{e_X\}) &\to M^* \times (Y + \{e_Y\}) \\
    (\ell,x) &\mapsto (\ell,f(x))\\
    (\ell, e_X) &\mapsto (\ell, e_Y).
\end{align*}
Checking this is a left adjoint to the forgetful functor can be done by constructing the natural bijection
$$
[X,A] \cong \Alg(\Fr(X),A).
$$
The bijection is given by
$
    [X,A] \to \Alg(\Fr(X),A),
    f \mapsto \widetilde f,
$
where
\begin{align*}
    \widetilde f\colon M^* \times (X + \{e_X\}) &\to Y \\
    ([\:], e_X) &\mapsto e_Y\\
    ([\:], x) &\mapsto f(x)\\
    (m:\ell, e_X) &\mapsto \beta(m, \widetilde f(\ell,e_X))\\
    (m:\ell, x) &\mapsto \beta(m, \widetilde f(\ell,x)).
\end{align*}
The inverse of the bijection $[X,A] \to \Alg(\Fr(X),A)$ is given by restriction to $X$.
Informally, $\widetilde{f}$ takes a list $\ell \in M^*$ and an element $x \in X$ and appends $\ell$ to $f(x)$ in $Y$.
A quick check will also show this bijection is natural.
The unit and counit are given by
$
    \eta^{\Fr}_X\colon X \to M^* \times (X + \{e_X\}) ,
    x \mapsto ([\:],x)
$
for any set $X$,
and
\begin{align*}
    \epsilon^{\Fr}_A\colon M^* \times (A + \{e_A\}) &\to A \\
    ([\:], e_A) &\mapsto \alpha(*)\\
    ([\:], a) &\mapsto a\\
    (m:\ell, a) &\mapsto \alpha(m, \epsilon^{\Fr}_A(\ell, a)),
\end{align*}
for an algebra $A$.

The cofree functor is defined by mapping $X \in \Set$ to
\begin{align*}
    \delta_{\Cof} \colon X \times (M \times X)^{\leq \infty} &\to 1 + M \times X \times (M \times X)^{\leq \infty}\\
    (x, [\:]) &\mapsto *\\
    (x, (m,y):\ell) &\mapsto (m,y,\ell).
\end{align*}
The behavior of $\Cof$ on functions $f\colon X \to Y$ is forced and given by applying $f$ to all elements of $X$ in any element of $X \times (M \times X)^{\leq \infty}$.
Checking this is a right adjoint to the forgetful functor can be done by constructing the natural bijection
$$
[C,Y] \cong \CoAlg(C, \Cof(Y)),
$$
The bijection is given by
$
    [C,Y] \to \CoAlg(C,\Cof(Y)),
    f \mapsto \widetilde f,
$
where
\begin{align*}
    \widetilde f = \langle \widetilde f_0,\widetilde f_1 \rangle \colon C &\to Y \times (M \times Y)^{\leq \infty} \\
    c &\mapsto
    \begin{cases}
        (f(c), (m,f(c')):\widetilde f_1(c')) \text{ if } \chi(c) = (m,c')\\
        (f(c), [\:]) \text{ if } \chi(c) = *.
    \end{cases}
\end{align*}
Notice the function $\widetilde f_0 = f$, and hence this also gives the inverse for the bijection.
Verifying $\widetilde f$ is an coalgebra morphism is left to the reader.
Informally, $\widetilde f$ takes a stream-like $c \in C$ and applies $f$ to all its elements.
The unit and counit are given by
$
    \eta^{\Cof}_C \colon C \to C \times (M \times C)^{\leq \infty} ,
    c \mapsto (c,[\:])
$ for any coalgebra $C$,
and
$
    \epsilon^{\Cof}_X \colon X \times (M \times X)^{\leq \infty} \to x ,
    (x, \ell) \mapsto x,
$ for any set $X$.

\subsubsection{Representing Objects}
Here, we will give full details when constructing the representing objects.
We will finally see which roles the free and cofree functor play in this construction.

First we wish to construct the universal measuring coalgebra $\underline{\Alg}(A,B)$.
The representing object $\underline{\Alg}(A,B)$ is given as the equalizer of the following parallel morphisms
\[\begin{tikzcd}
{[A,B] \times (M \times [A,B])^{\leq \infty}} & {[1+M \times A,B] \times (M \times [1+M \times A,B])^{\leq \infty}}
	%\arrow[dashed, "\operatorname{eq}", from=1-1, to=1-2]
	\arrow["{\Cof(\alpha^*)}", shift left, from=1-1, to=1-2]
	\arrow["\Psi"', shift right, from=1-1, to=1-2]
\end{tikzcd}\]
where $\Psi$ is the transpose of
\begin{multline*}
    \widetilde{\Psi}\colon [A,B] \times (M \times [A,B])^{\leq \infty} \xto{\chi_{\Cof}} 1+ M \times [A,B] \times  (M \times[A,B] )^{\leq \infty} \xto{F(\epsilon^{\Cof})}\\
    1+M \times [A,B] \xto{\widetilde{\nabla}_{A,B}} [1+M \times A,1+M \times B] \xto{\beta_*} [1+M\times A,B].
\end{multline*}
If we compute the composition we find
\begin{align*}
    \widetilde{\Psi} \colon [A,B] \times (M \times [A,B])^{\leq \infty}  &\to [1+M \times A,B] \\
    (f',[\:]) &\mapsto \const_{e_B}\\
    (f',(m,f):\ell) &\mapsto
    \begin{cases}
        * &\mapsto  e_B \\
        (m',a) &\mapsto \beta(m\bullet m', f(a)).
    \end{cases}
\end{align*}
We see $\widetilde{\Psi}$ only depends on the first element of the list $\ell$, so for brevity we may write $\widetilde{\Psi}((f', (m,f):\ell)) = \widetilde{\Psi}((m,f))$.
Taking its transpose $\Psi = \Cof(\widetilde{\Psi}) \circ \eta$ we find the explicit formula for $\Psi$ to be 
\begin{align*}
    \Psi \colon{[A,B] \times ([A,B] \times M)^{\leq \infty}} &\mapsto {[1+M \times A,B] \times ([1+M \times A,B] \times M)^{\leq \infty}}\\
    (f', [\:]) &\mapsto (\const_{e_B}, [\:]) \\
    (f', (m,f):\ell) &\mapsto ( \widetilde{\Phi}((m,f)), [\:] ).
\end{align*}
Taking the equalizer, we find the underlying set of $\underline{\Alg}(A,B)$ to be a subset of $[A,B] \times (M \times [A,B])^\infty$, given by
\begin{align*}
\{
(f_0, (m_i,f_{i+1})_{i\in I})
\mid
f_{i}(e_A) = e_B,\,
&(f_i\circ \alpha)(m,a)  = \beta( (m_i \bullet m), f_{i+1}(a) ), \\
& f_{\max}(\alpha(m,a)) = e_B
\text{ for all }  a \in A, m \in M
\},
\end{align*}
where $f_{\max}$ is the last function in the stream $(f_{i})_{i\in I}$ if the stream is finite.
Unpacking the above, we find $\underline{\Alg}(A,B)$ consists of two streams $(m_i)_{i \in I} \in M^{\leq \infty}$ and $(f_i)_{i \in I + 1} \in [A,B]^{\leq \infty}$, with the condition
$
f_i(\alpha(m,a)) = \beta((m_i \bullet m) , f_{i+1}(a)),
$
and the last element of the stream $(f_i)$ should be the function mapping all elements of the form $\alpha(m,a)$ to the element $e_B \in B$.
The coalgebra structure of $\underline{\Alg}(A,B)$ is given by shifting to the left, 
$
    (f_0, (m_i,f_{i+1})_i) \mapsto (m_0, (f_1, (m_{i+1},f_{i+2})_{i})),
    (f_0, [\:]) \mapsto *.
$
Lastly, we give the evaluation map
\begin{align*}
    \ev \colon \underline{\Alg}(A,B) \times A &\to B\\
    ((f_0, (m_i,f_{i+1})_{i\in I}) , a) &\mapsto f_0(a).
\end{align*}
The evaluation map is measuring from $A$ to $B$ by $\underline{\Alg}(A,B)$ by definition, making it an element of the category of measurings from $A$ to $B$.

To verify $\m_C(A,B) \cong \CoAlg(C, \underline{\Alg}(A,B))$, we explicitly construct the bijection.
It is given by the function
$
    \Xi \colon \m_C(A,B) \to \CoAlg(C, \underline{\Alg}(A,B))
$
where given $\phi\in \mu_C(A,B)$:
\begin{align*}
    \Xi(\phi)\colon C &\to \underline{\Alg}(A,B) \\
    c &\mapsto \left(\phi_c, \left(  \chi^i_0(c), \phi_{\chi^i_1(c)}  \right)_{1 \leq i \leq \len(c)}  \right),
\end{align*}
where we write $\chi = \langle \chi_0, \chi_1 \rangle \colon C \to 1 + M \times C$ and $\gls{len}(c) \in \eN$ is defined as the unique number such that $\chi^{\len(c)}(c) = *$.
First, we verify $\Xi(\phi)$ is well-defined.
This is the case since
$
\phi(\chi^{\len(c)}_1(c))(\alpha(m,a)) = e_B 
$
by definition of a measuring and $\len(c)$ and
$
\phi_{\chi^{i}_1(c)}(\alpha(m,a)) = \beta(m \bullet \chi^{i}_0(c), \phi_{\chi^{i+1}_1(c)}(a))
$
by definition of a measuring.
Second, we verify $\Xi(\phi)$ is a morphism of coalgebras.
To this end, we note
\[
\Xi(\phi) (\chi(c)) = \left(\phi_{\chi(c)}, \left(  \chi^{i+1}_0(c), \phi_{\chi^{i+1}_1(c)}  \right)_{2 \leq i \leq \len(c)}  \right),
\]
which agrees with the coalgebra structure on $\underline{\Alg}(A,B)$.

The inverse of the bijection $\Xi\colon \m_C(A,B) \to \CoAlg(C, \underline{\Alg}(A,B))$ is given by sending a coalgebra homomorphism $\widetilde \phi \colon C \to \underline{\Alg}(A,B)$
to the measuring
$
\Xi\inv(\widetilde{\phi})\colon (c,a) \mapsto (\pr_0 \circ \widetilde{\phi} (c))(a).
$
This is well-defined by the properties on $\underline{\Alg}(A,B)$.
The bijection $\Xi\colon \m_C(A,B) \to \CoAlg(C, \underline{\Alg}(A,B))$ is natural.
We conclude $\underline{\Alg}(A,B)$ represents $\m_{-}(A,B)$. 
By the above the evaluation map is also the terminal object in the category of measurings, since any measuring $\phi\colon C \times A \to B$ factors uniquely through $\ev\colon \underline{\Alg}(A,B) \times A \to B$.

Next we would like to compute $C \triangleright A$.
The measuring tensor $C \triangleright A$ is given by the coequalizer of
\[\begin{tikzcd}
    [column sep=tiny]
	{M^* \times (C \times (1+M \times A) + \{e_{C\times F(A)}\})} & [4em]{M^* \times (C \times A + \{e_{C\times A}\})},
	\arrow["\Psi"', shift right, from=1-1, to=1-2]
	\arrow["{\id \times (\id \times \alpha + \id)}", shift left, from=1-1, to=1-2]
	%\arrow[dashed, "\operatorname{coeq}", from=1-2, to=1-3]
\end{tikzcd}\]
where $\Psi$ is the transpose of
\begin{multline*}
    \widetilde{\Psi}\colon C \times 1 + M \times A \xto{\chi \times \id_{F(A)}} 1 + M \times C \times 1 + M \times A \xto{\nabla_{A,B}} 1 + M \times (C \times A) \xto{F(\eta^{\Fr})} \\
    1 + M \times (M^* \times (C \times A + \{e_{C \times A}\})) \xto{\alpha_{\Fr}} M^* \times (C \times A + \{e_{C \times A}\}).
\end{multline*}
If we compute the composition we find
\begin{align*}
    \widetilde{\Psi}\colon C \times 1 + M \times A &\to M^* \times (C \times A + \{e_{C \times A}\}) \\
    (c,*) &\mapsto \chi(c) \bullet ([\:], e_{C \times A})\\
    (c,(m,a)) &\mapsto
    \begin{cases}
        ([\:], e_{C \times A}) &\text{ if } \chi(c) = * \\
        ([m'\bullet m], (c', a)) &\text{ if }\chi(c) = (m',c').
    \end{cases}
\end{align*}
Taking its transpose $\Psi = \epsilon \circ \Fr(\widetilde{\Psi})$ we find the explicit formula for $\Psi$ to be 
\begin{align*}
    \Psi \colon	{M^* \times (C \times (1+M \times A) + \{e_{C\times F(A)}\})} &\to {M^* \times (C \times A + \{e_{C\times A}\})} \\
    (\ell, e_{C\times F(A)}) &\mapsto (\ell, e_{C\times A})\\
    (\ell, (c,*)) &\mapsto (\ell, e_{C\times A})
    \end{align*}
    and 
    \begin{align*}
    (\ell, (c,(m,a))) &\mapsto
    \begin{cases}
        (\ell, e_{C\times A} )&\text{ if } \chi(c) = * \\
        (\ell ++ [m'\bullet m], (c',a)) &\text{ if } \chi(c) = (m',c').
    \end{cases}
    \end{align*}
Now that we know explicitly which maps we want to coequalize, we can compute
$$
\Fr(C\times A)/\sim  \:\:= M^* \times ( (C \times A) + \{e_{C\times A}\} ) / \sim. 
$$
Here the equivalence relation $\sim$ is generated by
\begin{align*}
    (\ell, c, e_A) &\sim (\ell, e_{C\times A}) \\
    (\ell, c, \alpha(m,a))  &\sim (\ell, e_{C\times A}) \text{ if } \chi(c) = *\\
    (\ell, c, \alpha(m,a)) &\sim (\ell ++ [m \bullet m'], c', a) \text{ if } \chi(c) = (m',c').
\end{align*}
Its algebra structure is given by
\begin{align*}
    1 + M \times C \triangleright A &\to C \triangleright A\\
    * &\mapsto [[\:], e_{C \times A}]\\
    (m, [\ell, (c,a)]) &\mapsto [m:\ell, (c,a)].
\end{align*}
Intuitively, given an element $(\ell, c, a)$ the equivalence relation transfers the first elements of the stream-like $c$ and the list-like $a$ to the beginning of the list $\ell$ by combining them using the monoid structure on $M$.

To verify $\m_C(A,B) \cong \Alg(C\triangleright A, B)$, we explicitly construct the bijection.
It is given by
\begin{align*}
    \Xi\colon \m_C(A,B) &\to \Alg(C\triangleright A, B) \\
    \phi &\mapsto \Xi(\phi),
\end{align*}
where $\Xi(\phi)$ is given by
\begin{align*}
    \Xi(\phi)\colon C \triangleright A &\to B\\
    [([\:], e_{C \times A})] &\mapsto e_B\\
    [(m:\ell, e_{C \times A})] &\mapsto \beta(m, \Xi(\phi)([(\ell, e_{C \times A})]) )\\
    [([\:], (c,a))]          &\mapsto \phi(c,a)\\
    [(m:\ell, (c,a))]          &\mapsto \beta(m, \Xi(\phi)([(\ell, (c,a))])).
\end{align*}
To check $\Xi(\phi)$ is well-defined, one needs to perform a straightforward check using the fact that $\phi$ is a measuring.
To see $\Xi(\phi)$ is a morphism of algebras, we note this is the case by definition of $C \triangleright A$ and the definition of $\Xi(\phi)$.

The inverse of the bijection $\Xi\colon \m_C(A,B) \to \Alg(C\triangleright A, B)$ is given by sending a morphism $\widetilde \phi \colon C\triangleright A \to B$
to the measuring
$$
\Xi\inv(\widetilde \phi)\colon (c,a) \mapsto \widetilde \phi([([\:],(c,a))]),
$$
which is well-defined by the properties on $C\triangleright A$.
The bijection $ \Xi \colon \m_C(A,B) \to \Alg(C\triangleright A, B)$ is natural, which can be checked by a straightforward calculation.
We conclude $C\triangleright A$ represents $\m_{C}(A,-)$. 

Finally, we define the convolution algebra $[C,B]$ to have algebra structure given by
\begin{align*}
    1 + M \times [C,B] &\to [C,B] \\
    * &\mapsto (c \mapsto e_B)\\
    (m,f) &\mapsto 
    \left(
    c \mapsto 
    \begin{cases}
        \beta(m \bullet m', f(c')) \text{ if } \chi(c) =(m',c')\\
        e_A \text{ if } \chi(c) = *
    \end{cases}
    \right).
\end{align*}

\begin{example}
We would like to calculate $\underline{\Alg}(M^{\leq n}, B)$ for arbitrary $B \in \Alg$.
Our aim is to leverage the isomorphism
$$
\CoAlg(C, \underline{\Alg}(M^{\leq n}, B)) \cong \m_C(M^{\leq n}, B) \cong \Alg(M^{\leq n}, [C,B]).
$$
To do so, we make an observation about $\Alg(M^{\leq n}, Z)$ for an arbitrary algebra $(Z, \zeta) \in \Alg$.
Since $M^{\leq n}$ is preinitial, any morphism out of $M^{\leq n}$ is unique.
So, $\Alg(M^{\leq n}, Z)$ has at most one element.
The question now becomes if there is a condition on $Z$ for $\Alg(M^{\leq n}, Z)$ to be inhabited.
We claim this condition is
\begin{equation}\label{eq:condition}
    \zeta(m,\fromI{Z}(\ell)) = \zeta(m,\fromI{Z}(\take(n-1)(\ell))) \text { for all } (m,\ell) \in M\times M^{\leq n}.
\end{equation}
Notice that $\take(n-1)(\ell) = \ell$ for all lists of length less than $n$, so the only non-trivial case is when $\len(\ell) = n$.
If \cref{eq:condition} is satisfied, we claim the unique algebra morphism $M^{\leq n} \to Z$ is given by
$$
\fromI{Z} \circ \iota\colon M^{\leq n} \hookrightarrow M^* \to Z,
$$
where $\iota \in [M^{\leq n}, M^*]$ is an inclusion of sets, not an algebra morphism.
For brevity, we will write $\fromI{Z} \circ \iota = \fromI{Z}$ if the domain is understood.
To verify this is an algebra morphism, we must check it commutes with the algebra structures.
If we denote the algebra structure on $M^{\leq n}$ by $\alpha_n$, we need to check
$
\fromI{Z}(\alpha_n(m,\ell)) = \zeta(m, \fromI{Z}(\ell)),
$
since the case of the empty list is trivial.
Using \cref{eq:condition} and that $\fromI{Z}\colon M^* \to Z$ is an algebra morphism we can make the following deduction:
\begin{align*}
    (\fromI{Z}\circ \iota)(\alpha_n(m,\ell)) 
    &= (\fromI{Z} \circ \iota)(m:\take(n-1)(\ell)) \\
    &= \fromI{Z}(m:\take(n-1)(\ell)) \\
    &= \zeta(m, \fromI{Z}(\take(n-1)(\ell))) \\
    &= \zeta(m,\fromI{Z}(\ell))\\
    &= \zeta(m,(\fromI{Z} \circ \iota)(\ell)).
\end{align*}
Conversely, if there exists $\ell \in M^{\leq n}$ such that $\zeta(m,\fromI{Z}(\ell)) \neq \zeta(m,\fromI{Z}(\take(n-1)(\ell)))$ it is impossible to construct an algebra morphism $f\colon M^{\leq n} \to Z$.
This can be seen by trying to construct such a morphism $f$ and observing $f(\alpha_n(m,\ell)) \neq \zeta(m,f(\ell))$ in this case.

Now that we have an indication on whether $\Alg(M^{\leq n}, Z)$ is inhabited for arbitrary $Z \in \Alg$, we can focus in $Z = [C,B]$.
Observing the algebra structure on $[C,B]$, we see the morphism $\fromI{[C,B]}$ is given by
\begin{align*}
    \fromI{[C,B]} \colon M^* &\to [C,B] \\
    [\:] &\mapsto (c \mapsto e_B)\\
    m:\ell &\mapsto 
    \left(
        c \mapsto
        \begin{cases}
            \beta(m \bullet m', \fromI{[C,B]}(\ell)(c')) \text{ if } \chi(c) = (m',c') \\
            e_B \text{ if } \chi(c) = *.
        \end{cases}
    \right).
\end{align*}
Denoting the algebra structure on $[C,B]$ by $\alpha$, we would like to know under which condition on $C$ and $B$ $\alpha(m, \fromI{[C,B]}(\ell)) = \alpha(m,\fromI{[C,B]}(\take(n-1)(\ell)))$.
Unpacking the above, we find this is the case if
$$
\beta(m\bullet m', \fromI{[C,B]}(\ell)(c')) = \beta(m\bullet m', \fromI{[C,B]}(\take(n-1)(\ell))(c')) 
$$
for all $(m',c') \in \im(\chi)$.
By definition of $\fromI{[C,B]}$, the above condition can be satisfied if $\len(c) \leq n$ for all $c\in C$ or if condition \cref{eq:condition} holds for $B$.
So, we have the following:
$$
\Alg(M^{\leq n}, [C,B]) \cong 
\begin{cases}
    \{*\} \text{ if } \beta(m,\fromI{B}(\ell)) = \beta(m,\fromI{B}(\take(n-1)(\ell)))\,
    \\ \hspace{9em} \forall (m,\ell) \in M \times M^{\leq n}\\
    \{*\} \text{ if } \len(c) \leq n \text{ for all } c \in C\\
    \emptyset \text{ otherwise}.
\end{cases}
$$
Now we can make use the fact that $\CoAlg(C, \underline{\Alg}(M^{\leq n}, B)) \cong \Alg(M^{\leq n}, [C,B])$.
We observe that in the case of $\beta(m,\fromI{B}(\ell)) = \beta(m,\fromI{B}(\take(n-1)(\ell)))$, the object $\underline{\Alg}(M^{\leq n}, B)$ has the universal property of the terminal coalgebra $M^{\leq \infty}$.
Whenever this is not the case, $\underline{\Alg}(M^{\leq n}, B)$ has the universal property of $M^{\leq n} \in \CoAlg$.
%The universal property of X^*_n \in \CoAlg might need some explanation.
We conclude
$$
\underline{\Alg}(M^{\leq n}, B) = 
\begin{cases}
    M^{\leq \infty} \text{ if }  \beta(m,\fromI{B}(\ell)) = \beta(m,\fromI{B}(\take(n-1)(\ell))), \\ \hspace{10em} \forall (m,\ell) \in M \times M^{\leq n}\\
    M^{\leq n} \text{ otherwise}.
\end{cases}
$$
\end{example}

\subsubsection{$C$-Initial Algebras}
We have already seen that $\n$ is the terminal $\n^\circ$-initial algebra.
An educated guess is that $M^{\leq n}$ is the terminal $({M^{\leq n}})^\circ$-initial algebra as well.
and this is indeed the case. Moreover, for any $k \geq n$, $M^{\leq k}$ is a ${M^{\leq n}}^\circ$-initial algebra, just as in the case of natural numbers.
We have already seen a proof for the natural numbers, which in this context corresponds to consider the trivial monoid $M \cong 1$. 
One could adapt that proof to hold for any monoid $M$, but we will not do that here. In \cref{subsubsec: unbounded tree type : initial algebras}, we will prove this for a more complicated type: general trees.

One could also wonder about the preinitial algebra $(M')^*$.
An educated guess might be that $(M')^*$ is a ${(M')^*}^\circ$-initial algebra.
Alas, we have already seen at the start of this section that there does not exist a measuring 
$\phi\colon{(M')^*}^\circ \times (M')^* \to (M')^*,$
hence $(M')^*$ cannot be an ${(M')^*}^\circ$-initial algebra.

But what then can we say about coalgebras stemming from a subset $M' \subseteq M$, such as ${(M')^*}^\circ$ and ${((M')^{\leq n})}^\circ$?
As it turns out, they inherit their initial algebras from their counterparts stemming from the entire monoid $M$.
This is due to the fact that in this case, the monomorphisms $\iota_n \colon {(M')_n^*}^\circ \rightarrowtail {(M^{\leq n})}^\circ \in \CoAlg$ induce monomorphisms
\begin{align*}
    \iota_n \triangleright \id_A \colon {(M')_n^*}^\circ \triangleright A &\to {M_n^*}^\circ \triangleright A \\
    [\ell, \ell', a] &\mapsto [\ell, \ell', a].
\end{align*}
An algebra $A$ is ${M_n^*}^\circ$-initial if and only if ${M_n^*}^\circ \triangleright A \cong I$, where $I \cong M^*$ is the initial algebra.
A monomorphism into an initial object is an isomorphism, so using $\iota_n \triangleright \id_A$ we see ${(M')_n^*}^\circ \triangleright A \cong I$.
This implies any ${M_n^*}^\circ$-initial algebra $A$ is also ${(M')_n^*}^\circ$-initial.

\subsection{The Binary Tree Type}\label{subsec:binary}
Having seen the theory applied to the case of lists labeled in a monoid, we can continue to explore types frequently used in computer science.
One of them is the binary tree type. This type adds to the complexity by allowing a branching structure.
The branching structure is very predictable however, which makes it a nice stepping stone to the last section where we allow all branching structures.
The type of binary trees with nodes labeled in a commutative monoid $(M, \bullet, e)$ is the
initial algebra of the functor $F$ given by
\begin{align*}
    F \colon \Set &\longrightarrow \Set \\
    X &\longmapsto 1 + M \times X\times X.
\end{align*}
The functor $F$ is lax monoidal by
\begin{align*}
    \eta \colon &1 \to 1 + M \\
    &* \mapsto e
\end{align*}
and
\begin{align*}
    \nabla_{A,B} \colon (1+M \times X \times X) \times (1+M \times Y \times Y) &\to 1 + M \times X \times Y \times X \times Y \\
    (*,x) &\mapsto *\\
    (x,*) &\mapsto *\\
    ((m,x,x'),(m',y,y')) &\mapsto (m\bullet m', x,y,x',y').
\end{align*}
The category of $F$-algebras has elements
$
    \alpha \colon 1 + M \times A \times A \to A
$
denoted $(A,\alpha)$. 
Morphisms $f\colon (A,\alpha) \to (B,\beta)$ are given by functions $ f\colon A \to B $
which make the following diagram commute:
\[\begin{tikzcd}
1 + M \times A \times A  & 1 + M \times B \times B  \\
A & B.
            \arrow["F(f) ", from=1-1, to=1-2] 
\arrow["\alpha"',from=1-1, to=2-1] \arrow["\beta",from=1-2, to=2-2]
            \arrow["f", from=2-1, to=2-2]
\end{tikzcd}\]
The category of $F$-coalgebras has elements
$
    \chi\colon C \to 1 + M \times C \times C,
$
denoted $(C,\chi)$.
Morphisms $f\colon (C,\chi) \to (D,\delta)$ are given by functions $ f\colon C \to D $
which make the following diagram commute:
\[\begin{tikzcd}
    1 + M \times C \times C & 1 + M \times D \times D  \\
C & D.
            \arrow["F(f) ", from=1-1, to=1-2] 
\arrow["\chi",from=2-1, to=1-1] \arrow["\delta"',from=2-2, to=1-2]
            \arrow["f", from=2-1, to=2-2]
\end{tikzcd}\]

\subsubsection{Initial and Terminal Objects}
The initial algebra is given by the collection of all finite binary trees with nodes labeled in the monoid $M$.
In order to define the underlying set, we define the sets $\gls{Tree}$, which are the binary trees with values in $M$ of depth at most $n$.
Let $T_{M,0} = \{e_T\} \cong 1$, then we can write
$$
T_{M,n} = M \times T_{M,n-1} \times T_{M,n-1} + T_{M,0}.
$$
Note that we have a filtration $T_{M,0} \subseteq T_{M,1} \subseteq \dots \subseteq T_{M,n-1} \subset T_{M,n} \subseteq T_{M,n+1} \subseteq \cdots$.
The underlying set of the initial algebra is given by
$
T_M \coloneqq \bigcup_{n \in \N} T_{M,n}.
$
Its algebra structure is given by
\begin{align*}
    1 + M \times T_M \times T_M &\longrightarrow T_M \\
    * &\longmapsto e_T\\
    (m,\ell, r) &\longmapsto (m,\ell, r).
\end{align*}
Given an element $(m, \ell, r) \in T_M$, we will call $m$ the value stored in the node, $\ell$ the left child and $r$ the right child.
Some examples of preinitial algebras include $T_{M,n}$ for all $n \in \N$, $T_{M'}$ and $T_{M',n}$, where $M' \subseteq M$.
A slightly less obvious example is $M^*$, with algebra structure
\begin{align*}
    \alpha_\ell\colon 1 + M \times M^* \times M^* &\longrightarrow M^* \\
    * &\longmapsto [\:]\\
    (m,\ell, r) &\mapsto m:\ell.
\end{align*}
The unique algebra morphism $T_M \to (M^*, \alpha_\ell)$ returns the left most branch of the tree as a list.
Altering the algebra structure on $M^*$ does give different results.
Consider the algebra structure
\begin{align*}
    \alpha_{\max}\colon 1 + M \times M^* \times M^* &\to M^* \\
    * &\mapsto [\:]\\
    (m,\ell, r) &\mapsto
    \begin{cases}
        m:\ell &\text{ if } \len(\ell) \geq \len(r) \\
        m:r &\text{ otherwise}.
    \end{cases}
\end{align*}
In this case, the unique algebra morphism $T \to (M^*, \alpha_{\max})$ gives the (leftmost) longest branch.
Lastly, we remark there exists a morphism which can capture the shape of the binary tree while forgetting its contents.
It is defined as the unique algebra morphism
$
\shape\colon T_M \to T_{M,1}.
$

In order to define the terminal coalgebra we need to define consider the set $\widetilde{T}_{M,\infty}$ of trees of infinite depth, labeled in $M$.
The terminal coalgebra has as underlying set $\gls{TreeInf} = T_M + \widetilde{T}_{M,\infty}$ with coalgebra structure given by
\begin{align*}
    T_{M,\infty}  &\to 1+ M \times T_{M,\infty} \times T_{M,\infty}  \\
    e_T &\mapsto *\\
    (m,\ell, r) &\mapsto (m,\ell, r).
\end{align*}
Subterminal coalgebras are given by subsets of $T_{M,\infty}$ which respect the algebra structure.
Again, we will use $(-)^\circ$ to distinguish algebras and coalgebras which have the same underlying set.
Some examples are $T_{M,n}^\circ, T_M^\circ, \widetilde{T}_{M,\infty}$, but we also have the example of
\begin{align*}
    M^{\leq \infty} &\to 1 + M \times M^{\leq \infty} \times M^{\leq \infty} \\
    [\:] &\mapsto *\\
    m:\ell &\mapsto (m,\ell,[\:]).
\end{align*}
Again, changing the coalgebra structure on $M^{\leq \infty}$ will give a plethora of different coalgebras, each with their own interpretation.

\subsubsection{Measurings}
Given algebras $A, B \in \Alg$ and a coalgebra $C \in \CoAlg$, a measuring is a function $\phi\colon C \times A \to B$ satisfying
\begin{enumerate}
    \item $\phi_c(e_A) = e_B$ for all $c \in C$
    \item $\phi_c(\alpha(m,a_\ell,a_r)) = e_B$ if $\chi(c) = *$
    \item $\phi_c(\alpha(m,a_\ell,a_r)) = 
    \beta(m' \bullet m, \phi_{c_\ell}(a_\ell), \phi_{c_r}(a_r))$ 
    if $\chi(c) = (m',c_\ell, c_r)$.
\end{enumerate}
Comparing this with the definition of measuring found in  \cref{subsec:list}, we see they are very similar.
The only difference is that in this case, we are accommodating the branching structure of binary trees.

\begin{example}
    For similar reasons as in \cref{ex:meauringtoX^*}, we remark there does not exist an algebra morphism $T_{M,n} \to T_M$.
    There does however exist a measuring from $T_{M,n}$ to $T_M$ by $T_{M,n}^\circ$.
    It is given by
    \begin{align*}
        \phi\colon T_{M,n}^\circ \times T_{M,n} &\to T_{M} \\
        (e_T, (m,\ell, r)) &\mapsto e_T\\
        ((m',\ell', r'), e_T) &\mapsto e_T\\
        ((m',\ell',r') , (m,\ell, r)) &\mapsto y (m' \bullet m,\phi(\ell', \ell), \phi(r',r)),
    \end{align*}
    which is a measuring by definition.
    Again, for reasons similar to those seen in \cref{ex:meauringtoX^*} a measuring to $T_M$ does exist, where an algebra morphism does not.
    The coalgebra $T_{M,n}^\circ$ gives us control over which parts of an element in $T_{M,n}$ we consider.
\end{example}

\begin{example}
    If we only care about modifying the shape of an element in $T_M$ and don't want to introduce any twist, we can consider the measuring
    \begin{align*}
        \phi \colon T_{\{e\}}^\circ \times T_M &\to T_M \\
        (e_T, (m,\ell, r)) &\mapsto e_T\\
        ((e,\ell', r'), e_T) &\mapsto e_T\\
        ((e,\ell',r') , (m,\ell, r)) &\mapsto y (m,\phi(\ell', \ell), \phi(r',r)),
    \end{align*}
    which is a measuring by definition.
    This measuring takes a tree containing only the unit of the monoid $t' \in T_{\{e\}}^\circ$ and uses its shape to modify a tree $t \in T_M$ to fit within the shape of $t'$.
\end{example}

\begin{example}
    Again, we can generalize the map function using a measuring to obtain more fine grained control over which function gets applied to which elements.
    Recall we define the function $r_m\colon M \to M$ as multiplication by $m$ on the right.
    Let $[M,M]$ be the set of functions $M \to M$, and let $T_{[M,M]}$ be the set of all binary trees with labels in $[M,M]$.
    We can then define an algebra structure on $T_{[M,M]}$ by
    \begin{align*}
        1 + M \times T_{[M,M]} \times T_{[M,M]} &\to T_{[M,M]} \\
        * &\mapsto e_T\\
        (m,\ell,r) &\mapsto (r_m, \ell, r).
    \end{align*}
    Now we have a measuring given by
    \begin{align*}
        \phi\colon T_M^\circ \times T_{[M,M]} &\to T_M \\
        (e_T, (f, \ell, r)) &\mapsto e_T\\
        ((m', \ell', r'), e_T) &\mapsto e_T\\
        ((m', \ell', r'), (f, \ell, r)) &\mapsto (f(m), \phi(\ell',\ell), \phi(r',r)).
    \end{align*}
    To check this is a measuring, we verify
    $$
    \phi((m', \ell', r'), (r_m, \ell, r)) =  (r_m(m'), \phi(\ell',\ell), \phi(r',r)) =  (m' \bullet m, \phi(\ell',\ell), \phi(r',r)).
    $$
\end{example}

\subsubsection{Free and Cofree Functors}

Given $X \in \Set$, we use the following description of the free algebra on $X$. Its underlying set is the set of all finite trees where nodes are labeled by $M$ and leaves are labeled by $X + \{e_{\Fr}\}$. It has the algebra structure
\begin{align*}
    1 + M \times \Fr(X) \times \Fr(X) &\to \Fr(X) \\
    * &\mapsto e_{\Fr}\\
    (m,\ell, r) &\mapsto (m,\ell, r).
\end{align*}
One can think of $\Fr(X)$ as the set of all finite trees with elements in $M$, where we allow the leaves to contain elements of $X$.
The behavior of $\Fr$ on functions $f\colon X \to Y$ is forced and given by applying $f$ to all elements of $X$ in any element of $\Fr(X)$.
Checking this is a left adjoint to the forgetful functor can be done by constructing the natural bijection
$$
[X,A] \cong \Alg(\Fr(X), A).
$$
It is given by sending a function $f\colon X \to A$ to the morphism
\begin{align*}
    \widetilde{f}\colon \Fr(X) &\to A \\
    (m,\ell, r) &\mapsto \beta(m, \widetilde{f}(\ell), \widetilde{f}(r))\\
    x &\mapsto f(x)\\
    e_{\Fr} &\mapsto \beta(*).
\end{align*}
and sending a morphisms $\widetilde{f}\colon \Fr(X) \to A$ to its restriction to $X \subseteq \Fr(X)$.
The morphisms $\widetilde{f}$ then sends the nodes to their embedding into $A$, and then appends the leaves with the trees $f(x) \in A$.

The cofree functor is defined by mapping $X \in \Set$ to
\begin{align*}
    \delta\colon X \times T_{M \times X \times X, \infty} &\to 1 + M \times X \times T_{M \times X \times X, \infty} \times X \times T_{M \times X \times X, \infty}\\
    (x_0, e_T) &\mapsto *\\
    (x_0, (m,x_\ell,x_r,t_\ell,t_r)) &\mapsto (m,x_\ell,t_\ell,x_r,t_r).
\end{align*}
The behavior of $\Cof$ on functions $f\colon X \to Y$ is forced and given by applying $f$ to all elements of $X$ in any element of $Y \times T_{M \times Y \times Y, \infty}$.
Checking this is a right adjoint to the forgetful functor can be done by constructing the natural bijection
$$
[C,X] \cong \CoAlg(C, \Cof(X)).
$$
It is given by sending a function $f\colon C \to X$ to the morphism
\begin{align*}
    \widetilde{f}\colon C &\to X \times T_{M \times X \times X, \infty} \\
    c &\mapsto 
    \begin{cases}
        (f(c), e_T) \text{ if } \chi(c) = * \\
        (f(c), (m, f(c_\ell), f(c_r)), \widetilde{f}(c_\ell), \widetilde{f}(c_r)) \text{ if } \chi(c) = (m, c_\ell, c_r).
    \end{cases}
\end{align*}
Conversely, a morphism $\widetilde{f}\colon C \to \Cof(X)$ is sent to its composition with the projection on the first coordinate.

\subsubsection{Representing Objects}
Here, we we will be very brief about the representing objects of $\m$.
Constructing them and verifying they do indeed represent $\m$ is completely analogous to the constructions and verification done in \cref{subsec:list}

The representing object $\underline{\Alg}(A,B)$ is a subset
$
\underline{\Alg}(A,B) \subseteq T_{M\times [A,B],\infty}.
$
We can view elements of $T_{M \times [A,B],\infty}$ as two trees with the same shape, one containing elements in $M$ and the other functions $A \to B$.
We will write $(t_m,t_f) \in T_{M \times [A,B],\infty}$ for these two trees.
An element $(t_m,t_f) \in T_{M\times [A,B]}^\infty$ is included in $\underline{\Alg}(A,B)$ if and only if
$f_i(e_A) = e_B$ for all $f_i \in t_f$ and 
$f_i(\alpha(m,a_\ell,a_r)) = \beta(m\bullet m_i,f_{i_\ell}(a_\ell),f_{i_r}(a_r))$,
where $f_{i_\ell},f_{i_r}$ are the children of $f_i$ if they exist and are otherwise are taken to be the constant function $\const_{e_B}$.

The measuring tensor $C \triangleright A$ is given by
$$
C \triangleright A = \Fr(C \times A)/\sim
$$
where
\begin{align*}
    (c,e_A) &\sim e_{\Fr}\\
    (c,\alpha(m,a_\ell, a_r)) &\sim e_{\Fr} \text{ if } \chi(c) = *\\
    (c,\alpha(m,a_\ell, a_r)) &\sim (m'\bullet m, (c_\ell, a_\ell), (c_r, a_r)) \text{ if } \chi(c) = (m',c_\ell, c_r).
\end{align*}
and we also apply the above relation recursively throughout the tree in $\Fr(C \times A)$.
An element of $C \triangleright A$ can be viewed as a tree containing elements of $M$ in its nodes and having elements of $C\times A$ at its leaves.
Given a tuple $(c,a) \in C \times A$ at the leaf, we attempt to expand out its tree structure using $\alpha$ and $\chi$, and replace it with an empty leaf if one of these two turns out to be empty.

Given an algebra $B$ and a coalgebra $C$ the convolution algebra $[C,B]$ has the algebra structure
\begin{align*}
     1+ M \times [C,B] \times & [C,B] \to [C,B]\\
    * &\mapsto (c \mapsto e_B)\\
    (m,f_\ell, f_r) &\mapsto \left(
    c \mapsto
    \begin{cases}
        e_B \text{ if } \chi(c) = * \\
        \beta(m'\bullet m, f_\ell(c_\ell), f_r(c_r)) \text{ if } \chi(c) = (m', c_\ell, c_r)
    \end{cases}
    \right).
\end{align*}

\subsubsection{$C$-Initial Algebras}
Many of our results from \cref{subsec:list} carry over to this case.
In line with our expectations, $T_{M,n}$ is indeed the terminal $T_{M,n}^\circ$-initial algebra.
We also have that the algebra $T_{M,k}$ is $T_{M,n}^\circ$-initial for all $k \geq n$.
Moreover, the algebras $T_{M,k}$ are $T_{M',n}^\circ$-initial for all $k \geq n$ and subsets $M' \subseteq M$, again similar to \cref{subsec:list}.

\subsection{The Type of Trees}\label{subsec:unbounded}
For our final example, we explore what happens if we lose control over the amount of branching.
We want to study the type of trees, which contains trees where each node can have any (possibly infinite) number of children.
Again, we will see measurings introduce control over the shape of tree, as well as some element-wise twisting.

Let $F$ be the functor given by
\begin{align*}
    F\colon \Set &\longrightarrow \Set \\
    X &\longmapsto 1 + M \times X^{\leq \infty}\\
    f &\longmapsto \id_1 + \id_M \times \map(f)
\end{align*}
where $(M, \bullet, e)$ is a commutative monoid and $\gls{map}$ is a function frequently used in functional programming.
It is defined as
\begin{align*}
    \map\colon [X,Y] \times X^{\leq \infty} &\to Y^{\leq \infty} \\
    (f,[\:]) &\mapsto [\:]\\
    (f,x:\ell) &\mapsto f(x) : \map(f)(\ell),
\end{align*}
hence takes a list $\ell \in X^{\leq \infty}$ and a function $f\colon X \to Y$ and applies $f$ to every element of the list $\ell$, returning a list containing elements of $Y$.
We claim $F$ is lax monoidal by
$
    \eta\colon 1 \to 1 + M \times 1^{\leq \infty} \cong 1 + M \times \eN,
    * \mapsto (e, \infty)
$
and
\begin{align*}
    \nabla_{A,B}\colon (1+M \times X^{\leq \infty}) \times (1+M \times Y^{\leq \infty}) &\to 1+M \times (X \times Y)^{\leq \infty} \\
    (*,x) &\mapsto *\\
    (x,*) &\mapsto *\\
    ((m,\ell),(m',k)) &\mapsto (m\bullet m', \zip(\ell,k)),
\end{align*}
where $\gls{zip}$ is a function borrowed from functional programming and defined as
\begin{align*}
    \zip\colon X^{\leq \infty} \times Y^{\leq \infty} &\to (X\times Y)^{\leq \infty} \\
    ([\:], \ell) &\mapsto [\:]\\
    (\ell, [\:]) &\mapsto [\:]\\
    (x:\ell, y:k) &\mapsto (x,y):\zip(\ell, k).
\end{align*}
The category of $F$-algebras has elements
$
    \alpha\colon 1 + M \times A^{\leq \infty} \to A
$
denoted $(A,\alpha)$. 
Morphisms $f\colon(A,\alpha) \to (B,\beta)$ are given by functions $ f\colon A \to B $
which make the following diagram commute:
\[\begin{tikzcd}
1 + M \times A^{\leq \infty} & 1 + M \times B^{\leq \infty}  \\
A & B.
            \arrow["F(f) ", from=1-1, to=1-2] 
\arrow["\alpha"',from=1-1, to=2-1] \arrow["\beta",from=1-2, to=2-2]
            \arrow["f", from=2-1, to=2-2]
\end{tikzcd}\]
The category of $F$-coalgebras has elements
$
    \chi\colon C \to 1 + M \times C^{\leq \infty}$,
denoted $(C,\chi)$.
Morphisms $f\colon (C,\chi) \to (D,\delta)$ are given by functions $ f\colon C \to D $
which make the following diagram commute:
\[\begin{tikzcd}
    1 + M \times C^{\leq \infty} & 1 + M \times D^{\leq \infty} \\
C & D.
            \arrow["F(f) ", from=1-1, to=1-2] 
\arrow["\chi",from=2-1, to=1-1] \arrow["\delta"',from=2-2, to=1-2]
            \arrow["f", from=2-1, to=2-2]
\end{tikzcd}\]

\subsubsection{Initial and Terminal Objects}
The initial algebra is given by the collection of all trees labeled in $M$ with finite depth.
In order to define the underlying set, we define the sets $\gls{STreen}$, which are the trees labeled in $M$ of depth at most $n$.
Let $S_{M,0} = \{e_S\} \cong 1$, then we can write
$$
S_{M,n} = M \times (S_{M,n-1})^{\leq \infty} \times S_{M,0}.
$$
Note that we have a filtration $S_{M,0} \subseteq S_{M,1} \subseteq \dots \subseteq S_{M,n-1} \subseteq S_{M,n} \subseteq S_{M,n+1} \subseteq \dots$.
The underlying set of the initial algebra is given by
$
\gls{STree} \coloneqq \bigcup_{n \in \N} S_{M,n}.
$
Its algebra structure is given by
\begin{align*}
    1 + M \times (S_{M})^{\leq \infty} &\to S_{M} \\
    * &\mapsto e_S\\
    (m,\ell) &\mapsto (m,\ell).
\end{align*}
For $(m,\ell) \in S_M$ we will call $m$ the value stored at the node and $\ell$ the list of children.

Some examples of preinitial algebras include $S_{M,n}$ for all $n \in \N$, $S_{M'}$ and $S_{M',n}$, where $M' \subseteq M$.
A slightly more subtle example is $M^*$, with algebra structure
\begin{align*}
    \alpha\colon 1 + M \times (M^*)^{\leq \infty} &\to M^* \\
    * &\mapsto [\:]\\
    (m,[\:]) &\mapsto [m]\\
    (m,\ell:k) &\mapsto m:\ell.
\end{align*}
The unique algebra morphism $S_M \to M^*$ returns the left most branch of the tree as a list.
Altering the algebra structure on $M^*$ does give different results.
Consider the algebra structure
\begin{align*}
    \alpha_{\max}\colon 1 + M \times (M^*)^{\leq \infty} &\to M^* \\
    * &\mapsto [\:]\\
    (m,[\:]) &\mapsto [m]\\
    (m,k) &\mapsto m:\ell \text{ where } \ell = \argmax_{\ell'\in k}(\len(\ell'))
\end{align*}
where $\argmax_{\ell'\in k}(\len(\ell'))$ is the maximum $\len(\ell')$ over all $\ell'\in k$.
In this case, the unique algebra morphism $T \to (M^*, \alpha_{\max})$ gives the longest branch.
Lastly, we remark there exists a morphism which can capture the shape of the tree while forgetting its contents.
It is defined as the unique algebra morphism
$
\shape\colon S_M \to S_1.
$

In order to define the terminal coalgebra, let $\widetilde{S}_{M,\infty}$ denote the set which contains trees of infinite depth, labeled in $M$.
The terminal coalgebra has as underlying set $\gls{STreeInf} = S_M + \widetilde{S}_{M,\infty}$ with coalgebra structure given by
\begin{align*}
    S_{M,\infty}  &\to 1+ M \times (S_{M,\infty})^{\leq \infty} \\
    e_S &\mapsto *\\
    (m,\ell) &\mapsto (m,\ell).
\end{align*}
Subterminal coalgebras are given by subsets of $S_{M,\infty}$ which respect the algebra structure.
Again, we will use $(-)^\circ$ to distinguish algebras and coalgebras which have the same underlying set.
Some examples are ${S_{M,n}}^\circ, S_M^\circ, \widetilde{S}_{M,\infty}$, but we also have the example of
\begin{align*}
    M^{\leq \infty} &\to 1 + M \times (M^{\leq \infty})^{\leq \infty} \\
    [\:] &\mapsto *\\
    m:\ell &\mapsto (m,[\ell]).
\end{align*}
Again, changing the coalgebra structure on $M^{\leq \infty}$ will give a plethora of different coalgebras, each with their own interpretation.

\subsubsection{Measurings}
Given algebras $(A,\alpha), (B,\beta) \in \Alg$ and a coalgebra $(C,\chi) \in \CoAlg$, a measuring is a function $\phi\colon C \times A \to B$ satisfying
\begin{enumerate}
    \item $\phi_c(e_A) = e_B$ for all $c \in C$
    \item $\phi_c(\alpha(m,\ell)) = e_B$ if $\chi(c) = *$
    \item $\phi_c(\alpha(m,\ell)) = 
    \beta(m' \bullet m, \map(\phi)(\zip(cs,\ell)))$ 
    if $\chi(c) = (m',\ell)$.
\end{enumerate}
It is again very similar to previous definitions of measurings.
The most notable difference is that we are zipping the lists of children $c$ and $\alpha(m, \ell)$ have.
This means a measuring does not only give us control over the depth of a tree, which would be analogous to lists, but gives us control over the branching structure of the tree as well.

\begin{example}
    For similar reasons as in previous examples, we remark there does not exist an algebra morphism $S_{M,n} \to S_M$.
    There does however exist a measuring from $S_{M,n}$ to $S_M$ by $S_{M,n}^\circ$.
    It is given by
    \begin{align*}
        \phi\colon S_{M,n}^\circ \times S_{M,n} &\to S_{M} \\
        (e_S, (m,\ell)) &\mapsto e_S\\
        ((m',\ell'), e_T) &\mapsto e_T\\
        ((m',\ell') , (m,\ell)) &\mapsto (m' \bullet m,\,\map(\phi)(\zip(\ell',\ell))),
    \end{align*}
    which is a measuring by definition.
    Again, reasoning similar to that seen in Example \ref{ex:meauringtoX^*} applies as to why a measuring to $S_M$ does exist, where an algebra morphism does not.
    The coalgebra $S_{M,n}^\circ$ gives us control over which parts of an element in $S_{M,n}$ we consider.
\end{example}

\begin{example}
    Consider the measuring
    \begin{align*}
        \phi\colon S_{\{e\}}^\circ \times S_M &\to S_M \\
        (e_S, (m,\ell)) &\mapsto e_S\\
        ((e,\ell'), e_T) &\mapsto e_T\\
        ((e,\ell') , (m,\ell)) &\mapsto (m,\map(\phi)(\zip(\ell',\ell))),
    \end{align*}
    which is a measuring by definition.
    This measuring takes a tree containing only the unit of the monoid $t' \in S_1^\circ$ and uses its shape to modify a tree $t \in S_M$ to fit within the shape of $t'$.
\end{example}

\begin{example}
    Again, we can generalize the map function using a measuring to obtain more fine grained control over which function gets applied to which elements.
    Recall we define the function $r_m\colon M \to M$ as multiplication by $m$ on the right.
    We define an algebra structure on $S_{[M,M]}$ by
    \begin{align*}
        1 + M \times S_{[M,M]} \times S_{[M,M]} &\to S_{[M,M]} \\
        * &\mapsto e_T\\
        (m,\ell,r) &\mapsto (r_m, \ell, r).
    \end{align*}
    Now we have a measuring given by
    \begin{align*}
        \phi\colon S_M^\circ \times S_{[M,M]} &\to S_M \\
        (e_S, (f, \ell)) &\mapsto e_S\\
        ((m', \ell'), e_S) &\mapsto e_S\\
        ((m', \ell'), (f, \ell)) &\mapsto (f(m'), \map(\phi)(\zip(\ell',\ell))).
    \end{align*}
    To check this is a measuring, we verify
    $$
    \phi((m', \ell'), (r_m, \ell)) =  (r_m(m'), \map(\phi)(\zip(\ell',\ell))) =  (m' \bullet m, \map(\phi)(\zip(\ell',\ell))).
    $$
\end{example}

\subsubsection{Free and Cofree Functors}

We use the following description of the free algebras.
Given $X \in \Set$, the free algebra $\Fr(X)$ has as underlying set the set of all trees of finite depth whose nodes are labeled by $M$ and whose leaves are labeled by $X + \{e_{\Fr}\}$. Its algebra structure is given by
\begin{align*}
    1 + M \times \Fr(X)^{\leq \infty} &\to \Fr(X) \\
    * &\mapsto e_{\Fr}\\
    (m,\ell) &\mapsto (m,\ell).
\end{align*}
One can think of $\Fr(X)$ as the set of all finite-depth trees with elements in $M$, where we allow the leaves to contain elements of $X$.
The behavior of $\Fr$ on functions $f\colon X \to Y$ is forced and given by applying $f$ to all elements of $X$ in any element of $\Fr(X)$.
Checking this is a left adjoint to the forgetful functor can be done by constructing the natural bijection
$$
[X,A] \cong \Alg(\Fr(X), A).
$$
It is given by sending a function $f\colon X \to A$ to the morphism
\begin{align*}
    \widetilde{f}\colon \Fr(X) &\to B \\
    (m,\ell) &\mapsto \beta(m, \map(\widetilde{f})(\ell))\\
    x &\mapsto f(x)\\
    e_{\Fr} &\mapsto \beta(*).
\end{align*}
and sending a morphisms $\widetilde{f}\colon \Fr(X) \to A$ to its restriction to $X \subseteq \Fr(X)$.
The morphisms $\widetilde{f}$ sends the nodes to their embedding into $A$, and then appends the leaves with the trees $f(x) \in A$.

The cofree functor is defined by mapping $X \in \Set$ to
\begin{align*}
    \delta\colon X \times S_{\Cof(X)} &\to 1 + M \times (X \times S_{\Cof(X)})^{\leq \infty}\\
    (x_0, e_S) &\mapsto *\\
    (x_0, (m,\ell)) &\mapsto (m,\ell)
\end{align*}
The behavior of $\Cof$ on functions $f\colon X \to Y$ is forced and given by applying $f$ to all elements of $X$ in any part of $\Cof(Y)$.
Checking this is a right adjoint to the forgetful functor can be done by constructing the natural bijection
$$
[C,X] \cong \CoAlg(C, \Cof(X)).
$$
It is given by sending a function $f\colon C \to X$ to the morphism
\begin{align*}
    \widetilde{f} \colon C &\to X \times S_{\Cof(X)} \\
    c &\mapsto 
    \begin{cases}
        (f(c), e_S) \text{ if } \chi(c) = * \\
        (f(c), (m, \map(\widetilde{f})(\ell)) \text{ if } \chi(c) = (m, \ell).
    \end{cases}
\end{align*}
Conversely, a morphism $\widetilde{f}\colon C \to \Cof(X)$ is sent to its composition with the projection on the first coordinate.

\subsubsection{Representing Objects}
Here, we we will be very brief about the representing objects of $\m$.
Constructing them and verifying they do indeed represent $\m$ is analogous to the constructions and arguments done in \cref{subsec:list}

The universal measuring coalgebra is a subset
$
\underline{\Alg}(A,B) \subseteq S_{M \times [A,B]}^\infty.
$
We can view elements of $S_{M \times [A,B]}^\infty$ as two trees with the same shape, one containing elements in $M$ and the other functions $A \to B$.
We will write $(t_m,t_f) \in S_{M \times [A,B]}^\infty$ for these two trees.
An element $(t_m,t_f) \in \underline{\Alg}(A,B)$ if and only if
\begin{enumerate}
    \item $f_i(e_A) = e_B$ for all $f_i \in t_f$
    \item $f_i(\alpha(m,\ell)) = \beta(m\bullet m_i,\mathtt{zipapply}(f,\ell))$,
\end{enumerate}
where here by $\gls{zipapply}(f,\ell)$ we mean applying each function $f_i$ to each element $\ell_i$ to obtain a new list.

The measuring tensor $C \triangleright A$ is given by
$$
C \triangleright A = \Fr(C \times A)/\sim
$$
where
\begin{align*}
    (c,e_A) &\sim e_{\Fr}\\
    (c,\alpha(m,\ell)) &\sim e_{\Fr} \text{ if } \chi(c) = *\\
    (c,\alpha(m,\ell)) &\sim (m'\bullet m, \zip(k, \ell)) \text{ if } \chi(c) = (m',k)\\
\end{align*}
and we also apply the above relation recursively throughout the tree in $\Fr(C \times A)$.
The object $C \triangleright A$ can be interpreted as being a tree containing elements of $M$ in its nodes and having elements of $C \times A$ at its leaves.
Given a tuple $(c,a) \in C \times A$ at the leaf, we attempt to expand out its tree structure using $\alpha$ and $\chi$, and replace it with an empty leaf if one of the two turns out to be empty.

Given an algebra $B$ and a coalgebra $C$, the convolution algebra $[C, B]$ has algebra structure
\begin{align*}
    1+ M \times [C,B]^* &\to [C,B]\\
    * &\mapsto (c \mapsto e_B)\\
    (m,\ell) &\mapsto \left(
    c \mapsto
    \begin{cases}
        e_B \text{ if } \chi(c) = * \\
        \beta(m' \bullet m,\,\mathtt{zipapply}(\ell,k) ) \text{ if } \chi(c) = (m', k)
    \end{cases}
    \right).
\end{align*}

\subsubsection{$C$-Initial Algebras}
\label{subsubsec: unbounded tree type : initial algebras}
In this section we would like to give a proof of the fact that $S_{M,n}$ is the terminal $S_{M,n}^\circ$-initial algebra.
We have already seen a proof of a similar result for the natural numbers, and this proof has the same outline.

We first show there exists a morphism $A \to S_{M,n}$ for any $S_{M,n}^\circ$-initial algebra $A$ using induction.
After that, we will show the morphism is unique.
We start with a lemma which allows us to use induction later on.

\begin{lemma}
    Let $A$ be a $S_{M,n}^\circ$-initial algebra. Then $A$ is also $S_{M,n-1}^\circ$ initial.
\end{lemma}

\begin{proof}
    Consider the coalgebra morphism $\iota\colon S_{M,n-1}^\circ \to S_{M,n}^\circ, s \mapsto s$.
    This induces a morphism
    \begin{align*}
        \iota\triangleright A \colon S_{M,n-1}^\circ \triangleright A &\to  S_{M,n}^\circ \triangleright A \\
        [e_{\Fr}] &\mapsto [e_{\Fr}]\\
        [s,a] &\mapsto [s,a]\\
        [m,\ell] &\mapsto [m, \map(\iota \triangleright A)(\ell)].
    \end{align*}
    This morphism is monomorphic by definition.
    Since $A$ is $S_{M,n}^\circ$-initial, we know $S_{M,n} \triangleright A$ is an initial object.
    This means $\iota\triangleright A $ is a monomorphism into the initial object, hence an isomorphism.
\end{proof}

% As an immediate consequence we have the following corollary.

\begin{corollary}
    Let $A$ be a $S_{M,n}^\circ$-initial algebra. Then $A$ is also $S_{M,k}^\circ$ initial for all $k \leq n$.
\end{corollary}

The next lemma is a technical lemma which we will be able to leverage during the induction step.

\begin{lemma}\label{lem:adddepth}
    Let $A$ be a $S_{M,n}^\circ$-initial algebra and let $\phi\colon S_{M,n}^\circ \times A \to S_M$ be the unique measuring from $A$ to $S_{M}$ by $S_{M,n}^\circ$.
    For all $0 \leq i \leq n$, let  $s_i \in S_{M,n}$ be the largest tree of depth $i$ in $S_{M,n}$ which contains only the unit $e \in M$.
    Then for all $0 \leq i \in n$ and $0 \leq j \leq n-i$ we have
    $$
    \phi_{s_i}(a) = \phi_{s_i}(\fromI{A}\circ\phi_{s_{i+j}} (a)).
    $$
    % $$
    % \phi(s_i,a) = \phi(s_i, \fromI{A}\circ\phi(s_{i+j}, a)).
    % $$
\end{lemma}

\begin{proof}
    Let $k\leq n$.
    First, we define a function which takes two trees and appends the second tree to all the roots of the first tree.
    We define it as
    \begin{align*}
        (++) \colon S_{M,\infty} \times S_{M,\infty} &\to S_{M,\infty} \\
        (e_S, s) &\mapsto s\\
        ((m, \ell), s) &\mapsto (m, \map( (++) s )(\ell)).
    \end{align*}
    For all $0 \leq j \leq n-k$, define the coalgebra morphism
    \begin{align*}
        p_j\colon S_{M,k}^\circ &\to S_{M,n}^\circ \times S_{M,n}^\circ \\
        s &\mapsto (s,\widetilde{s} ++ s_j),
    \end{align*}
    where $\widetilde{s}$ is the unique tree $\widetilde{s} \in S_{\{e\}, k}$ such that $\shape(\widetilde{s}) = \shape(s)$.
    % TODO, show this is a coalgebra morphism.
    Consider the composition of measurings $\phi \circ (\id \times \fromI{A} \circ \phi) \in \m_{S_{M,n}^\circ \times S_{M,n}^\circ}(A,S_{M})$.
    We can precompose this measuring with the coalgebra morphism $p_j$ to obtain a measuring
    $$
    S_{M,k}^\circ \times A \xrightarrow{p_j \times \id_A} S_{M,n}^\circ \times S_{M,n}^\circ \times A \xrightarrow{\phi \circ (\id \times \fromI{A} \circ \phi)} S_{M}.
    $$
    %%%%%%%%%%%
    We also have the coalgebra morphism $\iota\colon S_{M,k}^\circ  \to S_{M,n}^\circ, s \mapsto s$ which we can precompose with $\phi$.
    This gives us a measuring
    $$
    S_{M,k}^\circ \times A \xrightarrow{\iota \times \id_A} S_{M,n}^\circ \times A \xrightarrow{\phi} S_{M}.
    $$  
    Since $A$ is also $S_{M,k}^\circ$-initial, we know these measurings must coincide.
    Hence we can state
    $
    \phi(s,a) = \phi(s,\fromI{A} \circ \phi(\widetilde{s} ++ s_j,a))
    $
    for all $s \in S_{M,k}^\circ$ and $0 \leq j \leq n-k$.
    The only restriction placed on $k$ was that $k \leq n$.
    Iterating over all $0 \leq k \leq n$, we arrive at the desired result
    $$
    \phi(s,a) = \phi(s,\fromI{A}\circ \phi(\widetilde{s} ++ s_j,a))
    $$
    for all $0 \leq i \leq n$ and $0 \leq j \leq n-i$.
    In particular this holds for $s = s_i$, and noting $s_i ++ s_j = s_{i + j}$ we obtain the desired result. 
\end{proof}

\begin{corollary}\label{cor:bound}
    Let $A$ be $S_{M,n}^\circ$-initial and let $\phi$ be the unique measuring to $S_{M}$.
    For all $0 \leq i \leq n$ and $a \in A$, we have $\phi(s_i,a) \in S_{M,i}$.
\end{corollary}

\begin{proof}
    We will proceed by induction over $i$.
    For the base case $i = 0$, we have $s_0 = e_S$, so $\phi(s_0, a) = e_S \in S_{M,0}$.
    For the induction step, assume $\phi(s_i, a) \in S_{M,i}$.
    Using $s_{i+1} = (e,(s_i)^{(\infty)})$ and the previous lemma we can write
    $
    \phi(s_{i+1}, a) = \phi( (e,(s_i)^{(\infty)}), \fromI{A}(\phi(s_{i+1}, a))).
    $
    Let $\alpha\colon 1 + M \times (S_M)^{\leq \infty} \to S_M$, then we can write $(m,s) = \alpha\inv(\phi(s_{i+1}, a))$ where $s$ is the list with components $s_i$.
    We then know by definition of a measuring that
    $$
    \phi(s_{i+1}, a)  = \phi( (e,(s_i)^{(\infty)}), \fromI{A}(\phi(s_{i+1}, a))) = 
    (e \bullet m, \map(\phi)(\zip((s_i)^{(\infty)}, k)   ) ),
    $$
    where $k = \map(\fromI{A})(s)$.
    By the induction hypothesis, we know $$\map(\phi)(\zip((s_i)^{(\infty)}, k)   )  \in (S_{M,i})^{\leq \infty},$$ hence $\phi(s_{i+1}, a) \in S_{M,i + 1}$ by definition.
    This concludes the induction step and the proof.
\end{proof}

Now we are ready to define the a family of functions which will culminate in an algebra morphism $A \to S_{M,n}$ for any $S_{M,n}^\circ$ initial algebra $A$.

\begin{definition}\label{def:AtoS_k}
    Let $A$ be $S_{M,n}^\circ$-initial, let $\phi$ be the unique measuring to $S_{M}$ and let $0 \leq k \leq n$. Define the functions
    \begin{align*}
        \phi_k \colon A &\to S_{M,k} \\
        a &\mapsto \phi(s_k,a),
    \end{align*}
    which are well-defined by the previous corollary.
\end{definition}

\begin{remark}
    With this new definition, we can restate  \cref{lem:adddepth} as
    $$
    \phi_i = \phi_i \circ \fromI{A} \circ \phi_{i+j}
    $$
    for all $0 \leq i \leq n$ and $0 \leq j \leq n-i$.
\end{remark}

\begin{proposition}\label{prop:AtoS_nexists}
    The functions $\phi_k\colon A \to S_{M,k}$ from  \cref{def:AtoS_k} are algebra morphisms for all $0 \leq k \leq n$.
\end{proposition}

\begin{proof}
    We proceed by induction over $k$.
    For $k = 0$, we have that $S_{M,k} \cong 1$, the terminal object in $\Set$.
    Hence $\phi_0\colon A \to S_{M,0}$ is an algebra morphism.
    For the inductive step, assume $\phi_{k-1}\colon A \to S_{M,k-1}$ is an algebra morphism.
    We wish to show $\phi_k$ is an algebra morphism.
    In more detail, we aim to show $$\alpha_k(m, \map(\phi_k)(\ell)) = \phi_k(\alpha(m,\ell))$$ for all $(m,\ell) \in M \times A^{\leq \infty}$.
    First, we will write $\ell = (a_i)_{i = 0}^j$, where $j \in \eN$.
    We can now write
    \begin{align*}
        \alpha_k(m, \map(\phi_k)(\ell))
        &= \alpha_k(m, \map(\phi_k)((a_i)_i))\\
        &= \alpha_k(m, (\phi_k(a_i))_i)\\
        &= \alpha_k(m, (\phi_k \circ \fromI{A} \circ \phi_k(a_i))_i).
    \end{align*}
    Writing $\phi_k \circ \fromI{A} \circ \phi_k(a_i) = t_i$, we can make the case distinction
    $$ 
    t_i = 
    \begin{cases}
        e_S &\text{ if } \phi_k(a_i) = e_S \in S_M \\
        (m_i, \map(\phi_{k-1} \circ \fromI{A})(a_i)) &\text{ if } \phi_k(a_i) = (m_i,a_i) \in S_M
    \end{cases}
    $$
    by $\phi \colon S_{M,n}^\circ \times A \to S_M$ being a measuring.
    This means $$\alpha_k(m, \map(\phi_k)(\ell)) = \alpha_k(m, (t_i)_i).$$
    By the algebra structure of $S_{M,k}$, we can write this as
    $
    \alpha_k(m, (t_i)_i) = (m,(r_i)_i),
    $
    where
    $$
    r_i = 
    \begin{cases}
        e_S &\text{ if } \phi_k(a_i) = e_S \in S_M \\
        \alpha_{k-1}(m_i, \map(\phi_{k-1} \circ \fromI{A})(a_i)) &\text{ if } \phi_k(a_i) = (m_i,a_i) \in S_M.
    \end{cases}
    $$
    By the induction hypothesis
    $
    \alpha_{k-1}(m_i, \map(\phi_{k-1} \circ \fromI{A})(a_i)) 
    = (\phi_{k-1} \circ \fromI{A})(m_i, a_i)
    $
    and we can write
    $$
    r_i = (\phi_{k-1} \circ \fromI{A})(\phi_k(a_i)) = (\phi_{k-1} \circ \fromI{A} \circ \phi_{k-1})(a_i)
    $$
    using \cref{lem:adddepth}.
    Retracing our steps, we now compute
    \begin{align*}
        \alpha_k(m, \map(\phi_k)(\ell))
        &= (m,(r_i)_i)\\
        &= (m, (\phi_{k-1} \circ \fromI{A} \circ \phi_{k-1})(a_i)_i)\\
        &= (m, \map(\phi_{k-1}) ((\fromI{A} \circ \phi_{k-1})(a_i))_i)\\
        &= \phi_k(\alpha(m,  (\fromI{A} \circ \phi_{k-1})(a_i))_i  ))\\
        &= (\phi_k \circ \fromI{A} \circ \phi_k)(\alpha(m, (a_i)_i ))\\
        &= \phi_k(\alpha(m, \ell))
    \end{align*}
    using that $\phi$ is a measuring twice.
    We conclude $\phi_k \colon A \to S_{M,k}$ is an algebra morphism for all $0 \leq k \leq n$.
\end{proof}

In particular, this lemma shows $\phi_n\colon A \to S_{M,n}$ is an algebra morphism.
It still remains to show this algebra morphism is unique.

\begin{lemma}\label{lem:AtoS_nunique}
    For any $S_{M,n}^\circ$-initial algebra $A$, there exist at most one algebra morphism $A \to S_{M,n}$.
\end{lemma}

\begin{proof}
    By \cref{prop:map from initial alg to dual}, there exists a unique morphism $A \to [S_{M,n}^\circ, S_{M}]$ for any $S_{M,n}^\circ$-initial algebra $A$.
    Since $S_{M,n}$ is $S_{M,n}^\circ$-initial we know there exists a unique morphism $\iota \colon S_{M,n} \rightarrowtail [S_{M,n}^\circ, S_{M}]$.
    %%%%%
    This map is given by
    \begin{align*}
        \iota \colon S_{M,n} &\to [S_{M,n}^\circ, S_{M}] \\
        e_S &\mapsto \const_{e_S}\\
        (m,\ell) &\mapsto
        \left( 
        s' \mapsto 
        \begin{cases}
            e_S &\text{ if } s' = e_S \\
            (m'\bullet m, \map(\iota)(\zip(\ell,\ell'))) &\text{ if } s' = (m', \ell').
        \end{cases}
        \right).
    \end{align*}
    We can easily conclude $\iota$ is a monomorphism by observing that $\iota(s)(s_n) = s$ for all $s \in S_{M,n}$.
    Given any two morphisms $f,g\colon A \to S_{M,n}$, we can draw the following diagram
    \[\begin{tikzcd}
        A & {[S_{M,n}^\circ, S_{M}]} \\
        & S_{M,n}
        \arrow["{!}", from=1-1, to=1-2]
        \arrow["f"', shift right, from=1-1, to=2-2]
        \arrow["g", shift left, from=1-1, to=2-2]
        \arrow["\iota"', tail, from=2-2, to=1-2].
    \end{tikzcd}\]
    Since the morphism $A \to [S_{M,n}^\circ, S_{M}]$ is unique, we know the composites $\iota \circ f = \iota \circ g$, and by $\iota$ being mono we conclude $f = g$.
    Hence, there can be at most one algebra morphism from a $S_{M,n}^\circ$-initial algebra $A$ to $S_{M,n}$.
\end{proof}

Putting all the above together, we arrive at the following result.

\begin{theorem}
    The algebra $S_{M,n}$ is the terminal $S_{M,n}^\circ$-initial algebra.
\end{theorem}

\begin{proof}
    Given an $S_{M,n}^\circ$-initial algebra $A$, by  \cref{prop:AtoS_nexists}, we obtain an algebra morphism $\phi_n \colon A \to S_{M,n}$.
    By \cref{lem:AtoS_nunique}, it is unique.
    We conclude $\n$ is the terminal $S_{M,n}^\circ$-initial algebra.
\end{proof}

%%%%%%%%% Indices

% Glossary of notations

\printglossary[title=Index of symbols]

% \printnomenclature

% Index cross-references
% \newcommand{\indexsee}[2]{\index{#1|see{#2}}}

% \indexsee{}{}
%\indexsee{type!exo-}{exotype}

\vspace{4em}
\raggedcolumns
\printindex

\bibliographystyle{spmpsci}
\bibliography{bib}

@inCollection{Goguen-Thatcher-Wagner78,
author = {Joseph A. Goguen AND James W. Thatcher AND Eric G.
Wagner},
title = {An initial algebra approach to the specification,
correctness,
and implementation of abstract data types},
booktitle = {Current trends in Programming Methodology, vol.
{IV}},
year = 1978,
pages = {80--149},
publisher = {Prentice-Hall}
}

@article {lawvere2004functorial,
    AUTHOR = {Lawvere, F. William},
     TITLE = {Functorial semantics of algebraic theories and some algebraic
              problems in the context of functorial semantics of algebraic
              theories},
      NOTE = {Reprinted from Proc. Nat. Acad. Sci. U.S.A. {\bf 50} (1963),
              869--872 [MR0158921] and {\it Reports of the Midwest Category
              Seminar. II}, 41--61, Springer, Berlin, 1968 [MR0231882]},
   JOURNAL = {Repr. Theory Appl. Categ.},
  FJOURNAL = {Reprints in Theory and Applications of Categories},
      YEAR = {2004},
     PAGES = {1--121},
   MRCLASS = {18C50},
  MRNUMBER = {2118935},
  volume = {5}
}

@misc{anel2013sweedler,
      title={Sweedler Theory for (co)algebras and the bar-cobar constructions}, 
      author={Anel, Mathieu and Joyal, André},
      year={2013},
      eprint={1309.6952},
      archivePrefix={arXiv},
      primaryClass={math.CT},
    url={test}
}

@book {Fox,
    AUTHOR = {Fox, Thomas F.},
     TITLE = {Universal Coalgebras},
      NOTE = {Thesis (Ph.D.)--McGill University (Canada)},
 PUBLISHER = {ProQuest LLC, Ann Arbor, MI},
      YEAR = {1976},
     PAGES = {(no paging)},
   MRCLASS = {99-05},
  MRNUMBER = {2626730},
       URL =
              {http://gateway.proquest.com/openurl?url_ver=Z39.88-2004&rft_val_fmt=info:ofi/fmt:kev:mtx:dissertation&res_dat=xri:pqdiss&rft_dat=xri:pqdiss:NK31770},
}

@misc{grignou,
      title={Mapping Coalgebras {I}: Comonads}, 
      author={Brice Le Grignou},
      year={2022},
      note={arXiv:2009.10041.v2},
}

@InProceedings{MRU22,
author="McDermott, Dylan
and Rivas, Exequiel
and Uustalu, Tarmo",
editor="Bouyer, Patricia
and Schr{\"o}der, Lutz",
title="Sweedler Theory of Monads",
booktitle="Foundations of Software Science and Computation Structures",
year="2022",
publisher="Springer International Publishing",
address="Cham",
pages="428--448",
isbn="978-3-030-99253-8",
       DOI = {10.1007/978-3-030-99253-8\_22},
}

@article{Per22,
title = {The coalgebraic enrichment of algebras in higher categories},
journal = {Journal of Pure and Applied Algebra},
volume = {226},
number = {3},
pages = {106849},
year = {2022},
issn = {0022-4049},
doi = {https://doi.org/10.1016/j.jpaa.2021.106849},
author ="Maximilien P{\'e}roux"
}

@article {vasila,
    AUTHOR = {Vasilakopoulou, Christina},
     TITLE = {Enriched duality in double categories: {$\mathcal{V}$}-categories
              and {$\mathcal{V}$}-cocategories},
   JOURNAL = {Journal of Pure and Applied Algebra},
    VOLUME = {223},
      YEAR = {2019},
    NUMBER = {7},
     PAGES = {2889--2947},
      ISSN = {0022-4049},
   MRCLASS = {18D05 (18C15 18D10 18D15 18D20 18D30)},
MRREVIEWER = {R. H. Street},
       DOI = {10.1016/j.jpaa.2018.10.003},
}

@misc{arxiv,
      title={Coinductive control of inductive data types}, 
      author={Paige Randall North and Maximilien Péroux},
      year={2023},
      eprint={2303.16793},
      archivePrefix={arXiv},
      primaryClass={math.CT}
}

@book{scott1970outline,
  title={Outline of a mathematical theory of computation},
  author={Scott, Dana},
  year={1970},
  publisher={Oxford University Computing Laboratory, Programming Research Group Oxford}
}

@article{benabou1965categories,
  title={Cat{\'e}gories relatives},
  author={B{\'e}nabou, Jean},
  journal={Comptes rendus hebdomadaires des séances de l'Académie des sciences},
  volume={260},
  number={14},
  pages={3824},
  year={1965},
  publisher={GAUTHIER-VILLARS/EDITIONS ELSEVIER 23 RUE LINOIS, 75015 PARIS, FRANCE}
}

@article{leinster,
title = {A general theory of self-similarity},
journal = {Advances in Mathematics},
volume = {226},
number = {4},
pages = {2935-3017},
year = {2011},
issn = {0001-8708},
doi = {https://doi.org/10.1016/j.aim.2010.10.009},
url = {https://www.sciencedirect.com/science/article/pii/S0001870810003713},
author = {Tom Leinster}
}

@article{yetter,
title = {Abelian categories of modules over a (lax) monoidal functor},
journal = {Advances in Mathematics},
volume = {174},
number = {2},
pages = {266-309},
year = {2003},
issn = {0001-8708},
doi = {https://doi.org/10.1016/S0001-8708(02)00041-5},
author = {D.N. Yetter}
}

@book {Kelly,
    AUTHOR = {Kelly, Gregory Maxwell},
     TITLE = {Basic concepts of enriched category theory},
    SERIES = {London Mathematical Society Lecture Note Series},
    VOLUME = {64},
 PUBLISHER = {Cambridge University Press, Cambridge-New York},
      YEAR = {1982},
     PAGES = {245},
      ISBN = {0-521-28702-2},
   MRCLASS = {18-02 (18D20)},
  MRNUMBER = {651714},
MRREVIEWER = {F.\ E. J. Linton},
}

@incollection{ml-wtypes,
  title={Constructive mathematics and computer programming},
  author={Martin-L{\"o}f, Per},
  booktitle={Studies in Logic and the Foundations of Mathematics},
  volume={104},
  pages={153--175},
  year={1982},
  publisher={Elsevier}
}

@article {blockleroux,
    AUTHOR = {Block, Richard E. and Leroux, Pierre},
     TITLE = {Generalized dual coalgebras of algebras, with applications to
              cofree coalgebras},
   JOURNAL = {Journal of Pure and Applied Algebra},
    VOLUME = {36},
      YEAR = {1985},
    NUMBER = {1},
     PAGES = {15--21},
      ISSN = {0022-4049,1873-1376},
   MRCLASS = {16A24},
       DOI = {10.1016/0022-4049(85)90060-X},
       URL = {https://doi.org/10.1016/0022-4049(85)90060-X},
}

@incollection {north2023coinductive,
    AUTHOR = {North, Paige Randall and P\'{e}roux, Maximilien},
     TITLE = {Coinductive control of inductive data types},
 BOOKTITLE = {10th {C}onference on {A}lgebra and {C}oalgebra in {C}omputer
              {S}cience},
    SERIES = {LIPIcs},
    VOLUME = {270},
     PAGES = {Art. No. 15, 17},
 PUBLISHER = {Schloss Dagstuhl. Leibniz-Zent. Inform., Wadern},
      YEAR = {2023},
      ISBN = {978-3-95977-287-7},
   MRCLASS = {68Q65},
       DOI = {10.4230/lipics.calco.2023.15},
       URL = {https://doi.org/10.4230/lipics.calco.2023.15},
}

@article {hylandetal,
    AUTHOR = {Hyland, Martin and L\'{o}pez Franco, Ignacio and
              Vasilakopoulou, Christina},
     TITLE = {Hopf measuring comonoids and enrichment},
  JOURNAL = {Proceedings of the London Mathematical Society. Third Series},
    VOLUME = {115},
      YEAR = {2017},
    NUMBER = {5},
     PAGES = {1118--1148},
      ISSN = {0024-6115,1460-244X},
   MRCLASS = {16T15 (16T05 18D10 18D20)},
MRREVIEWER = {Ram\'{o}n\ Gonz\'{a}lez Rodr\'{\i}guez},
       DOI = {10.1112/plms.12064},
       URL = {https://doi.org/10.1112/plms.12064},
}

@book {Sweedler,
    AUTHOR = {Sweedler, Moss E.},
     TITLE = {Hopf algebras},
    SERIES = {Mathematics Lecture Note Series},
 PUBLISHER = {W. A. Benjamin, Inc., New York},
      YEAR = {1969},
     PAGES = {vii+336},
   MRCLASS = {18.20 (16.00)},
MRREVIEWER = {K.-T.\ Chen},
}

@article {hazewinkel,
    AUTHOR = {Hazewinkel, Michiel},
     TITLE = {Cofree coalgebras and multivariable recursiveness},
   JOURNAL = {Journal of Pure and Applied Algebra},
    VOLUME = {183},
      YEAR = {2003},
    NUMBER = {1-3},
     PAGES = {61--103},
      ISSN = {0022-4049,1873-1376},
   MRCLASS = {16W30},
MRREVIEWER = {E.\ J.\ Taft},
       DOI = {10.1016/S0022-4049(03)00013-6},
       URL = {https://doi.org/10.1016/S0022-4049(03)00013-6},
}

@book{peano,
  title={Arithmetices principia: Nova methodo exposita},
  author={Peano, Giuseppe},
  year={1889},
  publisher={Fratres Bocca}
}

@article{dybjer,
title = {Representing inductively defined sets by wellorderings in {M}artin-{L}öf's type theory},
journal = {Theoretical Computer Science},
volume = {176},
number = {1},
pages = {329-335},
year = {1997},
issn = {0304-3975},
doi = {https://doi.org/10.1016/S0304-3975(96)00145-4},
url = {https://www.sciencedirect.com/science/article/pii/S0304397596001454},
author = {Peter Dybjer}
}

@article{mp,
title = {Wellfounded trees in categories},
journal = {Annals of Pure and Applied Logic},
volume = {104},
number = {1},
pages = {189-218},
year = {2000},
issn = {0168-0072},
doi = {https://doi.org/10.1016/S0168-0072(00)00012-9},
url = {https://www.sciencedirect.com/science/article/pii/S0168007200000129},
author = {Ieke Moerdijk and Erik Palmgren}
}

@mastersthesis{rech2017strictly,
  title={Strictly positive types in homotopy type theory},
  author={Rech, Felix},
  year={2017},
  school={Saarland University}
}

@article{lambek,
author = {Lambek, Joachim},
journal = {Mathematische Zeitschrift},
pages = {151-161},
title = {A Fixpoint Theorem for complete Categories.},
url = {http://eudml.org/doc/170906},
volume = {103},
year = {1968},
}

@article{sp,
author = {Smyth, M. B. and Plotkin, G. D.},
title = {The Category-Theoretic Solution of Recursive Domain Equations},
journal = {SIAM Journal on Computing},
volume = {11},
number = {4},
pages = {761-783},
year = {1982},
doi = {10.1137/0211062},

URL = { 
    
        https://doi.org/10.1137/0211062
    
    

}
}

@article{adamek74,
author = {Ad{\'{a}}mek, Ji{\v{r}}{\'{\i}}},
journal = {Commentationes Mathematicae Universitatis Carolinae},
number = {4},
pages = {589-602},
publisher = {Charles University in Prague, Faculty of Mathematics and Physics},
title = {Free algebras and automata realizations in the language of categories},
url = {http://eudml.org/doc/16649},
volume = {015},
year = {1974},
}

@book {presentable,
    AUTHOR = {Ad\'{a}mek, Ji\v{r}\'{\i} and Rosick\'{y}, Ji\v{r}\'{\i}},
     TITLE = {Locally presentable and accessible categories},
    SERIES = {London Mathematical Society Lecture Note Series},
    VOLUME = {189},
 PUBLISHER = {Cambridge University Press, Cambridge},
      YEAR = {1994},
     PAGES = {xiv+316},
      ISBN = {0-521-42261-2},
   MRCLASS = {18Axx (18-02)},
MRREVIEWER = {J. R. Isbell},
       DOI = {10.1017/CBO9780511600579},
}

@article{varieties, title={On varieties and covarieties in a category}, volume={13}, DOI={10.1017/S0960129502003882}, number={2}, journal={Mathematical Structures in Computer Science}, publisher={Cambridge University Press}, author={Ad\'{a}mek, Ji\v{r}\'{\i} and Porst, Hans-E.}, year={2003}, pages={201–232}}

\end{document}